\documentclass[
a4paper,
12pt,        
headsepline,
numbers=endperiod
]{scrartcl}
\usepackage{geometry}
\usepackage[new]{old-arrows}
\usepackage{microtype}
\usepackage{xfrac}    
\usepackage{caption}
\usepackage{mathdots}
\usepackage{faktor}
\usepackage[dvipsnames]{xcolor}
\usepackage{braket}
\usepackage{amsfonts}
\usepackage{amsmath}
\usepackage{accents}
\usepackage{mathtools}
\usepackage{amsthm}
\usepackage[capitalise]{cleveref}
\usepackage{pdfpages}
\usepackage{amssymb} 
\usepackage{faktor}
\usepackage[fleqn]{nccmath}
\usepackage{tikz}
\usepackage{tikz-cd}
\usetikzlibrary{arrows,calc}

\usepackage[arrow, matrix, curve]{xy}
\usepackage[titletoc]{appendix}
\setkomafont{sectioning}{\normalfont\normalcolor\bfseries}

\usepackage{extarrows}
\usepackage{paralist}
\usepackage{relsize}

\usepackage{enumerate}
\usepackage{setspace}
\usepackage{microtype}
\usepackage{enumitem,blindtext}
\usepackage{tikz}

\usetikzlibrary{matrix,arrows}
\usepackage[english]{babel}
\usepackage[T1]{fontenc}
\usepackage[utf8]{inputenc}
\usepackage{newunicodechar}
\newunicodechar{ﬁ}{fi}
\newunicodechar{ﬀ}{ff}
\usepackage{cleveref}

\newtheorem*{Thm*}{Theorem}
\newtheorem*{Prop*}{Proposition}
\newtheorem*{Lem*}{Lemma}
\newtheorem*{Cor*}{Corollary}
\newtheorem*{Conj*}{Conjecture}

\newtheorem{Thm}{Theorem}[section]
\newtheorem{Prop}[Thm]{Proposition}
\newtheorem{Lem}[Thm]{Lemma}

\theoremstyle{definition}
\newtheorem{Defi}[Thm]{Definition}

\newtheorem{Defiprop}[Thm]{Definition, Proposition}

\newtheorem{Exa}[Thm]{Example}
\newtheorem{Rem}[Thm]{Remark}
\newtheorem*{Defi*}{Definition}

\makeatother

\newcommand\abs[1]{\left\lvert#1\right\lvert}



\spacing{1.2}
\begin{document}

\title{\LARGE{Multidimensional Persistence: \\
Invariants and Parameterization}}

\author{Maximilian Neumann\footnote{E-mail address: neumann.mn@icloud.com}}
\date{}
\maketitle

\noindent \textbf{Abstract.} This article grew out of the theoretical part of my Master’s thesis at the Faculty of Mathematics and Information Science at Ruprecht-Karls-Universität Heidelberg under the supervision of PD Dr.~Andreas Ott. Following the work of G.~Carlsson and A.~Zomorodian on the theory of multidimensional persistence in 2007 and 2009, the main goal of this article is to give a complete classification and parameterization for the algebraic objects corresponding to the homology of a multifiltered simplicial complex. As in the work of G.~Carlsson and A.~Zomorodian, this classification and parameterization result is then used to show that it is only possible to obtain a discrete and complete invariant for these algebraic objects in the case of one-dimensional persistence, and that it is impossible to obtain the same in dimensions greater than one.

\tableofcontents

\section{Introduction}

Persistent homology allows us to analyze topological features of increasing filtrations of topological spaces with a complete and discrete invariant, the so-called \textit{(persistence) barcode}. Many applications of topology lead to multifiltrations, e.g. if we analyze time dependent point cloud datasets as in \cite{AndreasOtt}.

In this article, we investigate the theoretical aspects of multidimensional persistence, the extension of the one dimensional concept of persistent homology. This article is mainly based on the work of G.~Carlsson and A.~Zomorodian in \cite{article} and \cite{Carlsson2009}. \cite{article} and \cite{Carlsson2009} are very similar where the difference is that \cite{Carlsson2009} digs deeper into the details. Both \cite{article} and \cite{Carlsson2009} contain unproven statements. So, one of our main tasks is to fill in the gaps in \cite{article} and \cite{Carlsson2009}. Following \cite{article} and \cite{Carlsson2009}, our main goal is to carry out a complete classification (\cref{Chapter 1}) and parameterization (\cref{Chapter 2}) of the algebraic objects that correspond to the homology of a multifiltered simplicial complex. As in \cite{article} and \cite{Carlsson2009}, we then use the parameterization result to show that it is only possible to obtain a discrete and complete invariant for these algebraic objects  in the case of one-dimensional persistence, and that it is impossible to obtain the same in dimension greater than one. Here, a \textit{discrete invariant} is an invariant which is countable in size and which takes values in the same set independently of the underlying coefficient field (see \cref{discrete invariant}) and a \textit{complete} invariant is an invariant which is injective as a function on the objects of classification (see \cref{complete invariant}). At the end of \cref{Chapter 1}, we discuss a discrete invariant for multidimensional persistence, the so-called \textit{rank invariant} (introduced in \cite{article} and \cite{Carlsson2009}), which is practicable for analyzing multiparamter problems and equivalent to the barcode in dimension one.

Our approach in \cref{Chapter 1}, based on  \cite[Sec.~4]{article} and \cite[Sec.~4]{Carlsson2009}, can be summarized as follows: first of all, we have to determine the algebraic objects which correspond to the homology of a finite $n$-filtered simplicial complex (see \cref{section 1.2}).
As it turns out, any finite $n$-filtered simplicial complex defines a finitely generated $n$-graded $A_n^{\mathbb{F}}$-module $M$ where $A_n^{\mathbb{F}}=\mathbb{F}[x_1, \dots,x_n]$ (see Sections \ref{section 1.3} and \ref{section 1.4}). Moreover, if $\mathbb{F}=\mathbb{F}_p$ where $p \in \mathbb{N}$ is a prime number or $\mathbb{F}= \mathbb{Q}$, any such module $M$ can be realized as a finite $n$-filtered simplicial complex (see \cref{geom rel modules}). 

For the following, denote by $\mathbf{Grf}_n(A_n^{\mathbb{F}})$ the category of finitely generated $n$-graded $A_n^{\mathbb{F}}$-modules together with graded $A_n^{\mathbb{F}}$-module homomorphisms as morphisms. In \cref{Free hulls and free section}, we show that any $M \in \mathbf{Grf}_n(A_n^{\mathbb{F}})$ admits a \textit{free hull}, i.e. there exists a pair $(F,p)$ where ${F \in \mathbf{Grf}_n(A_n^{\mathbb{F}})}$ is graded free and $p: F \to M$ is a graded surjective $A_n^{\mathbb{F}}$-module homomorphism such that
\begin{equation}
\mathrm{id}_{\mathbb{F}} \otimes_{A_n^{\mathbb{F}}} p:\mathbb{F} \otimes_{A_n^{\mathbb{F}}} F \xlongrightarrow{\sim} \mathbb{F} \otimes_{A_n^{\mathbb{F}}} M \notag
\end{equation}
is a graded isomorphism of $n$-graded $\mathbb{F}$-vector spaces. Here $A_n^{\mathbb{F}}$ acts on $\mathbb{F}$ by setting the variable action identical to zero. 
By using an $n$-graded version of \textit{Nakayama's Lemma} (see \cref{graded Naka}), it is possible to show that $(F,p)$ is a free hull of $M$ if and only if $p$ maps a homogeneous $A_n^{\mathbb{F}}$-basis of $F$ to a minimal system of homogeneous generators of $M$ (see \cref{equiv char free hulls}). 

Every graded free $F \in \mathbf{Grf}_n(A_n^{\mathbb{F}})$ admits a graded isomorphism
\begin{equation}
F \cong \mathcal{F}_n^{\mathbb{F}}(\xi):= \bigoplus_{v \in V} A_n^{\mathbb{F}}(v)^{\mu(v)} \notag
\end{equation}
where 
\begin{equation}
\xi=(V, \mu):= \bigcup_{v \in V} \{ (v,1), \dots, (v, \mu(v))\} \notag
\end{equation}
is an $n$-\textit{dimensional multiset} with $V \subseteq \mathbb{N}^n$ and \textit{mulitplicity function} $\mu: V \to \mathbb{N}_{\geq 1}$ (see \cref{multiset defi}) and $A_n^{\mathbb{F}}(v)$ is a polynomial ring shifted upwards in grading by $v$ (see \cref{shift defi}).

The main theorem on free hulls (see \cref{Free hull thm}) states that free hulls are unique up to isomorphism and that every $M \in \mathbf{Grf}_n(A_n^{\mathbb{F}})$ admits a free hull
\begin{equation}
p:\mathcal{F}^{\mathbb{F}}_n(\xi) \longrightarrow M \notag
\end{equation}
for a finite $n$-dimensional multiset $\xi$. If $(\mathcal{F}_n^{\mathbb{F}}(\xi), p)$ is a free hull of $M$, then $\xi$ is also called the \textit{type} of $M$. If $G \subseteq M$ is a minimal set of homogeneous generators of $M$, the multiset $\xi$ tells us in which degrees the generators in $G$ are located. Free hulls can be viewed as the basic building block of a minimal graded free resolution of $M$ (see \cref{min grad res}).

Let us denote by \begin{equation}
\mathbf{Grf}_n(A_n^{\mathbb{F}})/_{\cong} \notag
\end{equation}
the set of all isomorphism classes $[M]$ of objects $M \in  \mathbf{Grf}_n(A_n^{\mathbb{F}})$. Indeed, $\mathbf{Grf}_n(A_n^{\mathbb{F}})/_{\cong}$ is a set (see \cref{is a set}). Denote by \begin{equation}
I_n^{\mathbb{F}}(\xi_0, \xi_1) \subseteq \mathbf{Grf}_n(A_n^{\mathbb{F}})/_{\cong} \notag
\end{equation}
the subset of all $[M]$ such that $M$ admits a free hull $p_0: \mathcal{F}_n^{\mathbb{F}}(\xi_0) \to M$ and $\mathrm{ker}(p_0)$ admits a free hull  $p_1: \mathcal{F}_n^{\mathbb{F}}(\xi_1) \to \mathrm{ker}(p_0)$.

Let $L \subseteq \mathcal{F}_n^{\mathbb{F}}(\xi_0)$ be a graded $A_n^{\mathbb{F}}$-submodule. We say that $L$ satisfies the \textit{tensor-condition} (see \cref{Tensor condition section}) if the image of the induced morphism \begin{equation}
 \mathbb{F} \otimes_{A_n^{\mathbb{F}}} L \longrightarrow \mathbb{F} \otimes_{A_n^{\mathbb{F}}} \mathcal{F}_n^{\mathbb{F}}(\xi_0)  \notag
\end{equation}
is zero. Denote by $S_n^{\mathbb{F}}(\xi_0,\xi_1)$ the set of all graded $A_n^{\mathbb{F}}$-submodules $L\subseteq \mathcal{F}_n^{\mathbb{F}}(\xi_0)$
of type $\xi_1$ which satisfy the tensor-condition.

As in \cite[Sec.~4.5]{Carlsson2009}, we obtain a well-defined set theoretic surjection 
\begin{equation}\label{Map intro}
S_n^{\mathbb{F}}(\xi_0, \xi_1) \longrightarrow I_n^{\mathbb{F}}(\xi_0, \xi_1), \quad L \longmapsto [\mathcal{F}_n^{\mathbb{F}}(\xi_0)/L] \notag.
\end{equation}
Without the tensor-condition this map would not be well-defined. In \cite{article} and \cite{Carlsson2009}, the tensor-condition is missing. This surjection induces a set theoretic bijection
\begin{align}
S_n^{\mathbb{F}}(\xi_0,\xi_1)/{\mathrm{Aut}(\mathcal{F}_n^{\mathbb{F}}(\xi_0))}\xlongrightarrow{\sim} {I}_{n}^{\mathbb{F}}(\xi_0,\xi_1) \notag
\end{align}
which provides a complete classification of ${I}_{n}^{\mathbb{F}}(\xi_0,\xi_1)$ (see \cref{classification bijection} and \cite[Thm.~9]{Carlsson2009}). Now the question is if there exists a discrete parameterization of \begin{equation}
S_n^{\mathbb{F}}(\xi_0,\xi_1)/{\mathrm{Aut}(\mathcal{F}_n^{\mathbb{F}}(\xi_0))} \notag.
\end{equation}

In \cref{Rank invariant section}, we approach this question by considering a possible candidate for a discrete and complete invariant for $\mathbf{Grf}_n(A_n^{\mathbb{F}})/_{\cong}$, the so-called \textit{rank invariant} (introduced in \cite[Sec.~6]{article} and \cite[Sec.~6]{Carlsson2009}). In dimension $n=1$, the rank invariant is equivalent to the barcode (see \cref{Barcode rank equiv}) and defines hence a complete and discrete invariant (see \cite[Thm.~5]{article} and \cite[Thm.~12]{Carlsson2009}). So, the rank invariant can be viewed as a discrete generalization of the barcode to dimension $n \geq 2$. But as we will see in \cref{Chapter 2}, the rank invariant is not complete in dimension $n \geq 2$.

In \cref{Chapter 2}, our first goal is to show that for $n\geq 2$, there exists no discrete and complete invariant for $\mathbf{Grf}_n(A_n^{\mathbb{F}})/_{\cong}$. For this, we parameterize $S_{n}^{\mathbb{F}}(\xi_0, \xi_1)$ as a subset of a product of Grassmannians together with a group action of ${\mathrm{Aut}(\mathcal{F}_n^{\mathbb{F}}(\xi_0))} $. This approach follows \cite[Sec.~5]{article} and \cite[Sec.~5]{Carlsson2009} and is carried out in Sections \ref{Relation Families section}-\ref{Framed relation families section}. Based on \cite[Sec.~5.2]{article} and \cite[Sec.~5.2]{Carlsson2009}, we then give a counterexample in \cref{cont inv} showing the non-existence of a discrete and complete invariant for $\mathbf{Grf}_n(A_n^{\mathbb{F}})/_{\cong}$ if $n\geq 2$: for $n=2$, consider \begin{align}
\xi_0 &=\{((0,0),1),((0,0),2) \} \notag \\
 \xi_1 &=\{((3,0),1),((2,1),1),((1,2),1),((0,3),1) \} \notag.
\end{align}
By using the parameterization result (see Theorems \ref{Paramet thm1} and \ref{all on sets prop}), we see that $S_{2}^{\mathbb{F}}(\xi_0, \xi_1)$ is in bijection to a fourfold product of the Grassmannian $\mathrm{G}_{\mathbb{F}}(1,2)$. Note that $\mathrm{G}_{\mathbb{F}}(1,2)$ equals the projective line $\mathbb{P}_1(\mathbb{F})$. The automorphism group $\mathrm{Aut}(\mathcal{F}_2^{\mathbb{F}}(\xi_0))$ is isomorphic to $\mathrm{GL}_2(\mathbb{F})$ (see \cref{second Autgroup isom}). Thus, we obtain set theoretic bijections
\begin{align}
I_2^{\mathbb{F}}(\xi_0, \xi_1) \cong S_2^{\mathbb{F}}(\xi_0, \xi_1) /\mathrm{Aut}(\mathcal{F}_2^{\mathbb{F}}(\xi_0)) \cong \mathbb{P}_1(\mathbb{F})^4/\mathrm{GL}_2({\mathbb{F}}) \notag
\end{align}
and $I_2^{\mathbb{F}}(\xi_0, \xi_1)  \cong \mathbb{P}_1(\mathbb{F})^4/\mathrm{GL}_2({\mathbb{F}})$ is uncountable if $\mathbb{F}$ is uncountable. For $n > 2$, we just have to append zeros to the entries of $\xi_0$ and $\xi_1$.

Based on and inspired by \cite[Sec.~5]{article} and \cite[Sec.~5]{Carlsson2009}, \cref{The Moduli space} is devoted to the investigation of goal number two in \cref{Chapter 2}: showing that $S_{n}^{\mathbb{F}}(\xi_0, \xi_1)$, which we have paramterized as a subset of a product of Grassmannians, corresponds to the set of $\mathbb{F}$-points of an algebraic variety together with an algebraic group action of ${\mathrm{Aut}(\mathcal{F}_n^{\mathbb{F}}(\xi_0))} $. In other words, we want to give an explicit parameterization of the moduli space of our classification problem.

\vspace{1cm}
\noindent \textbf{Acknowledgements.} I want to thank my supervisor Andreas Ott for his motivating support and the inspiring discussions about TDA and its current applications to the viral evolution of SARS-CoV-2.

\newpage 

\section{Preliminaries}\label{Chapter 0}
This section introduces the basic definitions and notations required for this article. 

\subsection{Multisets}\label{section 1.1}

This section follows \cite{article} and \cite{Carlsson2009}. We just add some details and definitions. The next definition introduces a partial order on ${\mathbb{N}^n}$, where $\mathbb{N}=\{0,1,2,3,\dots\}$ denotes the set of natural numbers.

\begin{Defi}[Partial order]\label{5} \,
\begin{enumerate}
\item We define a partial order $\preceq$ on ${\mathbb{N}^n}$ as follows: let \begin{equation}
s=(s_1,...,s_n), t=(t_1,...,t_n) \in {\mathbb{N}^n} \notag.
\end{equation}
We write $s\preceq t$ if and only if $s_i\leq t_i$ for all $i\in \{1,...,n\}$. We write $s\prec t$ if and only if $s\preceq t$ and $s_i  < t_i$ for some $i\in \{1,\dots,n\}$. If $s \preceq t$ we also write $t \succeq s$ and if $s \prec t$ we also write $t \succ s $.
\item If $n=1$, then $\preceq$ equals the total order $\leq$ on ${\mathbb{N}}$.
\item For $S\subseteq {\mathbb{N}^n}$ and  $c\in \mathbb{N}^n$, define
$S_{\succeq c}:= \{ z \in S \mid z \succeq c \}$ and ${S_{\succ c}:= \{ z \in S \mid z \succ c \}}$.
\item Let $v,w\in {\mathbb{N}^n}$. If $ \neg(v \succeq w)$ we also write $v \nsucceq w$. This means that $v \in {\mathbb{N}^n} \setminus {\mathbb{N}^n_{\succeq w}}$. Note that in general, $v \nsucceq w$ does not imply $v \prec w$ (except only if $n=1$). If $ \neg(v \succ w)$ we also write $v \nsucc w$. This means that $v \in {\mathbb{N}^n} \setminus {\mathbb{N}^n_{\succ w}}$. Note that in general, $v \nsucc w$ does not imply that $v \preceq w$ (except only if $n=1$). For $\prec, \preceq$, the symbols $\nprec,\npreceq$ are defined analogously.
\end{enumerate}
\end{Defi}

Multisets are one of the basic building blocks for the classification of finitely generated $n$-graded $A_n^{\mathbb{F}}$-modules. A multiset may be viewed as a set where elements are allowed to appear more than once.

\begin{Defi}\label{3}\label{multiset defi} Recall that $\mathbb{N}=\{0,1,2,3, \dots\}$. Let $n \in \mathbb{N}_{\geq 1}$. Let $S\subseteq {\mathbb{N}^n}$ and $\mu:S  \to \mathbb{N}_{\geq 1}$ be a map. Then \begin{equation} (S, \mu):= \bigcup_{s \in S} \{(s,1), \dots, (s, \mu(s)) \}\subseteq S \times \mathbb{N}_{\geq 1} \notag
\end{equation}
is called an $n$-\textit{dimensional multiset}.  The map $\mu: S \to \mathbb{N}_{\geq 1}$ is also called \textit{the multiplicity function of} $\xi$ and for $s \in S$, $\mu(s)$ is called the \textit{multiplicity of} $s$.
For ${(s,l), (t,k) \in (S, \mu)}$, define $(s,l)\preceq (t,k)$ if and only if $s\preceq t$ and $(s,l)\prec (t,k)$ if and only if $s\prec t$. Again, if $(s,l) \preceq (t,k)$ we also write $(t,k) \succeq (s,l)$ and if $(s,l) \prec (t,k)$ we also write $(t,k) \succ (s,l) $.
\end{Defi}

\begin{Exa}  We present two examples of multisets:
\begin{enumerate}
\item Consider the one-dimensional multiset
\begin{equation}
\xi=\{(1,1),(1,2),(2,1) \} \notag.
\end{equation}
We have $\xi=(S, \mu)$ where
$S=\{1,2\} \subseteq \mathbb{N}$, $\mu(1)=2$ and $\mu(2)=1$. 
 
\item Consider the two-dimensional multiset\begin{equation}
\xi=\{((1,1),1),((1,1),2),((2,2),1) \} \notag.
\end{equation}
We have $\xi=(S, \mu)$ where
$S=\{(1,1), (2,2)\} \subseteq \mathbb{N}^2$, ${\mu(1,1)=2}$ and $\mu(2,2)=1$. 
\end{enumerate}
\end{Exa}

\begin{Defi} Let $\xi_0, \xi_1$ be two $n$-dimensional multisets. We write $\xi_1 \succ_D \xi_0$ if $v \succ w$ for all $v \in \xi_1$ and all $w \in \xi_0$. In this case, we say that $\xi_1$ \textit{dominates} $\xi_0$.
\end{Defi}

\begin{Exa} Consider \begin{align}
\xi_0&=\{((0,0),1), ((0,0),2) \} \notag\\
\xi_1&=\{((1,1),1),((1,1),2),((2,2),1) \} \notag.
\end{align}
Then $\xi_1 \succ_D \xi_0$.
\end{Exa}

\subsection{Multifiltered simplicial complexes}\label{section 1.2}

For the next two definitions, we follow \cite[Def.~2.1]{harrington2019stratifying} (which is compatible with \cite{article} and \cite{Carlsson2009}). 
\begin{Defi} Let $(X_u)_{u \in {\mathbb{N}^n}}$ be a familiy of simplicial complexes (for a definition of simplicial complex see \cite[§1-3]{munkres2018elements}).
\begin{enumerate}
\item $(X_u)_{u \in {\mathbb{N}^n}}$ is called an $n$-\textit{filtration} (\textit{one-filtration} if $n=1$ and \textit{bifiltration} if $n=2$) if $X_u \subseteq X_v$ for all $u,v \in {\mathbb{N}^n}$ with $u \preceq v$. 
\item If $(X_u)_{u \in {\mathbb{N}^n}}$ is an $n$-filtration, then we say that $(X_u)_{u \in {\mathbb{N}^n}}$ \textit{stabilizes} if there exists a $w \in {\mathbb{N}^n}$ such that $X_u=X_w$ for all $ u \in {\mathbb{N}^n_{\succeq w}}$.
\end{enumerate}
\end{Defi}

\begin{Defi}
Let $X$ be a simplicial complex.
\begin{enumerate}
\item $X$ is called $n$-\textit{filtered} if $X=\bigcup_{u \in {\mathbb{N}^n}} X_u$ for an $n$-filtration of simplicial complexes $(X_u)_{u \in {\mathbb{N}^n}}$ which stabilizes. $X$ is called \textit{multifiltered} if $X$ is $n$-filtered for some $n \in \mathbb{N}$. We often write $X=(X_u)_{u \in {\mathbb{N}^n}}$ to indicate that $X$ is an $n$-filtered simplicial complex with $X= \bigcup_{u\in {\mathbb{N}^n}} X_u$. 

\item $X$ is called a \textit{finite} $n$-\textit{filtered simplicial complex} if $X$ is an $n$-filtered simplicial complex and if $X$ is finite (i.e. the set of vertices of $X$ is finite).
\end{enumerate}
\end{Defi}

\subsection{\textit{n}-graded rings and modules}\label{section 1.3}
This section introduces basic facts and definitions about $n$-graded commutative rings and modules. Our definitions are compatible with \cite{article}, \cite{Carlsson2009} and \cite{harrington2019stratifying}.

\begin{Defi} Let $R$ be a commutative ring and let $n \in \mathbb{N}_{\geq 1}$.
\begin{enumerate}
\item $R$ is called $n$-\textit{graded}, if $R= \bigoplus_{v \in {\mathbb{N}^n}} R_v$ as abelian groups where we have for $v \in {\mathbb{N}^n}$, ${(R_v,+) \subseteq (R,+)}$ is an abelian subgroup such that $R_{u} R_{v} \subseteq R_{u+v}$ for all $u \in {\mathbb{N}^n}$. 
\item A \textit{morphism} of $n$-graded commutative rings $\varphi:R \to  S$ (or \textit{graded} ring homomorphism  or \textit{graded} morphism) is a ring homomorphism such that $\varphi(R_v) \subseteq S_v$ for all $v \in {\mathbb{N}^n}$. 
\item An element $y \in R$ is called \textit{homogeneous} if $y \in R_v$ for some $v \in {\mathbb{N}^n}$. If $y \in R \setminus \{0\}$ is \textit{homogeneous}, then the unique $v \in {\mathbb{N}^n}$ such that $y \in R_v$ is called the \textit{degree of} $y$ and we write $\mathrm{deg}(y)=v$. In this case, $y$ is also called \textit{homogeneous of degree} $v$. Note that $0 \in R$ is homogeneous but its degree is not well-defined since $0 \in R_v$ for all $v \in {\mathbb{N}^n}$.
 \end{enumerate}
\end{Defi}

\begin{Defi}\label{Defi homog} Let $n \in \mathbb{N}_{\geq 1}$. Let $R$ be an $n$-graded commutative ring and let $M$ be an $R$-module. 
\begin{enumerate}
\item $M$ is called $n$-\textit{graded}, if $M= \bigoplus_{v \in {\mathbb{N}^n}} M_v$ as abelian groups where for $v \in {\mathbb{N}^n}$, ${(M_v,+) \subseteq (M,+)}$ is an abelian subgroup such that $R_{u} M_{v} \subseteq M_{u+v}$ for  all $u,v \in {\mathbb{N}^n}$. 
\item A \textit{morphism} of $n$-graded $R$-modules (or \textit{graded} $R$-module homomorphism  or \textit{graded} morphism) $\varphi:M \to  N$  is an  $R$-module  homomorphism such that $\varphi(M_v) \subseteq N_v$ for all $v \in {\mathbb{N}^n}$. 
\item Let $M$ be an $n$-graded $R$-module. An $R$-submodule $N\subseteq M$ is called \textit{graded}, if $N=\bigoplus_{v \in {\mathbb{N}^n}} N \cap M_v$ as abelian groups, i.e. $N$ is an $n$-graded $R$-module again with $n$-grading inherited from $M$. Note that in this case $M/N$ is an $n$-graded $R$-module again with $M/N=\bigoplus_{v \in {\mathbb{N}^n}} M_v /(N \cap M_v)$. If $I \subseteq R$ is a graded submodule (or \textit{homogeneous} ideal), then $R/I$ is an $n$-graded commutative ring with $R/I=\bigoplus_{v \in {\mathbb{N}^n}} R_v /(I \cap R_v)$ as abelian groups.
\item An $n$-graded $R$-module $M$ is called \textit{finitely generated} if $M$ is finitely generated as an $R$-module and \textit{free} if $M$ is free as an $R$-module.
\item An element $y \in M$ is called \textit{homogeneous} if $y \in M_v$ for some $v \in {\mathbb{N}^n}$. If $y \in M \setminus \{0\}$ is \textit{homogeneous}, then the unique $v \in {\mathbb{N}^n}$ such that $y \in M_v$ is called the \textit{degree of} $y$ and we write $\mathrm{deg}(y)=v$. In this case, $y$ is also called \textit{homogeneous of degree} $v$. Note that $0 \in M$ is homogeneous but its degree is not well-defined since $0 \in M_v$ for all $v \in {\mathbb{N}^n}$.
\item $M$ is called \textit{graded free} if $M$ admits a basis of homogeneous  elements.
\end{enumerate}
\end{Defi}

\begin{Defi}[Shift]\label{shift defi} Let $n \in \mathbb{N}_{\geq 1}$. Let $R$ be an $n$-graded commutative ring and let $M$ be an $n$-graded $R$-module. For $v\in {\mathbb{N}^n}$, the \textit{shifted} $n$-graded $R$-module $M(v)$ is defined by setting 
\begin{equation}
M(v)_u:= \begin{cases} M_{u-v}, & u\in \mathbb{N}^n_{\succeq v} \\
0 , & \mathrm{else} 
\end{cases} \notag
\end{equation}
\end{Defi}

\begin{Rem}\label{over comm rings fin} If $S$ is a commutative ring and $N$ a finitely generated $S$-module, then every minimal generating set of $N$ has finite length: let $G$ be a finite generating set of $N$ and $G'$ a minimal generating set. Then any $g \in G$ is a linear combination of finitely many elements $g' \in G'$. Since $G$ is finite and $G'$ is minimal, $G'$ has to be finite. In particular, if $N$ is free, then every basis of $N$ has finite length (every basis is a minimal generating set). 
\end{Rem}

\begin{Rem}\label{Homog generators} Let $n \in \mathbb{N}_{\geq 1}$. Let $R$ be an $n$-graded commutative ring and let $M$ be an $n$-graded $R$-module. Let  $G\subseteq M$ be a generating set. For  every $g  \in G$, we  have \begin{equation}
g=\sum_{v\in {\mathbb{N}^n}} g_v \notag
\end{equation}
for unique $g_v \in M_v$ with $g_v=0$ for all but finitely many $v  \in {\mathbb{N}^n}$. Then  \begin{equation}
G'=\{g_v \mid g\in G, v\in {\mathbb{N}^n} \} \subseteq M \notag
\end{equation}
is a set of homogeneous generators. Assume that $M$ is finitely generated as $R$-module. Then $M$ admits a finite generating set $G \subseteq M$ and hence a finite set of homogeneous generators $G' \subseteq M$. By removing elements of $G'$ we may obtain a minimal set of homogeneous generators $G'' \subseteq G'$ of $M$.
So, we conclude that if $M$ is finitely generated, then $M$ admits a finite minimal set of homogeneous generators and any minimal set of homogeneous generators of $M$ is finite by \cref{over comm rings fin}.
\end{Rem}

\begin{Prop}\label{basics about graded} Let $n \in \mathbb{N}_{\geq 1}$. Let $R$ be an $n$-graded commutative ring. \begin{enumerate}
\item If $M$ is an $n$-graded $R$-module and $N \subseteq M$ an $R$-submodule, then $N$ is graded if and only if $N$ is generated by homogeneous elements.
\item If $(M_i)_{i\in I}$ is a familiy of $n$-graded $R$-modules, then $\bigoplus_{i \in I} M_i$ is an $n$-graded $R$-module via $(\bigoplus_{i \in I} M_i)_u=\bigoplus_{i \in I} (M_i)_u$ for $u \in {\mathbb{N}^n}$.
\item Let $M,N$ be $n$-graded $R$-modules and $f: M \to N$ a graded $R$-module homomorphism. Then $\mathrm{ker}(f) \subseteq M$ and $\mathrm{im}(f) \subseteq N$ are graded submodules.
\item Let $S$ be another $n$-graded  commutative ring and $f:R \to S$ a graded ring homomorphism. Then $\mathrm{ker}(f) \subseteq R$ is a homogeneous ideal. 
\end{enumerate}
\end{Prop}
\begin{proof} $2$ is clear and $4$ follows from $3$ if we endow $S$ with the $n$-graded $R$-module structure via $f$.

$1.$ If $N$ is graded, then $N$ is generated by homogeneous elements (see \cref{Homog generators}). Now assume that $N$ is generated by homogeneous elements $\{a_{v_i} \}_{i \in I}$ with $a_{v_i} \in M_{v_i}$. Let $y \in N$. Then $y = \sum_{i \in I} \lambda_{i} a_{v_i}$ with $\lambda_i \in R$ and $\lambda_i=0$ for all but finitely many $i \in I$. Since $R$ is $n$-graded we have $\lambda_i=\sum_{v \in {\mathbb{N}^n}} \lambda^{(i)}_v$ for unique $\lambda^{(i)}_v \in R_v$ with $\lambda^{(i)}_v=0$ for all but finitely many $v \in {\mathbb{N}^n}$. Thus,
\begin{equation}
y = \sum_{i \in I} \lambda_{i} a_{v_i}= \sum_{i \in I} \sum_{v \in {\mathbb{N}^n}} \lambda^{(i)}_v a_{v_i} \notag
\end{equation}
is a finite sum with $ \lambda^{(i)}_v a_{v_i} \in N \cap M_{v+v_i}$. Therefore, $N \subseteq \bigoplus_{v \in {\mathbb{N}^n}} (N \cap M_v)$. The other inclusion  $\bigoplus_{v \in {\mathbb{N}^n}} (N \cap M_v) \subseteq N$ is clear. Hence, $N=\bigoplus_{v \in {\mathbb{N}^n}} (N \cap M_v)$.

$3.$ Let $y \in \mathrm{ker}(f)$. Then $y= \sum_{v \in {\mathbb{N}^n}} y_v$ for unique $y_v \in M_v$ with $y_v=0$ for all but finitely many $v \in {\mathbb{N}^n}$. Now
\begin{equation}
0=f(y)=f\left(\sum_{v \in {\mathbb{N}^n}} y_v \right)=\sum_{v \in {\mathbb{N}^n}} f(y_v) \notag
\end{equation}
and $f(y_v) \in N_v$ since $f$ is graded. So, $f(y_v)=0$ for all $v \in {\mathbb{N}^n}$ which implies that $y_v \in \mathrm{ker}(f)$ for all $v \in {\mathbb{N}^n}$. Hence, $\mathrm{ker}(f) \subseteq \bigoplus_{v \in {\mathbb{N}^n}} (\mathrm{ker}(f) \cap M_v)$. We clearly have $ \bigoplus_{v \in {\mathbb{N}^n}} (\mathrm{ker}(f) \cap M_v) \subseteq \mathrm{ker}(f)$ which shows that $ \bigoplus_{v \in {\mathbb{N}^n}} (\mathrm{ker}(f) \cap M_v) = \mathrm{ker}(f)$. 

Let $y \in \mathrm{im}(f)$. Then $y=f(z)$ for some $z \in M$. Now $z=\sum_{v \in {\mathbb{N}^n}} z_v$ for unique $z_v \in M_v$ with $z_v=0$ for all but finitely many $v \in {\mathbb{N}^n}$. Since $f$ is graded, we obtain
\begin{equation}
y=f(z)=f \left (\sum_{v \in {\mathbb{N}^n}} z_v \right)=\sum_{v \in {\mathbb{N}^n}} f(z_v) \notag.
\end{equation}
We have $f(z_v) \in \mathrm{im}(f) \cap N_v$ for all $v \in {\mathbb{N}^n}$. Hence, $\mathrm{im}(f) \subseteq \bigoplus_{v \in {\mathbb{N}^n}} (\mathrm{im}(f) \cap N_v)$. We clearly have $ \bigoplus_{v \in {\mathbb{N}^n}} (\mathrm{im}(f) \cap N_v) \subseteq \mathrm{im}(f)$ which shows that \begin{equation}
\bigoplus_{v \in {\mathbb{N}^n}} (\mathrm{im}(f) \cap N_v) = \mathrm{im}(f) \notag.
\end{equation}
\end{proof}

\begin{Defi} Let $n \in \mathbb{N}_{\geq 1}$. Let $R$ be an $n$-graded commutative ring.
\begin{enumerate}
\item $\mathbf{Gr}_n(R)$ denotes the category of $n$-graded $R$-modules (and $n$-graded $\mathbb{F}$-vector spaces if  $R=\mathbb{F}$ is a field). Morphisms in $\mathbf{Gr}_n(R)$ are graded morphisms of $n$-graded $R$-modules. 
\item $\mathbf{Grf}_n(R)\subseteq \mathbf{Gr}_{n}(R)$ denotes the full  subcategory of finitely generated $n$-graded $R$-modules.
\end{enumerate}
\end{Defi}

\begin{Defi}\label{Definition of the module structure}
Let $\mathbb{F}$ be a field and $n \in \mathbb{N}_{\geq 1}$. 
\begin{enumerate}
\item Denote by 
\begin{equation}
A_n^{\mathbb{F}}:=\mathbb{F}[x_1,\dots,x_n] \notag 
\end{equation}
the polynomial ring in $n$ variables.
\item For $v=(v_1, \dots,v_n)\in \mathbb{N}^n$, we define \begin{equation}
x^v:=x_1^{v_1} {\cdot {\dots}  \cdot} x_n^{v_n} \notag.
\end{equation}
\item $A_n^{\mathbb{F}}$ becomes an $n$-graded commutative ring where for $v \in {\mathbb{N}^n}$, \begin{equation}
{(A_n^{\mathbb{F}})_v:= \braket{x^v}_\mathbb{F}} \notag.
\end{equation}
Note that for $v=0$, $(A_n^{\mathbb{F}})_0= \braket{x^0}_{\mathbb{F}}=\mathbb{F}$. 
\item $\mathbb{F}$ becomes an $n$-graded commutative ring by setting \begin{equation}
\mathbb{F}_v:=\begin{cases} \mathbb{F}, & v=0 \\
0,& \mathrm{else} \end{cases} \notag
\end{equation}
Moreover, $\mathbb{F}$ becomes an $n$-graded $A_n^{\mathbb{F}}$-module by setting \begin{equation}
a \cdot x^v= x^v \cdot a := 0 \notag
\end{equation}
for all $a \in \mathbb{F}$ and all $v \in \mathbb{N}^n_{ \succ 0}$.
\item We usually denote the graded maximal ideal $\braket{x_1, \dots,x_n}_{A_n^{\mathbb{F}}}\subseteq  A_n^{\mathbb{F}}$ by $\mathfrak{m}^{\mathbb{F}}_n$. Note that \begin{equation}
\mathfrak{m}^{\mathbb{F}}_n=\bigoplus_{v \in {\mathbb{N}^n_{\succ 0}}} (A_n^{\mathbb{F}})_v \quad \textrm{and} \quad
A_n^{\mathbb{F}}/\mathfrak{m}^{\mathbb{F}}_n = \mathbb{F} \notag.
\end{equation}
\end{enumerate}
\end{Defi}

\begin{Rem}\label{finite image} Let $M \in \mathbf{Grf}_n(A_n^{\mathbb{F}})$ and $V \in \mathbf{Grf}_n(\mathbb{F})$.
\begin{enumerate}
\item $V_w$ is a finite dimensional $\mathbb{F}$-vector space for all $w \in \mathbb{N}^n$.
\item $M_w$ is a finite dimensional $\mathbb{F}$-vector space for all $w \in \mathbb{N}^n$.
\end{enumerate}
\end{Rem}
\newpage
\subsection{Invariants}\label{Invariants}
For the following, let $\mathbb{F}$ be a field and $n \in \mathbb{N}_{\geq 1} $. Denote by $\mathcal{K}$ the class of all fields. In this article, we are interested in invariants for $ \mathbf{Grf}_n(A_n^{\mathbb{F}})/ _{\cong}$ where $\mathbf{Grf}_n(A_n^{\mathbb{F}})/ _{\cong}$ denotes the set of all isomorphim classes $[M]$ of objects $M \in\mathbf{Grf}_n(A_n^{\mathbb{F}})$. Note that $\mathbf{Grf}_n(A_n^{\mathbb{F}})/ _{\cong}$ is indeed a set (see \cref{is a set}). For the following, let $Q_n^{\mathbb{F}}$ be a set and $X_n^{\mathbb{F}} \subseteq \mathbf{Grf}_n(A_n^{\mathbb{F}})/ _{\cong}$ be a subset.

\begin{Defi}[Invariant] An \textit{invariant for} $X_n^{\mathbb{F}}$ \textit{with values in} $Q_n^{\mathbb{F}}$ is a function
\begin{equation}
f_n^{\mathbb{F}}:  X_n^{\mathbb{F}} \longrightarrow  Q_n^{\mathbb{F}} \notag.
\end{equation}
Sometimes, if the context is clear, we just say that $f_n^{\mathbb{F}}$ is an \textit{invariant for} $X_n^{\mathbb{F}}$ or that $f_n^{\mathbb{F}}$ is an \textit{invariant}.
\end{Defi}

\begin{Defi}[Equivalence] Two invariants $f_n^{\mathbb{F}}: X_n^{\mathbb{F}}\to Q_n^{\mathbb{F}}$ and $g_n^{\mathbb{F}}: X_n^{\mathbb{F}} \to Z_n^{\mathbb{F}}$ are called \textit{equivalent}, write $f_n^{\mathbb{F}} \cong g_n^{\mathbb{F}}$, if there exists a set theoretic bijection \begin{equation}
s_n^{\mathbb{F}}: \mathrm{im}(f_n^{\mathbb{F}}) \xlongrightarrow{\sim} \mathrm{im}(g_n^{\mathbb{F}})\notag 
\end{equation}
such that the diagram
\begin{equation}
\begin{tikzcd}[row sep=1cm, column sep = 1.5cm] & \mathrm{im}(f_n^{\mathbb{F}}) \arrow{dd}[rotate=90,yshift=1ex,xshift=-1ex]{\sim}{s_n^{\mathbb{F}}} \\
X_n^{\mathbb{F}} \arrow[rd, "{g_n^{\mathbb{F}}}"']  \arrow[ru, "{f_n^{\mathbb{F}}}"]  \\
 & \mathrm{im}(g_n^{\mathbb{F}})
\end{tikzcd} \notag
\end{equation} 
commutes. Two classes of invariants $\{f_n^{\mathbb{F}}\}_{\mathbb{F} \in \mathcal{K}}$ and $\{g_n^{\mathbb{F}}\}_{\mathbb{F} \in \mathcal{K}}$ are called \textit{equivalent} if for all $\mathbb{F} \in \mathcal{K}$, $f_n^{\mathbb{F}} \cong g_n^{\mathbb{F}}$. In this case, we write $\{f_n^{\mathbb{F}}\}_{\mathbb{F} \in \mathcal{K}} \cong \{g_n^{\mathbb{F}}\}_{\mathbb{F} \in \mathcal{K}}$.
\end{Defi}

\begin{Defi}[Complete invariant]\label{complete invariant} An invariant
\begin{equation}
f_n^{\mathbb{F}}:  X_n^{\mathbb{F}} \longrightarrow  Q_n^{\mathbb{F}} \notag
\end{equation}
is called \textit{complete} if $f_n^{\mathbb{F}}$ is injective.  A class of invariants $\{f_n^{\mathbb{F}}\}_{\mathbb{F} \in \mathcal{K}}$ is called \textit{complete} if $f_n^{\mathbb{F}}$ is \textit{complete} for all $\mathbb{F} \in \mathcal{K}$.
\end{Defi}

\begin{Prop}\label{Complete invariants equivalent} If $f_n^{\mathbb{F}}: X_n^{\mathbb{F}} \to Q_n^{\mathbb{F}}$ and $g_n^{\mathbb{F}}: X_n^{\mathbb{F}} \to Z_n^{\mathbb{F}}$ are complete invariants, then $f_n^{\mathbb{F}} \cong g_n^{\mathbb{F}}$.
\end{Prop}
\begin{proof} This is immediate from the definition.
\end{proof}

Our next definition is from \cite[Sec.~1.2]{article} (see also \cite[Sec.~1.2]{Carlsson2009}).
\begin{Defi}[Discrete class of invariants]\label{discrete invariant} If we have a class of invariants \begin{equation}
\{f_n^{\mathbb{F}}: X_n^{\mathbb{F}} \longrightarrow Q_n^{\mathbb{F}}\}_{\mathbb{F} \in \mathcal{K}} \notag,
\end{equation}
then $\{f_n^{\mathbb{F}} \}_{\mathbb{F} \in \mathcal{K}}$ is called a \textit{discrete class of invariants} if $\mathrm{im}(f_n^{\mathbb{F}}) \subseteq Q_n^{\mathbb{F}}$ is countable and if $Q_n^{\mathbb{F}}=Q_n^{\mathbb{T}}$ for all $\mathbb{T}, \mathbb{F} \in \mathcal{K}$. 
\end{Defi}

\begin{Defi}[Continuous class of invariants] If $\{f_n^{\mathbb{F}} \}_{\mathbb{F} \in \mathcal{K}}$ is a class of invariants which is not discrete, then $\{f_n^{\mathbb{F}} \}_{\mathbb{F} \in \mathcal{K}}$ is also called a \textit{continuous class of invariants}.
\end{Defi}

\section{Complete Classification}\label{Chapter 1}

As we have already mentioned in the introduction, this section is devoted to the classification of finitely generated $n$-graded $A_n^{\mathbb{F}}$-modules. In \cref{section 1.4}, we answer the question of how the study of the homology of a finite $n$-filtered simplicial complex translates naturally into the study of finitely generated $n$-graded $A_n^{\mathbb{F}}$-modules. In \cref{Free hulls and free section}, we present the theory behind free hulls. Free hulls are the basic building block for the complete classification of finitely generated $n$-graded $A_n^{\mathbb{F}}$-modules. This complete classification is then carried out in \cref{complete classification section}. To approach the question of whether this classification can provide a complete and discrete class of invariants, we first review the one-dimensional case in \cref{One dim pers sec}, where the \textit{barcode} provides a complete and discrete class of invariants. Then we turn our attention to the \textit{rank invariant} for multidimensional persistence, which is discrete and equivalent to the barcode in dimension $n=1$. But as we will see in \cref{Chapter 2}, there is no complete and discrete class of invariants in dimension $n \geq 2$.

\subsection{Correspondence and realization}\label{section 1.4}

The next definition is standard and can be found in \cite{article}, \cite{Carlsson2009} and \cite{harrington2019stratifying}.
\begin{Defi}[Persistence module]\label{6} Let $\mathcal{C}$ be a small category. A \textit{persistence module over} $\mathbb{F}$ is a functor 
\begin{equation}
P:\mathcal{C}\to \mathbf{Vec}_\mathbb{F}  \notag
\end{equation}
where $\mathbf{Vec}_\mathbb{F}$ denotes the category of $\mathbb{F}$-vector spaces with $\mathbb{F}$-linear maps as morphisms. A morphism of persistence modules over $\mathbb{F}$ is a natural transformation of functors. 
\begin{enumerate}
\item $P$ is called $n$-\textit{dimensional} if $\mathcal{C}=({\mathbb{N}^n},\preceq)$.  \item $\mathbf{Pers}_n(\mathbb{F})$ denotes the category of $n$-dimensional persistence modules over $\mathbb{F}$.
\end{enumerate}
\end{Defi}

Our next definition is as in \cite[Def.~2]{article} and \cite[Def.~2]{Carlsson2009}.

\begin{Defi}\label{7} Given $P\in \mathbf{Pers}_n(\mathbb{F})$, we define an $n$-graded $A_n^{\mathbb{F}}$-module
\begin{ceqn}
\begin{align}
\alpha_{n}^{\mathbb{F}}(P):=\bigoplus_{v\in {\mathbb{N}^n}}P_v \in  \mathbf{Gr}_n(A_n^{\mathbb{F}}) \notag
\end{align}
\end{ceqn}
where for $y_v \in P_v$ and $u \in \mathbb{N}^n$, 
\begin{equation}
x^u \cdot y_v:=P(v \preceq v+u)(y_v) \in P_{v+u} \notag.
\end{equation}
$\alpha_n^{\mathbb{F}}(P)$ is also called \textit{the} $n$-\textit{graded} $A_n^{\mathbb{F}}$-\textit{module associated to} $P$.
\end{Defi}

\begin{Thm}[{{\cite[Thm.~2.6] {harrington2019stratifying}}}]\label{8} The correspondence $\alpha_n^{\mathbb{F}} $ defines an isomorphism of categories between
$\mathbf{Pers}_n(\mathbb{F})$ and $\mathbf{Gr}_n(A_n^{\mathbb{F}})$.
\end{Thm}
In \cite[Thm.~2.6]{harrington2019stratifying}, \cref{8} is formulated (without giving a proof) with a reference to \cite{Carlsson2009}. The proof mainly consists in reformulating the defining properties of objects and their morphisms in $\mathbf{Pers}_n(\mathbb{F})$ and $\mathbf{Gr}_n(A_n^{\mathbb{F}})$ accordingly and is therefore not elaborated here. The next definition is standard material (see \cite{harrington2019stratifying}, \cite{article} and \cite{Carlsson2009}).

\begin{Defi}\label{geom real forward} Let $
X=(X_{u})_{u \in {\mathbb{N}^n}}$ be an $n$-filtered simplicial complex. Let
\begin{equation}
P^{\mathbb{F}}_{n,l}(X): {\mathbb{N}^n} \longrightarrow \mathbf{Vec}_{\mathbb{F}}, \quad u  \longmapsto  H_l(X_{u}, \mathbb{F}) \notag
\end{equation}
where  $H_l(X_{u}, \mathbb{F})$ denotes the $l$-th simplicial homology group of $X_u$ (note that $H_l(X_{u}, \mathbb{F})$ is an $\mathbb{F}$-vector space) and let
\begin{equation}
P^{\mathbb{F}}_{n,l}(u\preceq v):= H_l(i_{u,v}, \mathbb{F}): H_l(X_{u}, \mathbb{F}) \longrightarrow H_l(X_{v}, \mathbb{F}) \notag
\end{equation}
where $i_{u,v}: X_{u} \lhook\joinrel\xlongrightarrow{\subseteq} X_{v}$ denotes the canonical inclusion. Then $P^{\mathbb{F}}_{n,l}(X) \in \mathbf{Per}_n(\mathbb{F})$ defines an $n$-dimensional persistence module over $\mathbb{F}$ for all $l \in \mathbb{N}$. Now define
\begin{equation}
M_{n,l}^{\mathbb{F}}(X):= \alpha_{n}^{\mathbb{F}} ( P_{n,l}^{\mathbb{F}}(X)) \in \mathbf{Gr}_n(A_n^{\mathbb{F}}) \notag.
\end{equation}
Note that if $X$ is finite, then \begin{equation}
M_{n,l}^{\mathbb{F}}(X) \in \mathbf{Grf}_n(A_n^{\mathbb{F}}) \notag.
\end{equation}
\end{Defi}

As we can see, every finite $n$-filtered simplicial complex defines a finitely generated $n$-graded $A_n^{\mathbb{F}}$-module. Moreover, for certain fields, objects in $\mathbf{Grf}_n(A_n^{\mathbb{F}})$ can be realized as a finite $n$-filtered simplicial complex:

\begin{Thm}[{{\cite[Thm.~2.11]{harrington2019stratifying}}}]\label{geom rel modules} Assume that $\mathbb{F}=\mathbb{F}_p$ where $p \in \mathbb{N}$ is a prime number or that $\mathbb{F}=\mathbb{Q}$. Let $M\in \mathbf{Grf}_n(A_n^{\mathbb{F}})$. Then for every $l \geq 1$, there exists an $n$-filtered finite simplicial complex $X$ such that $M_{n,l}^{\mathbb{F}}(X)\cong M$.
\end{Thm}
In \cite{harrington2019stratifying}, \cref{geom rel modules} is also stated and proven for $\mathbb{Z}$-coefficients. In \cite[Thm.~2]{article} and \cite[Thm.~2]{Carlsson2009}, \cref{geom rel modules} is stated (without giving a proof) for finite persistence modules with $\mathbb{F}=\mathbb{F}_p$.

\subsection{Free hulls}\label{Free hulls and free section}
In this section, we present the theory of free hulls. This section is mainly based on \cite[Sec.~4]{article} and \cite[Sec.~4.4]{Carlsson2009}.
\subsubsection{Free objects}

The next definition introduces the the canonical model for objects in $\mathbf{Grf}_n(\mathbb{F})$ and graded free objects in $\mathbf{Grf}_n(A_n^{\mathbb{F}})$.

\begin{Defi}\label{Defi Free objects} Let $\xi=(V, \mu)$ be a finite $n$-dimensional multiset. 
Then we associate to $\xi$ a finitely generated $n$-graded $\mathbb{F}$-vector space
\begin{equation}
\mathcal{V}_n^{\mathbb{F}}(\xi):= \bigoplus_{v \in V} \mathbb{F}(v)^{\mu(v)} \in \mathbf{Grf}_n(\mathbb{F}) \notag
\end{equation}
and an $n$-graded finitely generated graded free $A_n^{\mathbb{F}}$-module
\begin{equation}
\mathcal{F}_n^{\mathbb{F}}(\xi):= \bigoplus_{v \in V} A_n^{\mathbb{F}}(v)^{\mu(v)} \in \mathbf{Grf}_n(A_n^{\mathbb{F}}) \notag.
\end{equation}

\end{Defi}

Let us give some examples of \cref{Defi Free objects}.

\begin{Exa}\label{Exa free objects}\,
\begin{enumerate}
\item Consider the trivial one-dimensional multiset $\xi=\{(0,1)\}$. Then \begin{equation}{\mathcal{V}_1^{\mathbb{F}}(\xi)=\mathbb{F}(0) =\mathbb{F}} \quad \textrm{and} \quad 
\mathcal{F}_1^{\mathbb{F}}(\xi)=A_1^{\mathbb{F}}(0)=  A_1^{\mathbb{F}} \notag.
\end{equation}
\item Consider the one-dimensional multiset $\xi=\{(1,1),(2,1)\}$. Then  \begin{equation}
\mathcal{V}_1^{\mathbb{F}}(\xi)=\mathbb{F}(1) \oplus \mathbb{F}(2) \quad
\textrm{and}
\quad \mathcal{F}_1^{\mathbb{F}}(\xi)=A_1^{\mathbb{F}}(1) \oplus A_1^{\mathbb{F}}(2) \notag.
\end{equation}
This example is illustrated in Figure \ref{fig:Free module dim 1}. Recall that $A_1^{\mathbb{F}}=\mathbb{F}[x]$. The blue blocks represent the graded parts of $\mathcal{F}_1^{\mathbb{F}}(\xi)$ and $\mathcal{V}_1^{\mathbb{F}}(\xi)$. The blocks with an $\mathbb{F}$ are the locations of the homogeneous basis elements of $\mathcal{F}_1^{\mathbb{F}}(\xi)$ and $\mathcal{V}_1^{\mathbb{F}}(\xi)$. On the horizontal axis we see the number of direct summands. On the vertical $x$-axis we may read off the the degree of the blocks and we see how the generators are shifted. Moreover, the $x$-axis on the left-hand side of the picture indicates that we can shift the degree or, in other words, that we can navigate through the module via multiplication by $x$.
\begin{figure}[h]
\centering
\includegraphics[scale=0.4]{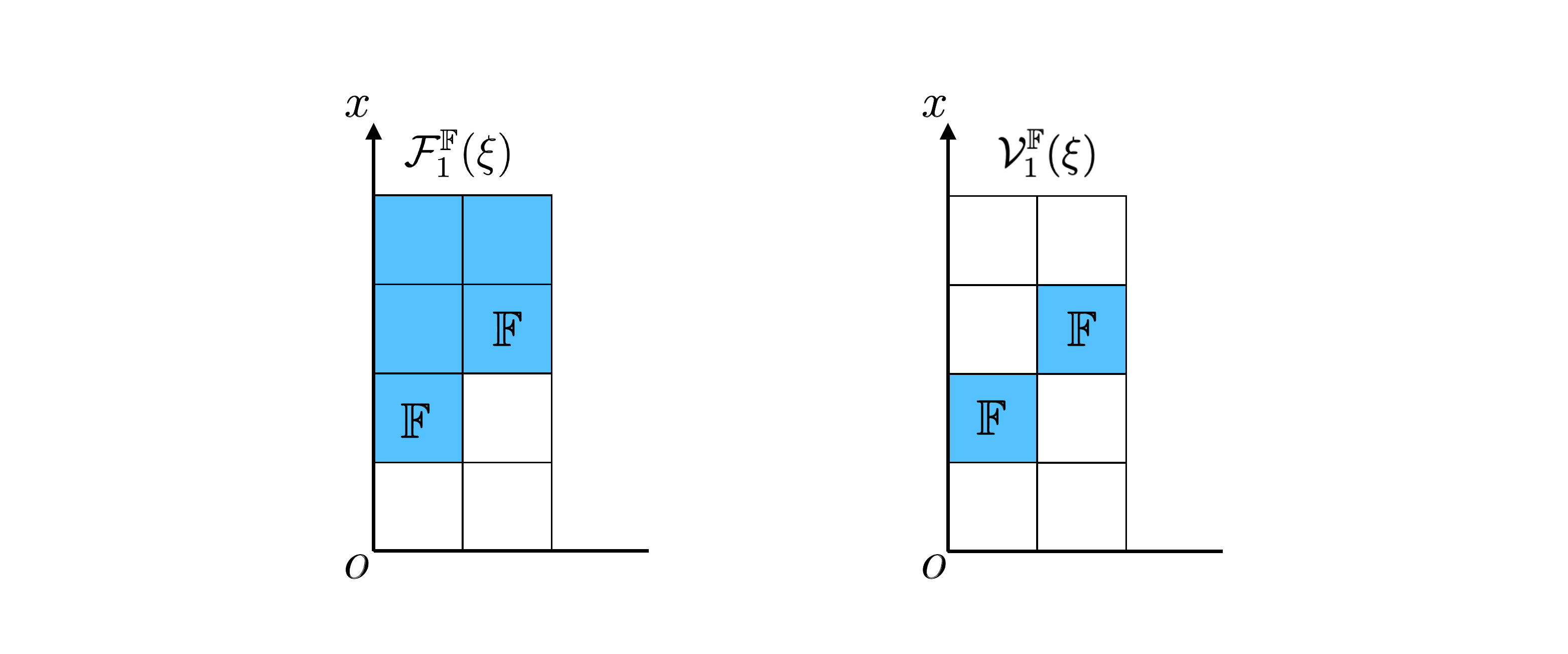} 
\caption{Consider the one-dimensional multiset ${\xi=\{(1,1),(2,1)\}}$. This figure illustrates $\mathcal{F}_1^{\mathbb{F}}(\xi)$ and $\mathcal{V}_1^{\mathbb{F}}(\xi)$ over $\xi$ (see \cref{Exa free objects}).}
\label{fig:Free module dim 1}
\end{figure}
\begin{figure}[h]
\centering
\includegraphics[scale=0.2]{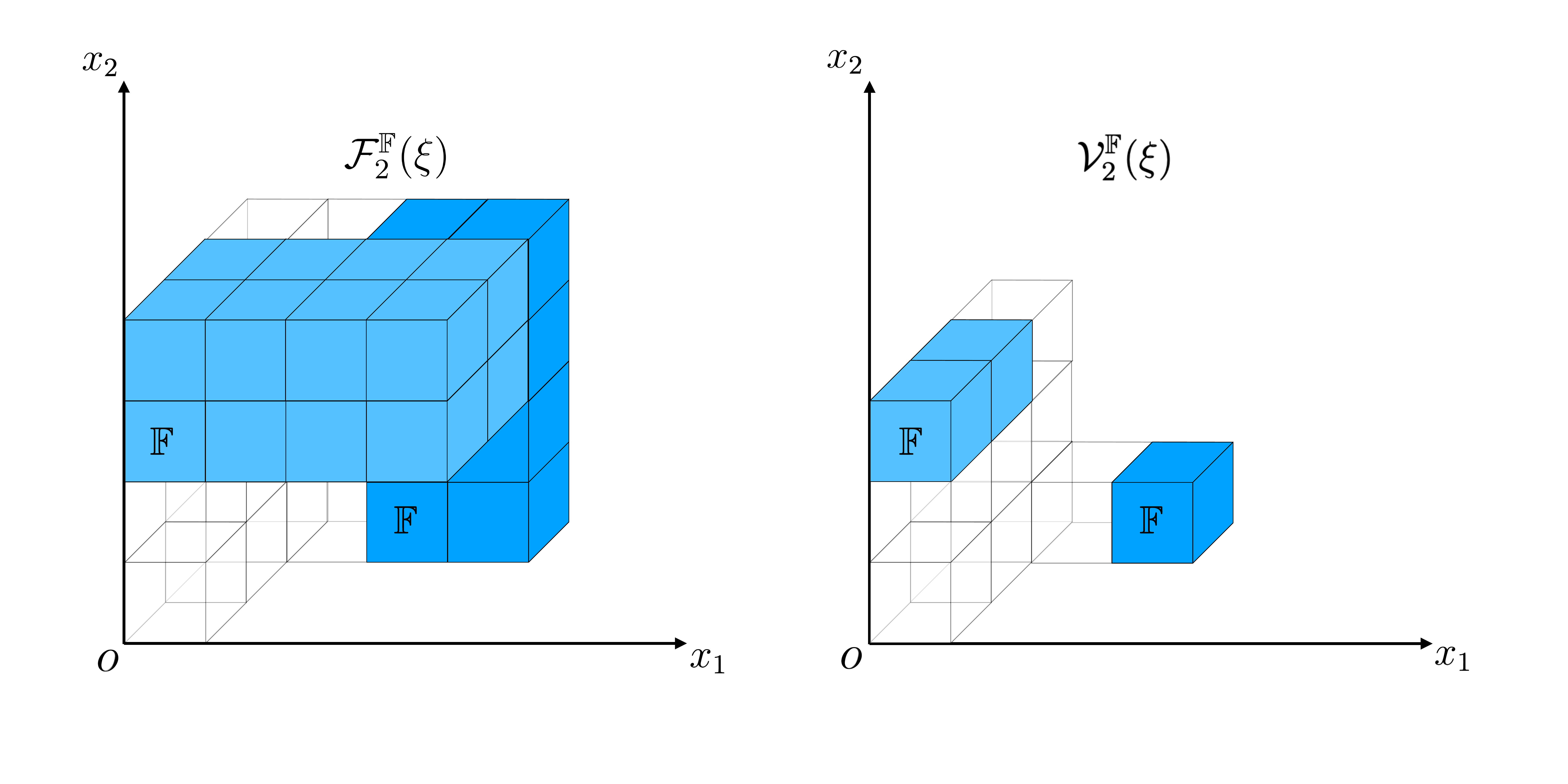} 
\caption{Consider the two-dimensional multiset ${\xi=\{((0,2),1),((0,2),2),((2,0),1)\}}$. This figure illustrates $\mathcal{F}_2^{\mathbb{F}}(\xi)$ and $\mathcal{V}_2^{\mathbb{F}}(\xi)$ over $\xi$ (see \cref{Exa free objects}).}
\label{fig:Freemodule Exa dim 2}
\end{figure}
\newpage
\item Consider the two-dimensional multiset $\xi=\{((0,2),1),((0,2),2),((2,0),1)\}$. Then \begin{equation}
\mathcal{V}_2^{\mathbb{F}}(\xi)=\mathbb{F}(0,2) \oplus \mathbb{F}(0,2) \oplus \mathbb{F}(2,0)\notag
\end{equation}
and
\begin{equation}
\mathcal{F}_2^{\mathbb{F}}(\xi)=A_2^{\mathbb{F}}(0,2) \oplus A_2^{\mathbb{F}}(0,2) \oplus A_2^{\mathbb{F}}(2,0) \notag.
\end{equation}
This example is illustrated in Figure \ref{fig:Freemodule Exa dim 2}. Recall that $A_2^{\mathbb{F}}=\mathbb{F}[x_1,x_2]$. The blue blocks represent the graded parts of $\mathcal{F}_2^{\mathbb{F}}(\xi)$ and $\mathcal{V}_2^{\mathbb{F}}(\xi)$. The blocks with an $\mathbb{F}$ are the locations of the homogeneous basis elements of $\mathcal{F}_2^{\mathbb{F}}(\xi)$ and $\mathcal{V}_2^{\mathbb{F}}(\xi)$. The third dimension parameterizes the number of direct summands. On the $x_1,x_2$-axis we may read off the the degree of the blocks and we see how the generators are shifted. Moreover, the $x_1,x_2$-axes on the left-hand side of the picture indicate that we can shift the degree or, in other words, that we can navigate through the module via multiplication by $x_1,x_2$.
\end{enumerate}
\end{Exa}

As one would expect, we have the following:
\begin{Defiprop}\label{free hull vec}
For any $V \in \mathbf{Grf}_n(\mathbb{F})$, there exists a finite $n$-dimensional multiset $\xi$ and a graded $\mathbb{F}$-vector space isomorphism \begin{equation}
p:  \mathcal{V}_n^{\mathbb{F}}(\xi) \xlongrightarrow{\sim} V \notag.
\end{equation} 
The pair $(\mathcal{V}_n^{\mathbb{F}}(\xi),p)$ is unique in the  following sense: if there exists another pair $(\mathcal{V}_n^{\mathbb{F}}(\xi'),p')$ such that $p':  \mathcal{V}_n^{\mathbb{F}}(\xi') \xlongrightarrow{\sim} V$ is a graded isomorphism, then 
\begin{equation}
p^{-1} \circ p':\mathcal{V}_n^{\mathbb{F}}(\xi') \xlongrightarrow{\sim} \mathcal{V}_n^{\mathbb{F}}(\xi) \notag
\end{equation}
is a graded isomorphism and we have $\xi=\xi'$.
\begin{equation}
\Xi^{\mathbb{F}}_n(V):=\xi \notag
\end{equation}
is called the \textit{type} of $V$.
\end{Defiprop}

\begin{proof} Let $G \subseteq V$ be a finite system of homogenous generators of the $n$-graded $\mathbb{F}$-vector space $V$ (such a set of homogeneous generators exists by \cref{Homog generators}) and choose a basis $B \subseteq G$. Then $B$ is a homogeneous basis of $V$. Now every $b \in B$ is of the form $b=b_v \in V_v$ for some unique $v \in {\mathbb{N}^n}$. This determines a finite $n$-dimensional mulitset $\xi$ and a graded isomorphism $p:\mathcal{V}_n^{\mathbb{F}}(\xi) \to V$. Formally, this is done as follows:

we can subdivide $B=\bigcup_{v\in T} \{b_{1,v}, \dots, b_{d_v,v}\}$ for a suitable subset $T \subseteq {\mathbb{N}^n}$ and suitable $d_v \in \mathbb{N}$ such that $b_{i,v} \in V_v$ for all $i \in \{1, \dots, d_v \}$ and all $v \in T$. Now define ${\xi:=\bigcup_{v \in T}\{(v,1), \dots, (v, d_v) \}}$.
So, $\xi=(T, \mu)$ where \begin{equation}
\mu: T \longrightarrow \mathbb{N}_{\geq 1}, \quad v \longmapsto d_v \notag.
\end{equation}
This leads to a graded isomorphism
\begin{equation}
p: \mathcal{V}_n^{\mathbb{F}}(\xi) \longrightarrow V, \quad e_{i,v} \longmapsto b_{i, v} \notag
\end{equation} 
where $e_{i,v}$ is the $i$-th standard basis vector in $\mathbb{F}(v)^{\mu(v)}$.

 If we have a second finite $n$-dimensional multiset $\xi'$ together with a graded isomorphism $p': \mathcal{V}_n^{\mathbb{F}}(\xi') \to V$, then $p \circ p': \mathcal{V}_n^{\mathbb{F}}(\xi') \to \mathcal{V}_n^{\mathbb{F}}(\xi)$ is a graded isomorphism. So, we conclude that $\xi=\xi'$.
\end{proof}
\begin{Rem}\label{Rank well definde} Let $R$ be a commutative ring and $F$ be a finitely generated free $R$-module. Then any basis of $F$ has finite length and any two bases are of the same length. So, $\mathrm{rank}(F)$ (the rank of $F$) is well-defined and finite.
\end{Rem}

\begin{Defiprop}\label{graded free isom multiset}
For any graded free $F \in \mathbf{Grf}_n(A_n^{\mathbb{F}})$, there exists an $n$-dimensional multiset $\xi$ and a graded $A_n^{\mathbb{F}}$-module isomorphism \begin{equation}
p:  \mathcal{F}_n^{\mathbb{F}}(\xi) \xlongrightarrow{\sim} F \notag.
\end{equation} 
The pair $(\mathcal{F}_n^{\mathbb{F}}(\xi),p)$ is unique in the  following sense: if there exists another pair $(\mathcal{F}_n^{\mathbb{F}}(\xi'),p')$ such that $p':  \mathcal{F}_n^{\mathbb{F}}(\xi') \xlongrightarrow{\sim} F$ is a graded isomorphism, then 
\begin{equation}
p^{-1} \circ p':\mathcal{F}_n^{\mathbb{F}}(\xi') \xlongrightarrow{\sim} \mathcal{F}_n^{\mathbb{F}}(\xi) \notag
\end{equation}
is a graded isomorphism and we have $\xi=\xi'$.
\begin{equation}
\Xi^{\mathbb{F}}_n(F):=\xi \notag
\end{equation}
is called the \textit{type} of $F$.
\end{Defiprop}
\begin{proof} Since $A_n^{\mathbb{F}}$ is commutative and $F$ is finitely generated graded free, $F$ admits a finite homogeneous $A_n^{\mathbb{F}}$-basis $B\subseteq F$. Now the argument is analogous as in the proof of \cref{free hull vec}: Every $b \in B$ is of the form $b=b_v \in F_v$ for some unique $v \in {\mathbb{N}^n}$. This determines a finite $n$-dimensional mulitset $\xi$ and a graded isomorphism $p:\mathcal{F}_n^{\mathbb{F}}(\xi) \to F$. Formally, this is done as follows:

we can subdivide $B=\bigcup_{v\in T} \{b_{1,v}, \dots, b_{d_v,v}\}$ for a suitable subset $T \subseteq {\mathbb{N}^n}$ and suitable $d_v \in \mathbb{N}$ such that $b_{i,v} \in F_v$ for all $i \in \{1, \dots, d_v \}$ and all $v \in T$. Now define ${\xi:=\bigcup_{v \in T}\{(v,1), \dots, (v, d_v) \}}$. So, $\xi=(T, \mu)$ where \begin{equation}
\mu: T \longrightarrow  \mathbb{N}_{\geq 1}, \quad v \longmapsto d_v \notag.
\end{equation}
This leads to an isomorphism of $n$-graded $A_n^{\mathbb{F}}$-modules
\begin{equation}
p: \mathcal{F}_n^{\mathbb{F}}(\xi) \longrightarrow F, \quad e_{i,v} \longmapsto b_{i, v} \notag
\end{equation} 
where $e_{i,v}$ is the $i$-th standard basis vector in $A_n^{\mathbb{F}}(v)^{\mu(v)}$. If we have a second finite $n$-dimensional multiset $\xi'$ together with a graded isomorphism $p': \mathcal{F}_n^{\mathbb{F}}(\xi') \to F$, then $p \circ p': \mathcal{F}_n^{\mathbb{F}}(\xi') \to \mathcal{F}_n^{\mathbb{F}}(\xi)$ is a graded isomorphism. So, we conclude that $\xi=\xi'$.
\end{proof}

\begin{figure}[h!]
\captionsetup{width=1\textwidth}
\centering
\includegraphics[scale=0.4]{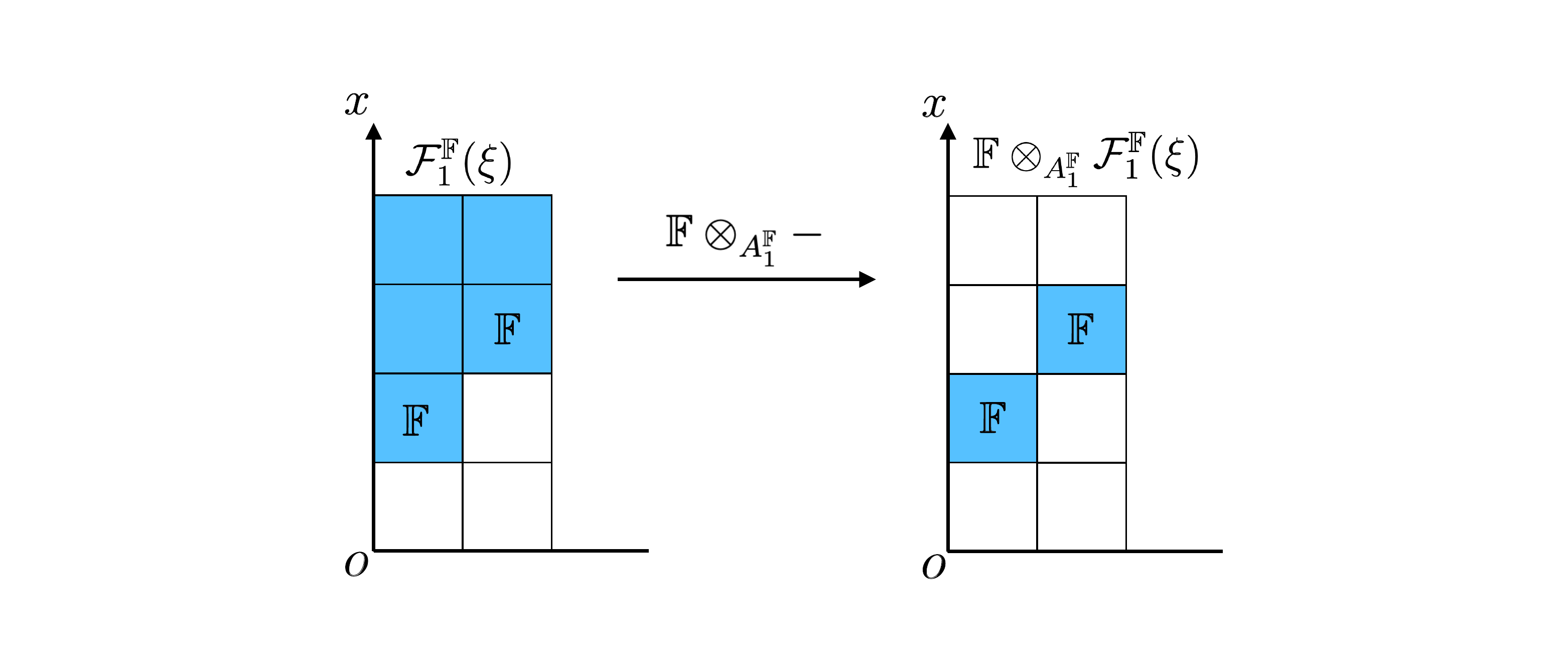} 
\caption{Illustration of the tensor-operation on $\mathcal{F}_1^{\mathbb{F}}(\xi)$ over the one-dimensional multiset $\xi=\{(1,1),(2,1)\}$ (continued from \cref{Exa free objects} and Figure \ref{fig:Free module dim 1}).}
\label{fig:Tensor operation}
\end{figure}

\subsubsection{The tensor-operation}

There is a natural way to map objects in $\mathbf{Grf}_n(A_n^{\mathbb{F}})$ to objects in $\mathbf{Grf}_n(\mathbb{F})$ by means of the so-called \textit{tensor-operation}. This tensor-operation will be necessary for the definition of \textit{free hulls} (see \cref{15}). Let $M \in \mathbf{Grf}_n(A_n^{\mathbb{F}})$.
Then \begin{equation}
\mathbb{F} \otimes_{A_n^{\mathbb{F}}} M \in \mathbf{Grf}_n(\mathbb{F}) \notag
\end{equation}
where for $u \in {\mathbb{N}^n}$, the $n$-grading is given by \begin{equation}\left (\mathbb{F} \otimes_{A_n^{\mathbb{F}}} M \right)_u := \left\{ \sum_{i=1}^l a_i \otimes m_i \, \bigg \vert \,   a_i \in \mathbb{F}, \, m_i \in M_u \right \} \notag
\end{equation}
(see also \cite[p.~153]{miller2004combinatorial}). Now
\begin{equation}
\mathbb{F} \otimes_{A_n^{\mathbb{F}}} -: \mathbf{Grf}_n(A_n^{\mathbb{F}}) \longrightarrow \mathbf{Grf}_n(\mathbb{F}) \notag
\end{equation}
is a right exact functor. We have a canonical isomorphism of finitely generated $n$-graded $\mathbb{F}$-vector spaces
\begin{equation}\label{can iso}
\mathbb{F} \otimes_{A_n^{\mathbb{F}}} M \xlongrightarrow{\sim} M/\mathfrak{m}_n^{\mathbb{F}} M, \quad a \otimes m \longmapsto a \cdot \pi(m)
\end{equation}
where $\pi: M \to M/\mathfrak{m}_n^{\mathbb{F}}M$ denotes the canonical projection. Note that $\pi$ is graded. We have \begin{equation}
\mathcal{V}_n^{\mathbb{F}}(\xi)=\mathcal{F}_n^{\mathbb{F}}(\xi)/\mathfrak{m}_n^{\mathbb{F}} \mathcal{F}_n^{\mathbb{F}}(\xi) \notag
\end{equation}
for any finite $n$-dimensional multiset $\xi$. So, (\ref{can iso}) yields an isomorphism of $n$-graded $\mathbb{F}$-vector spaces
\begin{equation}
\mathbb{F} \otimes_{A_n^{\mathbb{F}}} \mathcal{F}_n^{\mathbb{F}}(\xi) \xlongrightarrow{\sim} \mathcal{V}_n^{\mathbb{F}}(\xi) \notag.
\end{equation}
The tensor-operation is illustrated in Figures \ref{fig:Tensor operation} and \ref{fig:tensorop2}.

\begin{figure}[h!]
\captionsetup{width=1\textwidth}
\centering
\includegraphics[scale=0.2]{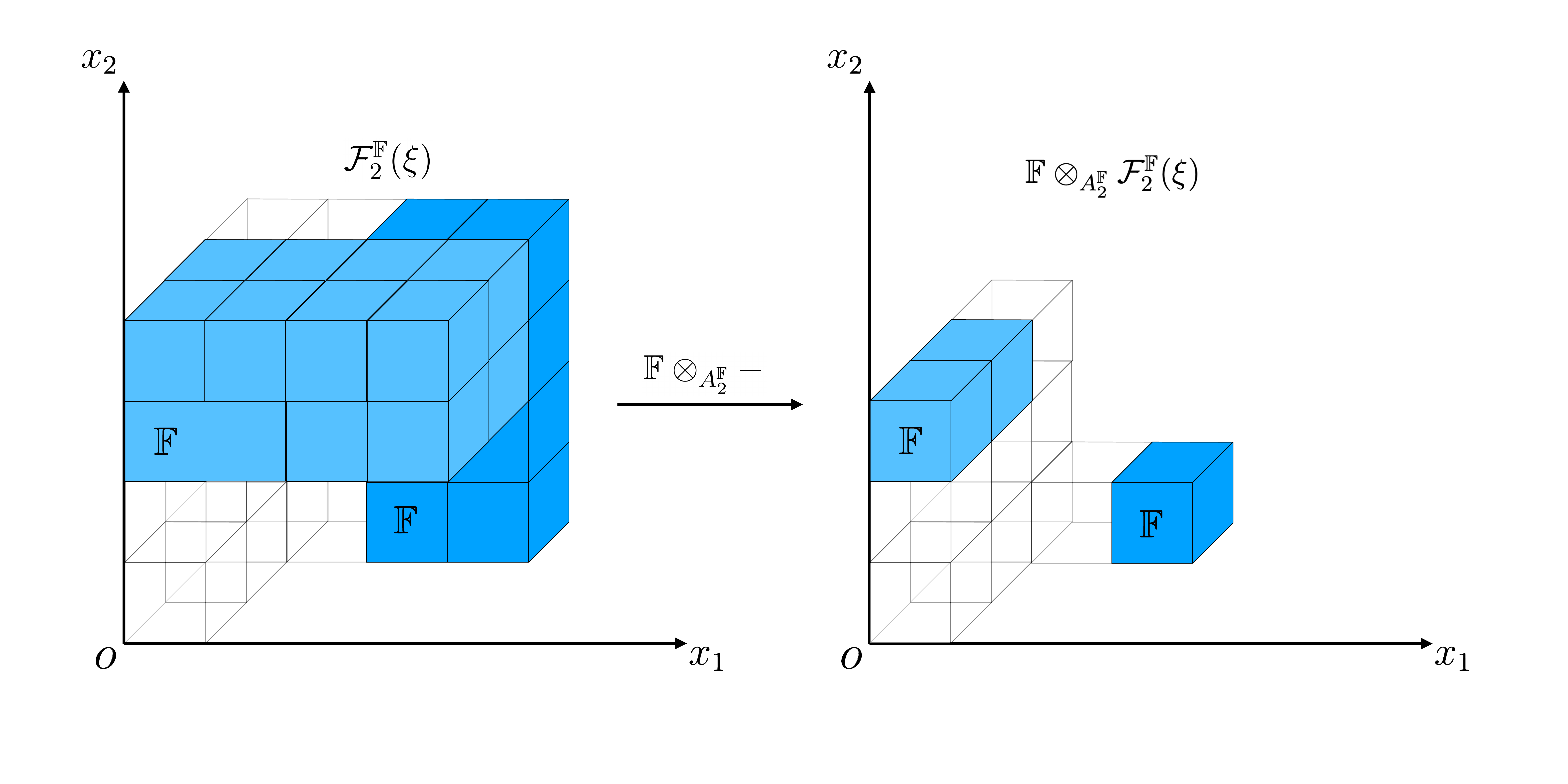} 
\caption{Illustration of the tensor-operation on $\mathcal{F}_2^{\mathbb{F}}(\xi)$ over the two-dimensional multiset $\xi=\{((0,2),1),((0,2),2),((2,0),1)\}$ (continued from \cref{Exa free objects} and Figure \ref{fig:Freemodule Exa dim 2}).}
\label{fig:tensorop2}
\end{figure}

\subsubsection{The main theorem on free hulls}
Using the tensor operation, we can now define \textit{free hulls}:
\begin{Defi}[{{\cite[Def.~4]{Carlsson2009}, \cite[Sec.~4.2]{article}, Free hull}}]\label{15} A \textit{free hull} of ${M\in \mathbf{Grf}_n(A_n^{\mathbb{F}})}$ is a pair $(F, p)$, where $p$ is a surjective graded $A_n^{\mathbb{F}}$-module homomorphism $p:F \to M$ and $F\in \mathbf{Grf}_n(A_n^{\mathbb{F}})$ is graded free such that
\begin{align}
\mathrm{id}_\mathbb{F} \otimes_{A_n^{\mathbb{F}}} p: \mathbb{F} \otimes_{A_n^{\mathbb{F}}} F \xlongrightarrow{\sim} \mathbb{F} \otimes_{A_n^{\mathbb{F}}} M \notag
\end{align}
is an isomorphism of $n$-graded $\mathbb{F}$-vector spaces.
\end{Defi}

The main goal for the rest of \cref{Free hulls and free section} is to prove the following theorem, the main theorem on free hulls, which states that free hulls exsist and that they are unique up to isomorphism.

\begin{Thm}[{{\cite[Thm.~7]{Carlsson2009}, \cite[Sec.~4.2]{article}}}]\label{Free hull thm} Every $M \in \mathbf{Grf}_n(A_n^{\mathbb{F}})$ admits a free hull
\begin{equation}
p:\mathcal{F}^{\mathbb{F}}_n(\xi) \longrightarrow M \notag
\end{equation}
for a finite $n$-dimensional multiset $\xi$. If $M$ is graded free, then $p$ is a graded isomorphism. Free hulls are unique in the following sense: if
$p:F \to M$ and $p':F' \to M$ are free hulls of $M$, then there exists a graded isomorphism $\psi: F \xlongrightarrow{\sim} F'$ such  that the  diagram
\begin{equation}
\begin{tikzcd} [row sep=1.5cm, column sep = 1.5cm]
F \arrow[rd,"p"'] \arrow[rr, "\psi", "\sim"'] && F' \arrow[ld, "p' "]   \\
 &M&
\end{tikzcd} \notag
\end{equation}
commutes. In particular, if $F=\mathcal{F}^{\mathbb{F}}_n(\xi)$ and $F'=\mathcal{F}^{\mathbb{F}}_n(\xi')$ for two finite $n$-dimensional multisets  $\xi$ and $\xi'$, then  $\xi=\xi'$.
\end{Thm}
In comparison to \cite[Thm.~7]{Carlsson2009} and \cite[Sec.~4.2]{article}, we just have added a few technical details to the formulation of \cref{Free hull thm}. In \cite[Thm.~7]{Carlsson2009} and \cite[Sec.~4.2]{article}, the authors give a sketch of the proof of \cref{Free hull thm}.
The proof of \cref{Free hull thm} requires some preparations and will be elaborated throughout Sections \ref{n-graded naka sec}-\ref{uniqueness sec}. The key to the proof is an $n$-graded version of \textit{Nakayama's lemma}.

Note that in \cref{Free hull thm}, the map $\psi$ is generally not unique. This is demonstrated by the next example.

\begin{Exa}\label{Exa tensoring dangerous} Recall that $A_1^{\mathbb{F}}=\mathbb{F}[x]$. Consider $F:=\mathbb{F}[x](1) \oplus \mathbb{F}[x](2)$. Then $F$ is graded free and $B:=\{(1,0),(x,1)\}$ and $B':=\{(1,0),(0,1)\}$ are homogeneous $\mathbb{F}[x]$-bases of $F$. Now 
\begin{align}
f:F \xlongrightarrow{\sim} F, & \quad (1,0) \longmapsto (1,0) \notag \\
& \quad (0,1) \longmapsto (0,1)  \notag \\
f':F \xlongrightarrow{\sim} F, & \quad (1,0) \longmapsto (1,0) \notag \\
& \quad (0,1) \longmapsto (x,1) \notag
\end{align}
are graded isomorphisms with $f \neq f'$ and $\mathrm{id}_{\mathbb{F}} \otimes_{\mathbb{F}[x]} f = \mathrm{id}_{\mathbb{F}}  \otimes_{\mathbb{F}[x]} f'$. This is due to the fact that for all $a \in \mathbb{F}$,
\begin{align}
a \otimes (x,1) & =a \otimes (x,0) +a \otimes (0,1)=a \otimes x \cdot  (1,0)  +a \otimes (0,1) \notag \\
& = a \cdot  x \otimes (1,0)+a\otimes (0,1) = 0 \otimes (1,0) +a \otimes (0,1)=a \otimes (0,1) \notag,
\end{align}
where $a\cdot x=x\cdot a=0$ for all $a \in \mathbb{F}$ by definition of the $\mathbb{F}[x]$-module structure on $\mathbb{F}$.
\end{Exa}

Before we proceed with the elaboration of the proof of \cref{Free hull thm}, we give an example which illustrates the concept of free hulls in two dimensions.
\newpage
\begin{Exa}\label{free hulls exa dim 2} This example is illustrated in Figure \ref{fig:Tensoroperation bigraded}. Recall that $A_2^{\mathbb{F}}=\mathbb{F}[x_1, x_2]$. Consider the graded submodule
\begin{align}
M&:=\braket{x_1^3, x_1^2x_2,x_1x_2^2,x_2^3 }_{A^{\mathbb{F}}_2}\subseteq A^{\mathbb{F}}_2 \notag.
\end{align}
Now
\begin{equation}
G:=\{x_1^3, x_1^2x_2,x_1x_2^2,x_2^3\} \subseteq M \notag
\end{equation}
is a minimal set of homogeneous generators of $M$. We have 
\begin{align}
\mathrm{deg}(x_1^3)=(3,0),\quad \mathrm{deg}(x_1^2x_2)=(2,1), \quad
\mathrm{deg}( x_1x_2^2)=(1,2), \quad 
\mathrm{deg}(x_2^3)=(0,3) \notag.
\end{align}
$M$ is not graded free (in \cref{equiv char free hulls}, we will see that for ${M\in \mathbf{Grf}_n(A_n^{\mathbb{F}})}$ graded free is equivalent to free). If $M$ were graded free, then any minimal set of homogeneous generators would be a basis (see \cref{minimal free graded basis}). But consider for example
\begin{equation}
x_1 \cdot x_1 x_2^2=x_1^2 x_2^2= x_2 \cdot x_1^2 x_2 \notag.
\end{equation}
As we can see, $x_1^2 x_2^2$ is not a unique $A_2^{\mathbb{F}}$-linear combination of elements in $G$. Hence, $G$ is not a basis of $M$ and since $G$ is minimal, $M$ is therefore not graded free. The fact that  $x_1^2x_2^2$ can be obtained by shifting different generators is illustrated in Figure \ref{fig:Tensoroperation bigraded} by different shades of the colour blue. These different shades indicate that we may have an overlap if we shift different generators.
\begin{figure}[h]
\centering
\includegraphics[scale=0.25]{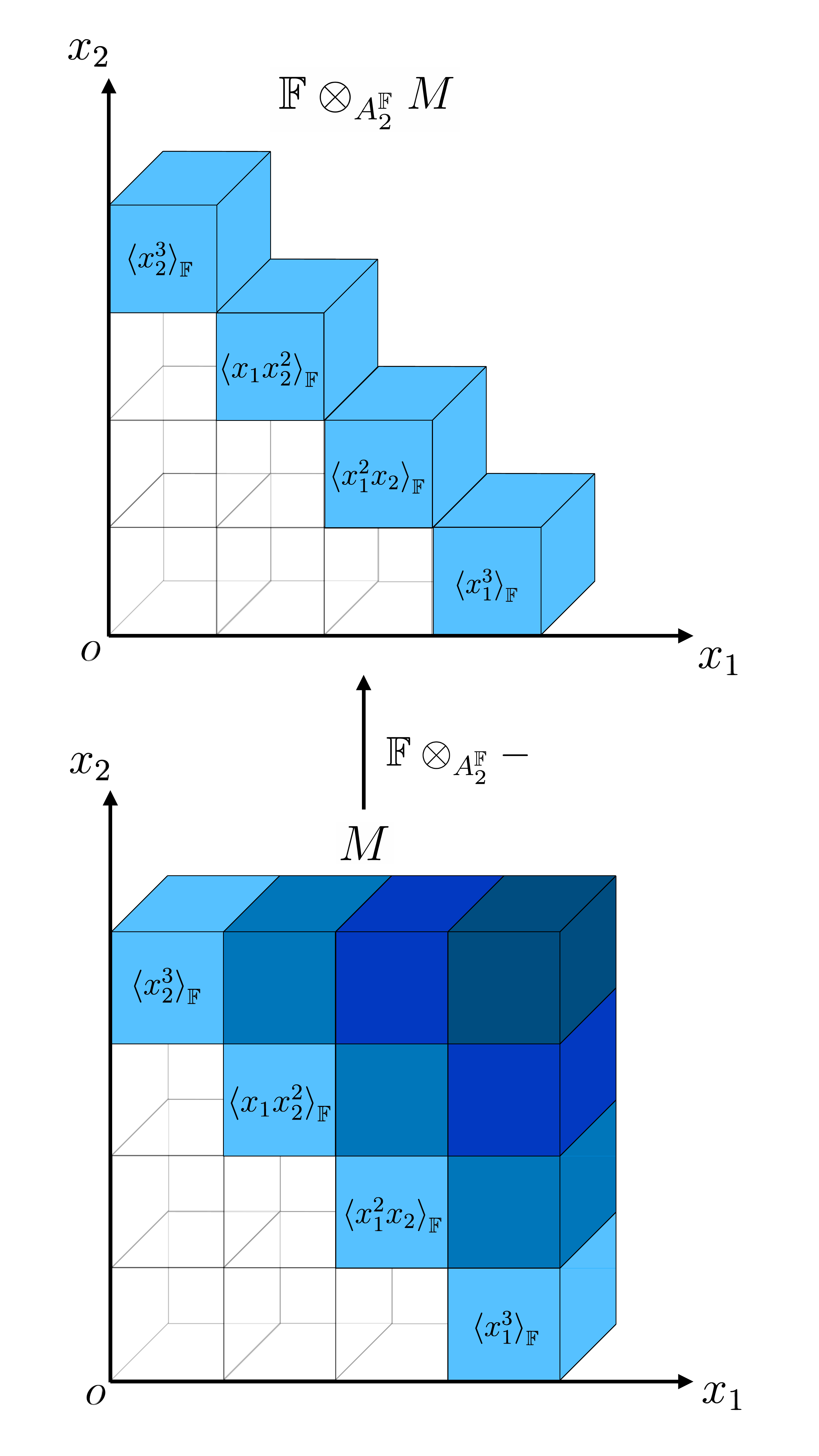} 
\caption{Illustration of \cref{free hulls exa dim 2}.}
\label{fig:Tensoroperation bigraded}
\end{figure}

Now consider the two-dimensional multiset
\begin{align}
\xi=(V, \mu)=\{((3,0),1),((2,1),1),((1,2),1),((0,3),1)\} \notag
\end{align}
where $V=\{(3,0),(2,1),(1,2),(0,3)\} \subseteq \mathbb{N}^2$ and where the mulitplicity function $\mu$ is given by ${\mu(3,0)=\mu(2,1)=\mu(1,2)=\mu(0,3)=1}$. 
We have 
\begin{align}
\mathcal{F}_2^{\mathbb{F}}(\xi)=A_2^{\mathbb{F}}(3,0) \oplus A_2^{\mathbb{F}}(2,1) \oplus A_2^{\mathbb{F}}(1,2) \oplus  A_2^{\mathbb{F}}(0,3) \notag
\end{align}
and
\begin{align}
p:\mathcal{F}_2^{\mathbb{F}}(\xi)& \longrightarrow M, \notag \\
(1,0,0,0) &\longmapsto x_1^3, \notag \\
(0,1,0,0) &\longmapsto x_1^2 x_2, \notag \\
(0,0,1,0) &\longmapsto x_1 x_2^2, \notag \\
(0,0,0,1) &\longmapsto x_2^3 \notag
\end{align}
defines a surjective graded $A_2^{\mathbb{F}}$-module homomorphism. We claim that $(\mathcal{F}_2^{\mathbb{F}}(\xi),p)$ is a free hull of $M$: 

since $p$ is surjective and tensoring is right exact,  
\begin{equation}
\mathrm{id}_{\mathbb{F}} \otimes_{A_2^{\mathbb{F}}} p: \mathbb{F} \otimes_{A_2^{\mathbb{F}}}\mathcal{F}_2^{\mathbb{F}}(\xi) \longrightarrow \mathbb{F} \otimes_{A_2^{\mathbb{F}}} M \notag
\end{equation}
is a surjective graded $\mathbb{F}$-vector space homomorphism. If we want to show that $(\mathcal{F}_2^{\mathbb{F}}(\xi),p)$ is a free hull of $M$ we have to show that $\mathrm{id}_{\mathbb{F}} \otimes_{A_2^{\mathbb{F}}} p$ is an isomorphism. Since $\mathrm{id}_{\mathbb{F}} \otimes_{A_2^{\mathbb{F}}} p$ is graded and surjective, $\mathrm{id}_{\mathbb{F}} \otimes_{A_2^{\mathbb{F}}} p$ induces surjective $\mathbb{F}$-vector space homomorphisms
\begin{equation}
(\mathrm{id}_{\mathbb{F}} \otimes_{A_2^{\mathbb{F}}} p)_v: \left(\mathbb{F} \otimes_{A_2^{\mathbb{F}}}\mathcal{F}_2^{\mathbb{F}}(\xi) \right)_v \longrightarrow \left(\mathbb{F} \otimes_{A_2^{\mathbb{F}}} M \right)_v \notag
\end{equation}  
for all $v \in \mathbb{N}^2$. So, it suffices to show that $(\mathrm{id}_{\mathbb{F}} \otimes_{A_2^{\mathbb{F}}} p)_v$ is an isomorphism for all $v \in \mathbb{N}^2$. For this, it suffices to show that 
\begin{equation}
\mathrm{dim}_{\mathbb{F}}\left( \left(\mathbb{F} \otimes_{A_2^{\mathbb{F}}}\mathcal{F}_2^{\mathbb{F}}(\xi) \right)_v \right) =\mathrm{dim}_{\mathbb{F}}\left(\left(\mathbb{F} \otimes_{A_2^{\mathbb{F}}} M \right)_v \right) \notag
\end{equation}  
for all $v \in \mathbb{N}^2$.

Since $\mathbb{F} \otimes_{A_2^{\mathbb{F}}}\mathcal{F}_2^{\mathbb{F}}(\xi) \cong \mathcal{V}_2^{\mathbb{F}}(\xi)$ as bigraded $\mathbb{F}$-vector spaces and since $(\mathrm{id}_{\mathbb{F}} \otimes_{A_2^{\mathbb{F}}} p)_v$ is surjective, we obtain 
\begin{equation}\label{important equation}
0 \leq \mathrm{dim}_{\mathbb{F}} \left( \left(\mathbb{F} \otimes_{A_2^{\mathbb{F}}} M \right)_v \right) \leq \mathrm{dim}_{\mathbb{F}}\left( \left(\mathbb{F} \otimes_{A_2^{\mathbb{F}}}\mathcal{F}_2^{\mathbb{F}}(\xi) \right)_v \right)= \mathrm{dim}_{\mathbb{F}} \left( \mathcal{V}_2^{\mathbb{F}}(\xi)_v \right)
\end{equation}
for all $v \in \mathbb{N}^2$. Since $\mathrm{dim}_{\mathbb{F}} \left( \mathcal{V}_2^{\mathbb{F}}(\xi)_v \right)=0$ for all $v \in \mathbb{N}^2 \setminus V$, (\ref{important equation}) shows that 
\begin{equation}
\mathrm{dim}_{\mathbb{F}} \left( \left(\mathbb{F} \otimes_{A_2^{\mathbb{F}}} M \right)_v \right)=\mathrm{dim}_{\mathbb{F}}\left( \left(\mathbb{F} \otimes_{A_2^{\mathbb{F}}}\mathcal{F}_2^{\mathbb{F}}(\xi) \right)_v \right)  = \mathrm{dim}_{\mathbb{F}} \left( \mathcal{V}_2^{\mathbb{F}}(\xi)_v \right)=0 \notag
\end{equation}
for all $v \in \mathbb{N}^2 \setminus V$. 

Let $v \in V$. We have \begin{equation}\mathrm{dim}_{\mathbb{F}}\left( \left(\mathbb{F} \otimes_{A_2^{\mathbb{F}}}\mathcal{F}_2^{\mathbb{F}}(\xi) \right)_v \right) =\mathrm{dim}_{\mathbb{F}} \left( \mathcal{V}_2^{\mathbb{F}}(\xi)_v \right)=1 \notag.
\end{equation}
So, we have to show that 
\begin{equation}
\mathrm{dim}_{\mathbb{F}} \left( \left(\mathbb{F} \otimes_{A_2^{\mathbb{F}}} M \right)_v \right)=1 \notag.
\end{equation}
(\ref{important equation}) implies that 
\begin{equation}
0 \leq \mathrm{dim}_{\mathbb{F}} \left( \left(\mathbb{F} \otimes_{A_2^{\mathbb{F}}} M \right)_v \right)  \leq 1 \notag.
\end{equation}
So, it suffices to show that $\left(\mathbb{F} \otimes_{A_2^{\mathbb{F}}} M \right)_v \neq 0$. We have $M_v= \braket{x^v}_{\mathbb{F}}$. Consider \begin{equation}
1 \otimes x^{v} \in \left(\mathbb{F} \otimes_{A_2^{\mathbb{F}}} M \right)_v \notag.
\end{equation}
Then $1 \otimes x^v \neq 0$. This is due to the fact that under the canonical graded isomorphism 
\begin{equation}
\mathbb{F} \otimes_{A_2^{\mathbb{F}}} M \xlongrightarrow{\sim} M/\mathfrak{m}_2^{\mathbb{F}} M, \quad a \otimes m \longmapsto a \cdot \pi(m) \notag
\end{equation}
the image of $1 \otimes x^v$ in $M/\mathfrak{m}_2^{\mathbb{F}} M$ is not zero.
\end{Exa}

As \cref{free hulls exa dim 2} indicates, a free hull of $M \in \mathbf{Grf}_n(A_n^{\mathbb{F}})$ is given as follows:
choose a minimal set of homogeneous generators $G\subseteq M$. This determines a finite $n$-dimensional multiset $\xi$ and a surjective graded morphism $p:\mathcal{F}_n^{\mathbb{F}}(\xi) \to M$. The multiset $\xi$ tells us in which degrees the generators in $G$ are located. Throughout Sections \ref{n-graded naka sec}-\ref{uniqueness sec}, we will see that this approach is the right one. We close this section with the next definition. Since free hulls are unique up to isomorphism by \cref{Free hull thm}, the next definition makes sense.

\begin{Defi}[Type] Let $M \in \mathbf{Grf}_n(A_n^{\mathbb{F}})$ and let $p: \mathcal{F}_n^{\mathbb{F}}(\xi)\to M$ be a free hull of $M$. Then 
\begin{equation}
\Xi_n^{\mathbb{F}}(M):=\xi \notag
\end{equation}
is called the \textit{type of} $M$. Note that 
\begin{equation}
\Xi_n^{\mathbb{F}}(M)=\Xi_n^{\mathbb{F}}(\mathbb{F} \otimes_{A_n^{\mathbb{F}}} M) \notag.
\end{equation}
\end{Defi}

\subsubsection{The \textit{n}-graded Nakayama Lemma}\label{n-graded naka sec}

The following $n$-graded version of \textit{Nakayama's Lemma} can be found in \cite[Ex.~7.8]{miller2004combinatorial}. Note that in \cite[Ex.~7.8]{miller2004combinatorial}, the setting is more general.
\begin{Lem}[{{\cite[Ex.~7.8]{miller2004combinatorial}, Nakayama}}]\label{graded Naka} Let $M \in \mathbf{Grf}_n(A_n^{\mathbb{F}})$, $m_1, \dots,m_d \in M$ be homogeneous elements and denote by $\pi: M \to M/\mathfrak{m}_n^{\mathbb{F}} M$ the canonical projection. Note that $\pi$ is graded. \begin{enumerate}
\item $m_1, \dots,m_d $ generate $M$ if and only if $\pi(m_1), \dots,\pi(m_d) $ generate the finitely generated $n$-graded $\mathbb{F}$-vector space $M/ \mathfrak{m}_n^{\mathbb{F}} M$. 
\item $\{m_1, \dots,m_d\} $ is a minimal generating set of $M$ if and only if $\{\pi(m_1), \dots,\pi(m_d) \}$ is an $\mathbb{F}$-basis of $M/ \mathfrak{m}_n^{\mathbb{F}} M$.
\item Let $G \subseteq M$ be a minimal set of homogeneous generators (such a $G$ always exists and is finite by \cref{Homog generators}). If $C \subseteq M$ is another minimal set of homogeneous generators, then $\abs{C}=\abs{G}$.
\end{enumerate}
\end{Lem}
\begin{proof} The idea of the proof is similar to the $\mathbb{N}$-graded setting (see \cite[Lem.~1.4]{eisenbud2005geometry}). We clearly have that if $m_1, \dots, m_d$ generate $M$, then $\pi(m_1), \dots, \pi(m_d)$ generate $M/\mathfrak{m}_n^{\mathbb{F}} M$. 

So let us prove the converse: Assume that  $\pi(m_1), \dots, \pi(m_d)$ generate $M/\mathfrak{m}_n^{\mathbb{F}}M$. Let ${\overline{M}:=M/ \braket{m_1, \dots, m_d}_{A_n^{\mathbb{F}}}}$. Since ${\braket{m_1, \dots, m_d}_{A_n^{\mathbb{F}}} \subseteq M}$ is graded, $\overline{M}$ is finitely generated and $n$-graded (as it is generated by finitely many homogeneous elements (see \cref{basics about graded})). Since $\pi(m_1), \dots, \pi(m_d)$ generate $M/\mathfrak{m}_n^{\mathbb{F}}M$, $\overline{M}/\mathfrak{m}_n^{\mathbb{F}} \overline{M}=0$. Hence, $\overline{M}= \mathfrak{m}_n^{\mathbb{F}} \overline{M}$. Now $\overline{M}$ has a non-empty finite set of homogeneous generators \begin{equation}
\bigcup_{v \in V} \{z_{1,v}, \dots, z_{d_{v},v} \} \subseteq \overline{M} \notag
\end{equation}
for a suitable finite $V \subseteq \mathbb{N}^n$ with $z_{i,v} \in \overline{M}_{v}$ for all $v \in V$ and all $i \in \{1, \dots, d_v\}$. The goal is to show that $z_{i,u}=0$ for all $u \in V$ and all $i \in \{1, \dots, d_u\}$. 

Let $Z_1 \subseteq V$ be the set of all $v \in V$ that are minimal with respect to the partial order $\preceq$. Then $\left(\mathfrak{m}_n^{\mathbb{F}} \overline{M}\right)_w=\overline{M}_w=0$ for all $w \in {\mathbb{N}^n} \setminus \bigcup_{v \in V} \mathbb{N}^n_{\succeq v}$. Let $u \in Z_1$ and $j \in \{1, \dots, d_u\}$. Since $\overline{M}= \mathfrak{m}_n^{\mathbb{F}} \overline{M}$, we can write \begin{equation}
z_{j,u}= \sum_{v \in V} \sum_{i=1}^{d_v} s^{(i,v)} r^{(i,v)} z_{i,v} \notag
\end{equation}
where $s^{(i,v)} \in A_n^{\mathbb{F}}$ and $r^{(i,v)} \in \mathfrak{m}_n^{\mathbb{F}}$. Since $\mathfrak{m}_n^{\mathbb{F}}=\bigoplus_{v \in {\mathbb{N}^n_{\succ 0}}} (A_n^{\mathbb{F}})_v$, we have 
\begin{equation}
s^{(i,v)} r^{(i,v)} z_{i,v}= \left(\sum_{ w \in \mathbb{N}^n} s^{(i,v)}_w \right)\left(\sum_{ w \in {\mathbb{N}}_{ \succ 0}^n} r^{(i,v)}_w \right) z_{i,v}=\sum_{ u\in {\mathbb{N}^n}} \sum_{ w \succ 0, y+w=u} s^{(i,v)}_y r^{(i,v)}_w z_{i,v} \notag
\end{equation}
for unique $s^{(i,v)}_w, r^{(i,v)}_w \in( A_n^{\mathbb{F}})_w$ with $s^{(i,v)}_w=0$ for all but finitely many $w \in \mathbb{N}^n$ and $r^{(i,v)}_w=0$ for all but finitely many $w \in {\mathbb{N}^n_{\succ 0}}$. If $s^{(i,v)} r^{(i,v)} z_{i,v} \neq 0$, then all nonzero homogeneous parts of $s^{(i,v)} r^{(i,v)} z_{i,v}$ have degree $\succ u$. Since $u$ is minimal, this implies that $z_{j,u}=0$. Therefore, \begin{equation}
\bigcup_{v \in V} \{z_{1,v}, \dots, z_{d_{v},v} \}\setminus \bigcup_{v \in Z_1} \{z_{1,v}, \dots, z_{d_{v},v} \} \subseteq \overline{M} \notag
\end{equation}
is a set of homogeneous generators of $\overline{M}$. Now repeat the same procedure for \begin{equation}
\bigcup_{v \in V} \{z_{1,v}, \dots, z_{d_{v},v} \}\setminus \bigcup_{v \in Z_1} \{z_{1,v}, \dots, z_{d_{v},v} \} \subseteq \overline{M} \notag
\end{equation}
Then $z_{i,v}=0$ for all $v \in Z_1 \cup Z_2$ and all $i \in \{1, \dots, d_v\}$, where $Z_2$ denotes the set of all $v\in V \setminus  Z_1$ that are minimal with respect to $\preceq$. Now proceed inductively. Since $V$ is finite, there exists an $m \in \mathbb{N}$ such that $ V = \bigcup_{i=1}^m Z_i$ and we conclude that $z_{i,v}=0$ for all $v \in V$ and all $i \in \{1, \dots, d_{v}\}$. So $\overline{M} =0$, which proves that $m_1, \dots, m_d$ generate $M$.

For the second part, let us first assume that $\{m_1, \dots, m_d\}\subseteq M$ is minimal and that ${\{\pi(m_1), \dots, \pi(m_d)\}\subseteq M/\mathfrak{m}_n^{\mathbb{F}} M}$ is not minimal. Then there exists an $i \in \{1, \dots, d\}$ such that $\{\pi(m_1), \dots, \pi(m_d)\} \setminus \{\pi(m_i)\}$ is a generating set of $M/\mathfrak{m}_n^{\mathbb{F}} M$. By using the first part of this lemma, this implies that $\{m_1, \dots, m_d \} \setminus \{m_i\}$ is a generating set of $M$. But this is a contradiction to the minimality of $m_1, \dots, m_d$. Thus, $\{\pi(m_1), \dots, \pi(m_d)\}$ is a minimal generating set and hence an $\mathbb{F}$-basis of $M/\mathfrak{m}_n^{\mathbb{F}} M$.

$\{\pi(m_1), \dots, \pi(m_d)\}$ is an $\mathbb{F}$-basis of $M/\mathfrak{m}_n^{\mathbb{F}}$ if and only if $\{\pi(m_1), \dots, \pi(m_d)\}$ is a minimal generating set. As before, if $\{m_1, \dots, m_d\}$ were not minimal this would be a contradiction to the minimality of $\{\pi(m_1), \dots, \pi(m_d)\}$.

The third part is clear by the second part since all minimal generating sets (or equivalently $\mathbb{F}$-bases) of $M/\mathfrak{m}_n^{\mathbb{F}}$ are finite and of the same length.
\end{proof}

\newpage
\subsubsection{Existence of free hulls}

The next Lemma is a standard result about Noetherian modules. 
\begin{Lem}\label{Surj Noeth} Let $R$ be a commutative Noetherian ring and let $N$ be a Noetherian $R$-module. Let $f: N \to N$ be a surjective $R$-module endomorphism. Then $f$ is an automorphism.
\end{Lem}
\begin{proof} We have to show that $f$ is injective. We have $\mathrm{ker}(f^r) \subseteq \mathrm{ker}(f^{r+1})$ for all $r \in \mathrm{N}$. Since $N$ is Noetherian there exists an $m \in \mathbb{N}$ such that $\mathrm{ker}(f^m)=\mathrm{ker}(f^{m+r})$ for all $r \geq  m$. As $f$  is surjective, we have $\mathrm{im}(f^n)=N$. Now we show that $\mathrm{im}(f^m) \cap \mathrm{ker}(f^m)=0$. Let $x \in \mathrm{im}(f^m) \cap \mathrm{ker}(f^m)$. Then $f^m(x)=0$ and $f^m(y)=x$ for some $y \in N$ since $f$ is surjective. Hence, $f^{2m}(y)=0$. Since $\mathrm{ker}(f^m)=\mathrm{ker}(f^{2m})$, we have $x= f^{m}(y)=0$.
\end{proof}

\begin{Prop}\label{minimal free graded basis} Let $F \in \mathbf{Grf}_n(A_n^{\mathbb{F}})$ be free as an $A_n^{\mathbb{F}}$-module and let $G \subseteq F$ be a minimal set of homogeneous generators of $F$ (such a $G$ always exists and is finite by \cref{Homog generators}). Then $G$ is a homogeneous basis of $F$.
\end{Prop}
\begin{proof} $A_n^{\mathbb{F}}$ is commutative and $F$ is free and finitely generated. So, $\mathrm{rank}(F)$ is finite and well-defined (see \cref{Rank well definde}). Suppose that $\mathrm{rank}(F)=d$ and let $B=\{b_1, \dots, b_d\}$ be an $A_n^{\mathbb{F}}$-basis of $F$. We have $\mathrm{dim}_{\mathbb{F}} (\mathbb{F} \otimes_{A_n^{\mathbb{F}}} F)=d$ and
\begin{equation}
\mathbb{F} \otimes_{A_n^{\mathbb{F}}} F = A_n^{\mathbb{F}}/\mathfrak{m}^{\mathbb{F}}_n \otimes_{A_n^{\mathbb{F}}} F \cong F/\mathfrak{m}^{\mathbb{F}}_n F \notag
\end{equation}
which shows that $\mathrm{dim}_{\mathbb{F}} (F/\mathfrak{m}^{\mathbb{F}}_n F)=d$. Thus, $\abs{G}=d$ by \cref{graded Naka}. So, there exists a set-theoretic bijection $\tau: B \to G$. Now $\tau$ induces a surjective $A_n^{\mathbb{F}}$-module endomorphism 
\begin{equation}
f_{\tau}: F \longrightarrow F, \quad \sum_{i=1}^d a_i b_i \longmapsto \sum_{i=1}^d a_i \tau(b_i) \notag.
\end{equation}  
Since $A_n^{\mathbb{F}}$ is Noetherian and $F$ is finitely generated, $F$ is a Noetherian $A_n^{\mathbb{F}}$-module. Thus, $f_{\tau}$ has to be an automorphism by \cref{Surj Noeth}. Hence, $G$ is an $A_n^{\mathbb{F}}$-basis of $F$. $G$ consists only of homogeneous elements by assumption. So, $G$ is a homogeneous $A_n^{\mathbb{F}}$-basis of $F$.
\end{proof}

\begin{Prop}\label{graded free mult} For $M \in \mathbf{Grf}_n(A_n^{\mathbb{F}})$, the following statements are equivalent:
\begin{enumerate}
\item $M$ is graded free as an $A_n^{\mathbb{F}}$-module.
\item $M$ is free as an $A_n^{\mathbb{F}}$-module.
\end{enumerate}
\end{Prop}

\begin{proof} $(1. \Longrightarrow 2.)$ is clear and $(2. \Longrightarrow 1.)$ follows from \cref{minimal free graded basis}.
\end{proof}

\newpage

\begin{Prop}\label{equiv char free hulls} Let $F,M \in \mathbf{Grf}_n(A_n^{\mathbb{F}})$ with $F$ graded free. Let $p: F \to M$ be a surjective graded morphism. The following statements are equivalent:
\begin{enumerate}
\item $p:F \to M$ is a free hull of $M$.
\item for any homogeneous basis $B$ of $F$, $p(B)$ is a minimal set of homogeneous generators of $M$ with $\abs{p(B)}=\abs{B}$.
\item there exists a homogeneous basis $B$ of $F$ such that $p(B)$ is a minimal set of homogeneous generators of $M$ with $\abs{p(B)}=\abs{B}$.
\end{enumerate}
\end{Prop}
\begin{proof} Since $\mathbb{F} \otimes_{A_n^{\mathbb{F}}}-$ is right exact, we obtain a surjective graded morphism
\begin{align}
\mathrm{id}_\mathbb{F} \otimes_{A_n^{\mathbb{F}}} p: \mathbb{F} \otimes_{A_n^{\mathbb{F}}} F \longrightarrow \mathbb{F} \otimes_{A_n^{\mathbb{F}}} M \notag.
\end{align}
Now we obtain a commutative diagram in $\mathbf{Grf}_n(\mathbb{F})$
\begin{equation}\label{free equ char diag}
\begin{tikzcd}[row sep=1cm, column sep = 1.5cm]
\mathbb{F} \otimes_{A_n^{\mathbb{F}}} F \arrow[d,equals] \arrow[r,"\mathrm{id}_\mathbb{F} \otimes_{A_n^{\mathbb{F}}} p"] & \mathbb{F} \otimes_{A_n^{\mathbb{F}}} M \arrow[d, equals] \\
A_n^{\mathbb{F}}/\mathfrak{m}^{\mathbb{F}}_n \otimes_{A_n^{\mathbb{F}}} F \arrow[r] \arrow[d] & A_n^{\mathbb{F}}/\mathfrak{m}^{\mathbb{F}}_n \otimes_{A_n^{\mathbb{F}}} M \arrow[d]\\
F/\mathfrak{m}^{\mathbb{F}}_n F\arrow[r]  &  M /\mathfrak{m}^{\mathbb{F}}_n M 
\end{tikzcd} 
\end{equation}
where the vertical maps are all graded isomorphisms and the horizontal maps are all graded surjective  morphisms. If we identify $\mathbb{F} \otimes_{A_n^{\mathbb{F}}} F= F/F\mathfrak{m}_n^{\mathbb{F}}$ and ${\mathbb{F} \otimes_{A_n^{\mathbb{F}}} M= M/ M\mathfrak{m}_n^{\mathbb{F}}}$ along the vertical graded ismorphisms in (\ref{free equ char diag}), we obtain a commutative diagram
\begin{equation}
\begin{tikzcd}[row sep=1cm, column sep = 1.5cm]
F \arrow[d, "\pi_F"]\arrow[r,"p"] & M \arrow[d, "\pi_M"] \\
\mathbb{F} \otimes_{A_n^{\mathbb{F}}} F  \arrow[r, "\mathrm{id}_{\mathbb{F}} \otimes_{A_n^{\mathbb{F}}} p"] & \mathbb{F} \otimes_{A_n^{\mathbb{F}}} M
\end{tikzcd}  \notag
\end{equation}
where $\pi_F: F \to  F/\mathfrak{m}_n^{\mathbb{F}} F=\mathbb{F} \otimes_{A_n^{\mathbb{F}}} F$ and $\pi_M: M \to  M/\mathfrak{m}_n^{\mathbb{F}} M=\mathbb{F} \otimes_{A_n^{\mathbb{F}}} M$ denotes the canonical projections.

$(1. \Longrightarrow 2.)$: Since $p:F \to M$ is a free hull, we know that $\mathrm{id} \otimes_{A_n^{\mathbb{F}}} p$ has to be a graded isomorphism. Using the commutativity of (\ref{free equ char diag}), this implies that $F/\mathfrak{m}_n^{\mathbb{F}} \to M/\mathfrak{m}_n^{\mathbb{F}} M$ is a graded isomorphism. Now let $B$ be a homogeneous $A_n^{\mathbb{F}}$-basis of $F$. Then $B$ is finite and $\pi_F(B)$ is a homogenenous $\mathbb{F}$-basis of $\mathbb{F} \otimes_{A_n^{\mathbb{F}}} F$ by \cref{graded Naka} and $\abs{\pi_F(B)}=\abs{B}$. Hence, $C:=(\mathrm{id}_{\mathbb{F}} \otimes_{A_n^{\mathbb{F}}} p)(\pi_F(B))$ is a homogeneous $\mathbb{F}$-basis of $ \mathbb{F} \otimes_{A_n^{\mathbb{F}}} M$. Since $p$ is surjective and graded, $p(B)\subseteq M$ is a set of homogeneous generators of $M$. Since $\pi_M(p(B))=C$, we conclude that
\begin{equation}
\abs{B} \geq \abs{p(B)} \geq  \abs{\pi_M(p(B))} \geq \abs{C}=\abs{(\mathrm{id}_{\mathbb{F}} \otimes_{A_n^{\mathbb{F}}} p)(\pi_F(B))}= \abs{\pi_F(B)}=\abs{B} \notag.
\end{equation}
Hence, $\abs{p(B)}=\abs{\pi_M(p(B))}=\abs{C}=\abs{B}$. By using \cref{graded Naka}, this implies that $p(B)$ is minimal. 

$(2. \Longrightarrow 1.)$: Let $B \subseteq F$ be a homogeneous basis of $F$ such that $p(B)$ is a minimal generating system of $M$ with $\abs{p(B)}=\abs{B}$ ($B$ and $p(B)$ are finite). By \cref{graded Naka}, $\pi_M(p(B))$ is an $\mathbb{F}$-basis of $M/\mathfrak{m}^{\mathbb{F}}_n M$ with $\abs{\pi_M(p(B))}=\abs{p(B)}$ and $\pi_F(B)$ is an $\mathbb{F}$-basis of $F/\mathfrak{m}^{\mathbb{F}}_n F$ with $\abs{\pi_F(B)}=\abs{B}$. Therefore, \begin{equation}
\abs{\pi_F(B)}=\abs{B}=\abs{p(B)}=\abs{\pi_M(p(B))} \notag 
\end{equation} which shows that the horizontal maps in (\ref{free equ char diag}) are graded isomorphisms.

$F$ has a homogenous basis $B$ since $F$ is graded free which shows that $(2. \Longrightarrow 3.)$. $(3. \Longrightarrow 1.)$: Assume that $F$ has a homogenous basis $B \subseteq F$ such that $p(B)$ is a minimal set of homogeneous generators of $M$ with $\abs{B}=\abs{p(B)}$. Now the proof is analogous to $(2. \Longrightarrow 1.)$.
\end{proof}

We are now ready to prove the existence of free hulls:

\begin{proof}[Proof of \ref{Free hull thm}, existence]
Let $G=\{m_1, \dots, m_d\}$ be a minimal set of homogenous generators of $M$. Then we can subdivide $G=\bigcup_{v\in T} \{m_{1,v}, \dots, m_{d_v,v}\}$ for a suitable subset $T \subseteq {\mathbb{N}^n}$ and suitable $d_v \in \mathbb{N}$ such that $m_{i,v} \in M_v$ for all $i \in \{1, \dots, d_v \}$ and all $v \in T$. Now define $\xi:=\bigcup_{v \in T}\{(v,1), \dots, (v, d_v) \}$. So, $\xi=(T, \mu)$ defines an $n$-dimensional multiset where \begin{equation}
\mu: T \longrightarrow  \mathbb{N}_{\geq 1}, \quad v \longmapsto d_v \notag.
\end{equation}
This leads to a surjective graded morphism
\begin{equation}
p: \mathcal{F}_n^{\mathbb{F}}(\xi) \longrightarrow M, \quad e_{i,v} \longmapsto m_{i, v} \notag
\end{equation} 
where $e_{i,v}$ is the $i$-th standard basis vector in $A_n^{\mathbb{F}}(v)^{\mu(v)}$. Now $p:\mathcal{F}_n^{\mathbb{F}}(\xi) \to M$ is a free hull by \cref{equiv char free hulls}. In particular, if $M$ is graded free, then $p$ is a graded isomorphism (see proof of \cref{graded free isom multiset}).
\end{proof}

\newpage

\subsubsection{Uniqueness of free hulls}\label{uniqueness sec}
Before we proof the uniqueness of free hulls, we need the following:

\begin{Prop}\label{graded free projective} Let $R$ be an $n$-graded commutative ring and let $M,N,P$ be $n$-graded $R$-modules. 
\begin{enumerate}
\item Let $g:M \to N$ and $f:P \to N$ be graded $R$-module homomorphisms. If $h: P \to M$ is an $R$-module homomorphism such that $f=g \circ h$, then there exists a graded $R$-module homomorphism $\psi: P \to M$ such that $f=g \circ \psi$.
\item If $P$ is projecive, then $P$ is graded projective, i.e. if $g:M \to N$ is a surjective graded $R$-module homomorphism and if $f:P \to N$ is a graded $R$-module homomorphism, then there exists a graded $R$-module homomorphism ${\psi: P \to M}$ such that  $f=g \circ \psi$.
\item If $P$ is graded free, then $P$ is graded projective.
\end{enumerate}
\end{Prop}
\begin{proof} Let us start with the proof of the first part: for this, let $y \in P$. Then $\psi(y) \in M$ is constructed as follows: we have $y= \sum_{v \in {\mathbb{N}^n}} y_v$ for unique  $y_v \in P_v$ with $y_v=0$ for all but finitely many $v \in {\mathbb{N}^n}$. Let $v \in {\mathbb{N}^n}$. Since $f$ is graded, we have $f(y_v) \in N_v$. It holds that $h(y_v)= \sum_{w \in {\mathbb{N}^n}} z_w$ for unique $z_w \in M_w$ with $z_w=0$ for all but finitely many $w \in {\mathbb{N}^n}$. Now we have 
\begin{equation}f(y_v)= g(h(y_v))= \sum_{w \in {\mathbb{N}^n}} g(z_w) \in N_v \notag
\end{equation}
which implies that (since $g$ is graded) $g(z_w)=0$ for all $v \neq w$. So, we have $f(y_v)=g(z_v)$. Now define $\psi(y_v):=z_v$ and \begin{equation}
\psi(y):=\sum_{v \in {\mathbb{N}^n}} \psi(y_v)=\sum_{v \in {\mathbb{N}^n}} z_v \notag.
\end{equation}
This defines a set theoretic map $\psi: P \to M$. We have
\begin{align}\label{all sums fin eq2}
f(y)&=f \left( \sum_{v \in {\mathbb{N}^n}} y_v \right)= \sum_{v \in {\mathbb{N}^n}} f(y_v) =\sum_{v \in {\mathbb{N}^n}} g(z_v)
\notag \\
&=\sum_{v \in {\mathbb{N}^n}} g(\psi(y_v))=g \left(\sum_{v \in {\mathbb{N}^n}} \psi(y_v) \right)=g\left( \psi \left(\sum_{v \in {\mathbb{N}^n}} y_v \right) \right)\overset{\textrm{Def. of } \psi}{=} g(\psi(y)) .
\end{align}
where all sums in (\ref{all sums fin eq2}) are finite. It remains to show that $\psi$ is a graded $R$-module homomorphism. $\psi$ is $R$-linear: let $r \in R$ and $y,y' \in P$. We have
\begin{equation}\label{all sums fin eq}
ry+y'= \sum_{v \in {\mathbb{N}^n}} r y_v + \sum_{w \in {\mathbb{N}^n}} y'_w = \sum_{w \in {\mathbb{N}^n}} \sum_{v \in {\mathbb{N}^n}} r_v y_w + \sum_{w \in {\mathbb{N}^n}} y'_w =\sum_{w \in {\mathbb{N}^n}} \sum_{v \in {\mathbb{N}^n}} (r_v y_w + y'_w)
\end{equation}
where all sums in (\ref{all sums fin eq}) are finite. We claim that if $\psi$ maps $y_w$ to $z_w$ and $y'_w$ to $z'_w$, then $\psi$ maps $r_v y_w+y'_w$ to $r_v z_w+z'_w$. To prove this claim, we go through the construction of $\psi$ step by step: We have $f(r_v y_w)=r_v f(y_w) \in N_{v+w}$ and 
\begin{equation}\label{sums finite5}
h(r_v y_w)=r_v h(y_w)= r_v \sum_{w \in {\mathbb{N}^n}} z_w=\sum_{w \in {\mathbb{N}^n}} r_v z_w
\end{equation}
where all sums in (\ref{sums finite5}) are finite. Now we claim that $\psi(r_v y_w)=r_v\psi(y_w)=r_v z_w$. We have $f(r_v y_w+y'_w)=r_v f(y_w)+f(y'_w) \in N_{v+w}$ and
\begin{equation}\label{all sums fin3}
h(r_v y_w+y'_w)=r_v h(y_w)+h(y'_w)= r_v \sum_{u \in {\mathbb{N}^n}} z_u+ \sum_{u \in {\mathbb{N}^n}} z'_u= \sum_{u \in {\mathbb{N}^n}} (r_v z_u + z'_u)
\end{equation} 
where all sums in (\ref{all sums fin3}) are finite. Hence,
\begin{equation} f(r_v y_w+y'_w)= g(h(r_v y_w+y'_w))= \sum_{u \in {\mathbb{N}^n}} g(r_v z_u+z'_u) \in N_{v+w} \notag
\end{equation}
which shows that (since $g$ is graded) $g(r_v z_u+z'_u)=0$ for all $u \neq w$. So, if $\psi$ maps $y_w$ to $z_w$ and $y'_w$ to $z'_w$, then $\psi$ maps $r_v y_w+y'_w$ to $r_v z_w+z'_w$ which shows that
\begin{align}\label{all sums fin4}
\psi(ry+y')&= \psi \left( \sum_{w \in {\mathbb{N}^n}} \sum_{v \in {\mathbb{N}^n}} r_v y_w+y'_w \right) \overset{\textrm{Def. of } \psi}{=}  \sum_{w \in {\mathbb{N}^n}} \sum_{v \in {\mathbb{N}^n}} \psi(r_v y_w +y'_w ) \notag \\
&=\sum_{w \in {\mathbb{N}^n}} \sum_{v \in {\mathbb{N}^n}} (r_v z_w +z'_w) = r \sum_{w \in {\mathbb{N}^n}}  z_w+ \sum_{w \in {\mathbb{N}^n}} z'_w  \overset{\textrm{Def. of } \psi}{=} r \psi(y)+ \psi(y')
\end{align}
where all sums in (\ref{all sums fin4}) are finite. Hence, $\psi$ is $R$-linear. $\psi$ is graded by construction.

The second part follows immediately from the first part. For the third part, assume that $P$ is graded free. Then $P$ is free as an $R$-module and hence projective. Thus, $P$ is graded projective by $2$.
\end{proof}

The following proposition can be found in \cite[Thm.~6]{Carlsson2009} (the authors do not give a proof).
\begin{Prop}[{{\cite[Thm.~6]{Carlsson2009}}}] \label{reflect isos} Let $F, F' \in \mathbf{Grf}_n(A_n^{\mathbb{F}})$ be graded free and let ${\psi: F \to F'}$ be a graded $A_n^{\mathbb{F}}$-module homomorphism. If 
\begin{equation}
\mathrm{id}_{\mathbb{F}} \otimes_{A_n^{\mathbb{F}}} \psi:\mathbb{F} \otimes_{A_n^{\mathbb{F}}} F \xlongrightarrow{\sim} \mathbb{F} \otimes_{A_n^{\mathbb{F}}} F' \notag
\end{equation}
is an isomorphism of $n$-graded $\mathbb{F}$-vector spaces, then $\psi$ is an isomorphism.
\end{Prop}

\begin{proof} We have a commutative diagram in $\mathbf{Grf}_n(\mathbb{F})$
\begin{equation}
\begin{tikzcd}[row sep=1cm, column sep = 1.5cm]
\mathbb{F} \otimes_{A_n^{\mathbb{F}}} F \arrow[d,equals] \arrow[r, "\mathrm{id}_{\mathbb{F}} \otimes_{A_n^{\mathbb{F}}} \psi"] & \mathbb{F} \otimes_{A_n^{\mathbb{F}}} F' \arrow[d,equals] \\
A_n^{\mathbb{F}}/\mathfrak{m}^{\mathbb{F}}_n \otimes_{A_n^{\mathbb{F}}} F \arrow[r] \arrow[d] & A_n^{\mathbb{F}}/\mathfrak{m}^{\mathbb{F}}_n \otimes_{A_n^{\mathbb{F}}} F' \arrow[d] \\
F/\mathfrak{m}^{\mathbb{F}}_n F \arrow[r]  &  F' /\mathfrak{m}^{\mathbb{F}}_n F' 
\end{tikzcd} \notag
\end{equation}
where all arrows are graded isomorphisms. We can indentify $\mathbb{F} \otimes_{A_n^{\mathbb{F}}} F=F/\mathfrak{m}^{\mathbb{F}}_n F$ as well as $\mathbb{F} \otimes_{A_n^{\mathbb{F}}} F'=F'/\mathfrak{m}^{\mathbb{F}}_n F'$ along the vertical graded isomorphisms.
We denote by ${\pi: F \to F/\mathfrak{m}^{\mathbb{F}}_n F = \mathbb{F} \otimes_{A_n^{\mathbb{F}}} F}$ and $\pi':  F' \to  F'/\mathfrak{m}^{\mathbb{F}}_n F' =\mathbb{F} \otimes_{A_n^{\mathbb{F}}} F'$ the canonical projections. We obtain a commutative diagram
\begin{equation}
\begin{tikzcd}[row sep=1cm, column sep = 1.5cm]
F \arrow[d, "\pi"]\arrow[r,"\psi"] & F' \arrow[d, "\pi'"] \\
\mathbb{F} \otimes_{A_n^{\mathbb{F}}} F  \arrow[r, "\mathrm{id}_{\mathbb{F}} \otimes_{A_n^{\mathbb{F}}} \psi"] & \mathbb{F} \otimes_{A_n^{\mathbb{F}}} F' 
\end{tikzcd} \notag
\end{equation} 
Let $B \subseteq F$ be a homogeneous $A_n^{\mathbb{F}}$-basis. Then $B$ is mapped to a homogeneous $\mathbb{F}$-basis $\pi(B) \subseteq \mathbb{F} \otimes_{A_n^{\mathbb{F}}} F $ by \cref{graded Naka} and we have $\abs{\pi(B)}=\abs{B}$. Hence, 
$C:=(\mathrm{id}_{\mathbb{F}} \otimes_{A_n^{\mathbb{F}}} \psi)(\pi(B)) $
is a homogeneous $\mathbb{F}$-basis of $\mathbb{F} \otimes_{A_n^{\mathbb{F}}} F'$. By \cref{graded Naka}, $\psi(B)$ is a generating set of $F'$ since $\pi'(\psi(B))=C$ is an $\mathbb{F}$-basis of $\mathbb{F} \otimes_{A_n^{\mathbb{F}}} F'$. We have
\begin{equation}
\abs{B}  \geq \abs{\psi(B)}\geq \abs{\pi'(\psi(B))} \geq \abs{C}=\abs{(\mathrm{id}_{\mathbb{F}} \otimes_{A_n^{\mathbb{F}}} \psi)(\pi(B))}= \abs{\pi(B)}=\abs{B} \notag.
\end{equation}
Hence, $\abs{\psi(B)}=\abs{\pi'(\psi(B))}=\abs{C}=\abs{B}$. Using \cref{graded Naka}, this implies that $\psi(B)$ is minimal. Hence, $\psi(B)$ is a homogeneous $A_n^{\mathbb{F}}$-basis of $F'$ by \cref{minimal free graded basis}. Since $\abs{\psi(B)}=\abs{B}$ this shows that $\psi$ is a graded isomorphism.
\end{proof}

Now we are ready to prove the uniqueness of free hulls:

\begin{proof}[Proof of \ref{Free hull thm}, uniqueness] Let $(F,p),(F',p')$ be free hulls of $M$. Then $F$ and $F'$ are graded projective by \cref{graded free projective}. Thus, there exists a graded $A_n^{\mathbb{F}}$-module homomorphism $ {\psi: F \to F'}$ such that $p' \circ \psi =p$. Since $(F,p),(F',p')$ are free hulls of $M$, $\mathrm{id}_{\mathbb{F}} \otimes_{A_n^{\mathbb{F}}} p$ and $\mathrm{id}_{\mathbb{F}} \otimes_{A_n^{\mathbb{F}}} p'$ are isomorphisms of $n$-graded $\mathbb{F}$-vector spaces. Thus, $\mathrm{id}_{\mathbb{F}} \otimes_{A_n^{\mathbb{F}}} \psi=(\mathrm{id}_{\mathbb{F}} \otimes_{A_n^{\mathbb{F}}} p' )^{-1} \circ \mathrm{id}_{\mathbb{F}} \otimes_{A_n^{\mathbb{F}}} p$ is an ismorphism of $n$-graded $\mathbb{F}$-vector spaces. Hence, $\psi$ is a graded isomorphism by \cref{reflect isos}. 
\end{proof}

\subsection{Complete classification}\label{complete classification section}
Now that we have proven the main theorem on free hulls, we are ready to proceed with the complete classification of objects in $\mathbf{Grf}_n(A_n^{\mathbb{F}})/ _{\cong}$. Here $\mathbf{Grf}_n(A_n^{\mathbb{F}})/ _{\cong}$ denotes the class of all isomorphim classes $[M]$ of objects $M \in\mathbf{Grf}_n(A_n^{\mathbb{F}})$. To do this, we first need some preparations.
\subsubsection{A discrete invariant}
\begin{Prop}\label{is a set} $\mathbf{Grf}_n(A_n^{\mathbb{F}})/_{\cong}$ is a set.
\end{Prop}
\begin{proof} For an $n$-dimensional multiset $\xi $, let
\begin{equation}
I_n^{\mathbb{F}}(\xi):=\{ [M] \in \mathbf{Grf}_n(A_n^{\mathbb{F}})/_{\cong} \mid \Xi_n^{\mathbb{F}}(M)=\xi \} \notag
\end{equation}
and 
\begin{equation}
Z_n^{\mathbb{F}}(\xi):=\{ L \subseteq \mathcal{F}_n^{\mathbb{F}}(\xi) \mid L \textrm{ is a graded submodule} \} \notag
\end{equation}
For $[M] \in I(\xi)$, let $p_M: \mathcal{F}_n^{\mathbb{F}}(\xi) \to M$ be a free hull. If ${p_M': \mathcal{F}_n^{\mathbb{F}}(\xi) \to M}$ is another free hull, then, using \cref{Free hull thm}, there exists an automorphism $\psi \in \mathrm{Aut}(\mathcal{F}_n^{\mathbb{F}}(\xi)$ such that the diagram
\begin{equation}
\begin{tikzcd} [row sep=1.5cm, column sep = 1.5cm]
\mathcal{F}_n^{\mathbb{F}}(\xi) \arrow[rd,"p_M"'] \arrow[rr, "\psi", "\sim"'] && \mathcal{F}_n^{\mathbb{F}}(\xi)  \arrow[ld, "p'_M "]   \\
 &M&
\end{tikzcd} \notag
\end{equation}
commutes. Thus, $\mathrm{ker}(p_M)=\mathrm{ker}(p'_M)$. If $N \in [M]$, then there exists a graded isomorphism $\varphi:N \xlongrightarrow{\sim} M$. Let $p_N: \mathcal{F}_n^{\mathbb{F}}(\xi) \to N$ be a free hull. Then $\varphi \circ p_N$ is a free hull of $M$. Now, using \cref{Free hull thm}, there exists an automorphism $\psi \in \mathrm{Aut}(\mathcal{F}_n^{\mathbb{F}}(\xi)$ such that \begin{equation}
\begin{tikzcd}[row sep=1.5cm, column sep = 1.5cm]
      \mathcal{F}_n^{\mathbb{F}} (\xi_0) \arrow[r, "\psi", "\sim"']  \arrow[d,"p_N"]  &  \mathcal{F}_n^{\mathbb{F}}(\xi_0) \arrow[d,"p_M"]  \\
     N \arrow[r, "\varphi",  "\sim"'] &  M
\end{tikzcd} \notag
\end{equation}
commutes. This shows that $\mathrm{ker}(p_N)=\mathrm{ker}(p_M)$. If $[N] \neq [M]$, then we clearly have that $\mathrm{ker}(p_N) \neq \mathrm{ker}(p_M)$.

So, we obtain a well-defined class theoretic injection
\begin{equation}
I_n^{\mathbb{F}}(\xi) \longrightarrow Z_n^{\mathbb{F}}(\xi), \quad [M] \longmapsto \mathrm{ker}(p_M) \notag.
\end{equation}
Since $Z(\xi)$ is a set, $I(\xi)$ has to be a set. Hence, \begin{equation}\mathbf{Grf}_n(A_n^{\mathbb{F}})/_{\cong}= \bigsqcup_{\xi } I(\xi) \notag
\end{equation}
is a set.
\end{proof}

The next definition is as in \cite[Sec.~4.5]{Carlsson2009}. We just have added a some details.
Since free hulls are unique up to isomorphism by \cref{Free hull thm}, the next definition makes sense. 
\begin{Defi}\label{Defi type 0,1} Let $M \in \mathbf{Grf}_n(A_n^{\mathbb{F}})$, let $p_0: \mathcal{F}_n^{\mathbb{F}}(\xi_0) \to M$ be a free hull of $M$ and let $ p_1: \mathcal{F}_n^{\mathbb{F}}(\xi_1) \to \mathrm{ker}(p_0)$ be a free hull of $\mathrm{ker}(p_0)$. Define \begin{equation}\Xi_{n,0}^{\mathbb{F}}(M):=\Xi_{n}^{\mathbb{F}}(M)=\xi_0 \quad \textrm{and} \quad \Xi_{n,1}^{\mathbb{F}}(M):=\Xi_{n}^{\mathbb{F}}(\mathrm{ker}(p_0))=\xi_1 \notag.
\end{equation}
Denote by $\mathcal{T}_{n}$ the set of all pairs $(\xi_0, \xi_1)$ of finite $n$-dimensional multisets. Note that $\mathcal{T}_{n} $ is countable. Let $(\xi_0,\xi_1) \in \mathcal{T}_n$. Define \begin{align}
I_n^{\mathbb{F}}(\xi_0,\xi_1)&:= \{ [M]\in \mathbf{Grf}_n(A_n^{\mathbb{F}}) /_{\cong} \mid {\Xi_{n,0}^{\mathbb{F}}}(M) =\xi_0 \textrm{ and } {\Xi_{n,1}^{\mathbb{F}}}(M) =\xi_1  \} \notag.
\end{align}
Note that
\begin{equation}
\mathbf{Grf}_n(A_n^{\mathbb{F}})/_{\cong}=\bigsqcup_{ (\xi_0,\xi_1) \in \mathcal{T}_{n}} I_n^{\mathbb{F}}(\xi_0,\xi_1) \notag.
\end{equation}
For $\mathbb{F} \in \mathcal{K}$, define
\begin{equation}
 \boldsymbol{\Xi}_{n}^{\mathbb{F}}: \mathbf{Grf}_n(A_n^{\mathbb{F}})/_{\cong} \longrightarrow \mathcal{T}_{n}  , \quad [M] \longmapsto  \left(\Xi_{n,0}^{\mathbb{F}}(M) , \Xi_{n,1}^{\mathbb{F}}(M) \right)\notag.
\end{equation}
Then $\left \{  \boldsymbol{\Xi}_{n}^{\mathbb{F}} \right \}_{\mathbb{F} \in \mathcal{K}}$ defines a discrete class of invariants.
\end{Defi}

$\left \{ \boldsymbol{\Xi}_{n}^{\mathbb{F}} \right \}_{\mathbb{F} \in \mathcal{K}}$ can not expected to be complete. In \cref{Dim 2 Tensor Ex}, the incompleteness is illustrated for $n=2$. In dimension one, $\xi_0=\Xi_{1,0}^{\mathbb{F}}(M)$ and $\xi_1=\Xi_{1,1}^{\mathbb{F}}(M)$ may be viewed as the start and end points of the \textit{barcode} of $M$ (see \cref{One dim pers sec}). The \textit{barcode} is a complete and discrete invariant for $\mathbf{Grf}_n(A_n^{\mathbb{F}})/_{\cong}$ but needs in addition to $(\xi_0, \xi_1)$ information about which start point in $\xi_0$ corresponds to which end point in $\xi_1$ in order to determine the isomorphism class of $M$ completely. This is the reason why $\left \{ \boldsymbol{\Xi}_{1}^{\mathbb{F}} \right \}_{\mathbb{F} \in \mathcal{K}}$ is not complete.

\begin{Rem}\label{min grad res} Let $M\in \mathbf{Grf}_n(A_n^{\mathbb{F} })$. Let $p_0:\mathcal{F}_n^{\mathbb{F}}(\xi_0)\to M$ be a free hull of $M$ (see \cref{Free hull thm}). Let $K_{0}:= \mathrm{ker}(p_0)\in \mathbf{Grf}_n(A_n^{\mathbb{F} })$. Now let $p_1:\mathcal{F}_n^{\mathbb{F}}(\xi_1) \to K_0$ be a free hull of $K_0$.
If we proceed inductively, we obtain a minimal graded free resolution
\begin{equation}\label{min grad reseq}
\begin{tikzcd}
\dots \arrow[r,"p_{i+1}"] & \mathcal{F}_n^{\mathbb{F}}(\xi_{i}) \arrow[r,"p_{i}"]  & \dots \arrow[r,"p_{1}"] &\mathcal{F}_n^{\mathbb{F}}(\xi_0) \arrow[r,"p_0"] & M \arrow[r] & 0
\end{tikzcd}  
\end{equation}
of $M$, where minimal means that $(\mathcal{F}_n^{\mathbb{F}}(\xi_i),p_i)$ is a free hull of $\mathrm{im}(p_i)$. Using (\ref{min grad reseq}), we can assign a sequence of finite $n$-dimensional multisets $(\xi_0, \xi_1, \xi_2, \dots)$ to every $[M]\in \mathbf{Grf}_n(A_n^{\mathbb{F}})/_{\cong}$. Now the question is if (\ref{min grad reseq}) has finite length, where finite means that for some $m \in \mathbb{N}$, $\mathcal{F}_n^{\mathbb{F}}(\xi_i)=0$ for all $i > m$. By \cite[Thm.~3.37] {miller2004combinatorial}, this is true if $M$ is \textit{Cohen-Macaulay}. In this case, a bound of the length is given by $n-r$, where $r$ is the \textit{dimension} of $M$. Note that in the case $n=1$ (recall that $A_1^{\mathbb{F}} =\mathbb{F}[x]$), any minimal graded free resolution of a finitely generated one-graded $\mathbb{F}[x]$-module has length one. This due to the fact that $\mathbb{F}[x]$ is a principal ideal domain and that in this setting, every submodule of a free module is free, which implies that submodules of graded free $\mathbb{F}[x]$-modules are graded free (see \cref{graded free mult}). 
\end{Rem}
The goal for the rest of this section is to give a complete classification of $I_n^{\mathbb{F}}(\xi_0, \xi_1)$, which then leads to a complete classification of $ \mathbf{Grf}_n(A_n^{\mathbb{F}})/_{\cong}$ since 
\begin{equation}
\mathbf{Grf}_n(A_n^{\mathbb{F}})/_{\cong}=\bigsqcup_{ (\xi_0,\xi_1) \in \mathcal{T}_{n}} I_n^{\mathbb{F}}(\xi_0,\xi_1) \notag.
\end{equation}
Now one might ask the question why we do not consider 
\begin{align}
I_n^{\mathbb{F}}(\xi_0, \dots, \xi_i)&:=\{ [M]\in \mathbf{Grf}_n(A_n^{\mathbb{F}}) /_{\cong} \mid {\Xi_{n,j}^{\mathbb{F}}}(M) =\xi_j \, \forall \, j \in \{1, \dots, i\} \} \notag
\end{align}
for $i \in \mathbb{N}$ or
\begin{align}
I_n^{\mathbb{F}}(\xi_0, \xi_1, \xi_2, \dots)&:=\{ [M]\in \mathbf{Grf}_n(A_n^{\mathbb{F}}) /_{\cong} \mid {\Xi_{n,j}^{\mathbb{F}}}(M) =\xi_j \, \forall \, j \in \mathbb{N} \} \notag.
\end{align}
For $n=1$, $I_1^{\mathbb{F}}(\xi_0, \xi_1)$ is already the most accurate approach in terms of minimal graded free resolutions (see \cref{min grad res}). But what is about $i=0$? As we will see in \cref{One dim pers sec}, the space of \textit{barcodes over} $(\xi_0, \xi_1)$ yields a complete and discrete classification for $I_1^{\mathbb{F}}(\xi_0, \xi_1)$. Such a \textit{barcode} needs start and end points as input. If $i=0$ there would be no end points. In \cref{Chapter 2}, we will see that it is possible to obtain a nice parameterization of $I_n^{\mathbb{F}}(\xi_0, \xi_1)$ for all $n \in \mathbb{N}_{\geq 1}$. Nevertheless, it would be possible to accomplish the complete classification, carried out below for $I_n^{\mathbb{F}}(\xi_0,\xi_1)$, for $I_n^{\mathbb{F}}(\xi_0, \dots, \xi_i)$ and $I_n^{\mathbb{F}}(\xi_0, \xi_1, \xi_2, \dots)$ as well.

\subsubsection{The tensor-condition}\label{Tensor condition section}

For the  following, let $\xi_0=(V_0, \mu_0)$ and $\xi_1=(V_1, \mu_1)$ be two $n$-dimensional multisets.
\begin{Defi}[Tensor-condition]\label{defi tensor cond} Let $L \subseteq \mathcal{F}_n^{\mathbb{F}}(\xi_0)$ be a graded $A_n^{\mathbb{F}}$-submodule. We say that $L$ satisfies the \textit{tensor-condition} if \begin{equation}
\mathrm{im}\left(\mathrm{id}_{\mathbb{F}}\otimes_{A_n^{\mathbb{F}}} i_L : \mathbb{F} \otimes_{A_n^{\mathbb{F}}} L \longrightarrow \mathbb{F} \otimes_{A_n^{\mathbb{F}}} \mathcal{F}_n^{\mathbb{F}}(\xi_0) \right)=0 \notag
\end{equation}
where $i_L: L \lhook\joinrel\xlongrightarrow{\subseteq} \mathcal{F}_n^{\mathbb{F}}(\xi_0)$ denotes the canonical inclusion.
\end{Defi}
If $L \subseteq \mathcal{F}_n^{\mathbb{F}}(\xi_0)$ is a graded $A_n^{\mathbb{F}}$-submodule which satisfies the tensor-condition, then $(\mathcal{F}_n^{\mathbb{F}}(\xi_0), \pi_L)$ is a free hull of $\mathcal{F}_n^{\mathbb{F}}(\xi_0)/L$. Here $\pi_L:\mathcal{F}_n^{\mathbb{F}}(\xi_0) \to \mathcal{F}_n^{\mathbb{F}}(\xi_0)/L$ denotes the canonical projection. The reason is that if we tensor the graded short exact sequence
\begin{equation}
0 \longrightarrow L \xlongrightarrow{i_L} \mathcal{F}_n^{\mathbb{F}}(\xi_0) \longrightarrow \mathcal{F}_n^{\mathbb{F}}(\xi_0)/L \longrightarrow 0 \notag
\end{equation}
we obtain a graded short exact sequence
\begin{equation}
\mathbb{F} \otimes_{A_n^{\mathbb{F}}} L \xlongrightarrow{ \mathrm{id}_{\mathbb{F}} \otimes_{A_n^{\mathbb{F}}} i_{L}} \mathbb{F} \otimes_{A_n^{\mathbb{F}}} \mathcal{F}_n^{\mathbb{F}}(\xi_0)\xlongrightarrow{ \mathrm{id}_{\mathbb{F}} \otimes_{A_n^{\mathbb{F}}} \pi_{L}}  \mathbb{F} \otimes_{A_n^{\mathbb{F}}} \mathcal{F}_n^{\mathbb{F}}(\xi_0)/L \longrightarrow 0 \notag
\end{equation}
since tensoring is right exact. Now the tensor condition tells us that 
\begin{equation}
\mathrm{im}\left(\mathrm{id}_{\mathbb{F}} \otimes_{A_n^{\mathbb{F}}} i_{L}\right)= \mathrm{ker}\left( \mathrm{id}_{\mathbb{F}} \otimes_{A_n^{\mathbb{F}}} \pi_L \right)=0 \notag.
\end{equation}
Hence, $\mathrm{id}_{\mathbb{F}} \otimes_{A_n^{\mathbb{F}}} \pi_{L} $ has to be an isomorphism, which shows that $(\mathcal{F}_n^{\mathbb{F}}(\xi_0), \pi_L)$ is a free hull of $\mathcal{F}_n^{\mathbb{F}}(\xi_0)/L$. In \cref{counterex}, we will give an equivalent characterization of the tensor-condition.

\subsubsection{Complete classification: the main result}
The next definition is from \cite[Sec.~4.5]{Carlsson2009} with the difference that in \cite[Sec.~4.5]{Carlsson2009} the \textit{tensor-condition} is missing. 

\begin{Defi} Denote by $S_n^{\mathbb{F}}(\xi_0,\xi_1)$ the set of all graded $A_n^{\mathbb{F}}$-submodules \begin{equation}
L\subseteq \mathcal{F}_n^{\mathbb{F}}(\xi_0) \notag
\end{equation}
which satisfy the tensor-condition and $\Xi_{n,0}^{\mathbb{F}}(L)=\xi_1$.
\end{Defi}

For the following, we denote by $\mathrm{Aut}(\mathcal{F}_n^{\mathbb{F}}(\xi_0))$ the automorphism group of $\mathcal{F}_n^{\mathbb{F}}(\xi_0)$ in $\mathbf{Grf}_n(A_n^{\mathbb{F}})$. The next definition, proposition is from \cite[Sec.~4.5]{Carlsson2009}.

\begin{Defiprop}[{{\cite[Sec.~4.5]{Carlsson2009}}}]\label{classification surjective map} Assume that ${S_n^{\mathbb{F}}(\xi_0,\xi_1) \neq \emptyset}$. Then $\mathrm{Aut}(\mathcal{F}_n^{\mathbb{F}}(\xi_0))$ acts as a group on $S_n^{\mathbb{F}}(\xi_0,\xi_1)$ where
\begin{equation}f \cdot L := f(L) \notag
\end{equation}
for $ f \in \mathrm{Aut}(\mathcal{F}_n^{\mathbb{F}}(\xi_0))$ and ${L \in S_n^{\mathbb{F}}(\xi_0,\xi_1)}$. We obtain a well-defined set theoretic surjection
\begin{align}
\varrho^{\mathbb{F}}_{n}(\xi_0,\xi_1):S_n^{\mathbb{F}}(\xi_0,\xi_1) \longrightarrow I_n^{\mathbb{F}}(\xi_0,\xi_1), \quad L \longmapsto \left[{\mathcal{F}_n^{\mathbb{F}}(\xi_0)}/{L}\right]. \notag
\end{align}
\end{Defiprop}

Note that without the tensor-condition, the map in \cref{classification surjective map} would not be well-defined. This is due to the fact that without the tensor-condition $\Xi^{\mathbb{F}}_{n,0}(\mathcal{F}_n^{\mathbb{F}}(\xi_0)/L)$ is generally not equal to $\xi_0$. A counterexample is given in \cref{counterex}.

\begin{proof}[Proof of \cref{classification surjective map}] Let $\varrho:=\varrho^{\mathbb{F}}_{n}(\xi_0,\xi_1)$ and $\mathcal{F}_0:=\mathcal{F}_n^{\mathbb{F}}(\xi_0)$. First of all we have to check that $\varrho$ is well-defined: for this, let $L \in S_n^{\mathbb{F}}(\xi_0,\xi_1)$. The discussion after \cref{defi tensor cond} shows that $\mathcal{F}_0/L$ satisfies $\Xi_{n,0}^{\mathbb{F}}(\mathcal{F}_0/L)=\xi_0$. By definition of $S_n^{\mathbb{F}}(\xi_0, \xi_1)$, we have $\Xi_{n,0}^{\mathbb{F}}(L)=\xi_1$. Therefore, we obtain
\begin{align}
\Xi_{n,0}^{\mathbb{F}}([\mathcal{F}_0/L])&=\Xi_{n,0}^{\mathbb{F}}(\mathcal{F}_0/L)=\xi_0, \notag \\
\Xi_{n,1}^{\mathbb{F}}([\mathcal{F}_0/L])&=\Xi_{n,1}^{\mathbb{F}}(\mathcal{F}_0/L)=\Xi_{n,0}^{\mathbb{F}}(L)=\xi_1 \notag
\end{align}
which shows that $[\mathcal{F}_0/L] \in I_n^{\mathbb{F}}(\xi_0, \xi_1)$. Hence, $\varrho$ is well-defined.

$\varrho$ is surjective: let $[M]\in I_n^{\mathbb{F}}(\xi_0,\xi_1)$ and $p_0:\mathcal{F}_0 \to M$ be a free hull. Let $K_0:=\mathrm{ker}(p_0)$. We claim that $K_0 \in S_n^{\mathbb{F}}(\xi_0, \xi_1)$: since $[M] \in I_n^{\mathbb{F}}(\xi_0, \xi_1)$, we have \begin{equation}
\Xi_{n,1}(M)=\Xi_{n,0}(K_0)=\xi_1 \notag.
\end{equation}
So, it remains to show that $K_0$ satisfies the tensor condition.
$p_0$ induces a graded isomorphism $\overline{p_0}: \mathcal{F}_0/K_0 \to M$ such that the diagram \begin{equation}
\begin{tikzcd}
\mathcal{F}_0 \arrow[r, "p_0"] \arrow[d, "\pi_{K_0}"'] & M \\
\mathcal{F}_0/K_0 \arrow[ru," \overline{p_0}"'] 
\end{tikzcd} \notag
\end{equation}
commutes. Here $\pi_{K_0}: \mathcal{F}_0 \to \mathcal{F}_0/K_0$ denotes the canonical projection. Let ${q_0:=\overline{p_0}^{-1} \circ p_0}$. Then $(\mathcal{F}_0,q_0)$ is a free hull of $\mathcal{F}_0 / K_0$ and $\mathrm{ker}(q_0)=K_0$.
Thus, we obtain a graded short exact sequence
\begin{equation}
0 \longrightarrow K_0 \xlongrightarrow{i_{K_0}} \mathcal{F}_0 \xlongrightarrow{q_0} \mathcal{F}_0/K_0 \longrightarrow 0 \notag
\end{equation}
where $i_{K_0}: K_0 \lhook\joinrel\xlongrightarrow{\subseteq} \mathcal{F}_0$ denotes the canonical inclusion. Since tensoring is right exact, we obtain a graded short exact sequence
\begin{equation}
\mathbb{F} \otimes_{A_n^{\mathbb{F}}} K_0 \xlongrightarrow{ \mathrm{id}_{\mathbb{F}} \otimes_{A_n^{\mathbb{F}}} i_{K_0}} \mathbb{F} \otimes_{A_n^{\mathbb{F}}} \mathcal{F}_0 \xlongrightarrow{ \mathrm{id}_{\mathbb{F}} \otimes_{A_n^{\mathbb{F}}} q_0}  \mathbb{F} \otimes_{A_n^{\mathbb{F}}} \mathcal{F}_0/K_0 \longrightarrow 0 \notag. 
\end{equation}
 Since $(\mathcal{F}_0,q_0)$ is a free hull of $\mathcal{F}_0/K_0$,
\begin{equation}
\mathrm{im}\left(\mathrm{id}_{\mathbb{F}} \otimes_{A_n^{\mathbb{F}}} i_{K_0}\right)= \mathrm{ker}\left( \mathrm{id}_{\mathbb{F}} \otimes_{A_n^{\mathbb{F}}} q_0 \right)=0 \notag.
\end{equation}
Thus, $K_0$ satisfies the tensor condition and we conclude that $K_0 \in S^{\mathbb{F}}_{n} (\xi_0,\xi_1)$. Since $\mathcal{F}_0 /K_0 \cong M$, we have $\varrho(K_0)=[\mathcal{F}_0 /K_0]=[M]$.
\end{proof}
\newpage

With the next theorem we accomplish the main goal of this section: the complete classification of $\mathbf{Grf}_n(A_n^{\mathbb{F}})/_{\cong}$.

\begin{Thm}[{{\cite[Thm.~9]{Carlsson2009}}}]\label{classification bijection} Assume that $S_n^{\mathbb{F}}(\xi_0,\xi_1) \neq \emptyset$. The map $\varrho^{\mathbb{F}}_{n}(\xi_0,\xi_1)$ satisfies the formula \begin{equation}\varrho^{\mathbb{F}}_{n}(\xi_0,\xi_1)(f \cdot L)=\varrho^{\mathbb{F}}_{n}(\xi_0,\xi_1)(L) \notag
\end{equation}
for all $f \in \mathrm{Aut}(\mathcal{F}_n^{\mathbb{F}}(\xi_0))$ and all $L \in S_n^{\mathbb{F}}(\xi_0,\xi_1)$. Now $\varrho^{\mathbb{F}}_{n}(\xi_0,\xi_1)$ induces a set theoretic bijection 
\begin{ceqn}
\begin{align}
\overline{\varrho}^{\mathbb{F}}_{n}(\xi_0,\xi_1):S_n^{\mathbb{F}}(\xi_0,\xi_1)/{\mathrm{Aut}(\mathcal{F}_n^{\mathbb{F}}(\xi_0))}\xlongrightarrow{\sim} {I}_{n}^{\mathbb{F}}(\xi_0,\xi_1) \notag,
\end{align}
\end{ceqn}
where $S_n^{\mathbb{F}}(\xi_0,\xi_1)/{\mathrm{Aut}(\mathcal{F}_n^{\mathbb{F}}(\xi_0))}
:=\{\mathrm{Aut}(\mathcal{F}_n^{\mathbb{F}}(\xi_0))\cdot L \mid L \in S_n^{\mathbb{F}}(\xi_0,\xi_1)\}$
denotes the orbit space.
\end{Thm}

This theorem is stated in \cite[Thm.~9]{Carlsson2009}. In \cite{Carlsson2009}, the authors also give a proof for it. Our proof follows \cite{Carlsson2009}.

\begin{proof} Let $\varrho:=\varrho^{\mathbb{F}}_{n}(\xi_0,\xi_1) $ and $\mathcal{F}_0:=\mathcal{F}_n^{\mathbb{F}}(\xi_0)$.
For $f\in \mathrm{Aut}(\mathcal{F}_n^{\mathbb{F}}(\xi_0))$, we obtain a commutative diagram
\begin{equation}
\begin{tikzcd}[row sep=1.5cm, column sep = 1.5cm]
     \mathcal{F}_0 \arrow[r, "f", "\sim"'] \arrow[d, "\pi_L"]  &  \mathcal{F}_0 \arrow[d,"\pi_{f(L)}"]  \\
     \mathcal{F}_0/{L} \arrow[r, "{\overline{f}}", "\sim"'] &  \mathcal{F}_0/{f( L)}
\end{tikzcd} \notag 
\end{equation}
where $\overline{f}$ is an isomorphism. Thus, $\varrho(f\cdot L)=\varrho(f( L))=\varrho(L)$. 

$\overline{\varrho}$ is surjective since $\varrho$ is surjective. To show that $\overline{\varrho}$ is injective, suppose that we have ${L,L' \in S_n^{\mathbb{F}}(\xi_0,\xi_1)}$ with $\varrho(L)=\varrho(L')$, i.e. there exists an isomorphism \begin{equation}
f:\mathcal{F}_0/{L} \xlongrightarrow{\sim} \mathcal{F}_0/{L'} \notag.
\end{equation}
By \cref{Free hull thm}, $f$ lifts to $\widetilde{f} \in \mathrm{Aut}(\mathcal{F}_n^{\mathbb{F}}(\xi_0))$ such that
\begin{equation}
\begin{tikzcd}[row sep=1.5cm, column sep = 1.5cm]
     \mathcal{F}_0 \arrow[r, "\widetilde{f}", "\sim"']  \arrow[d,"\pi_{L}"]  &  \mathcal{F}_0 \arrow[d,"\pi_{L'}"]  \\
      \mathcal{F}_0/{L} \arrow[r, "{f}",  "\sim"'] & \mathcal{F}_0/{L'} 
\end{tikzcd} \notag
\end{equation}
commutes. Hence, $\widetilde{f} \cdot L=L'$.
\end{proof}
\newpage
\begin{Defi} For every $(\xi_0,\xi_1) \in \mathcal{T}_{n}$ and every $\mathbb{F} \in \mathcal{K}$,
\begin{align} \mathfrak{J}^{\mathbb{F}}_{n}(\xi_0,\xi_1): I_{n}^{\mathbb{F}} (\xi_0,\xi_1) \xlongrightarrow{\sim}  S_n^{\mathbb{F}}(\xi_0,\xi_1)/{\mathrm{Aut}(\mathcal{F}_n^{\mathbb{F}}(\xi_0))}, \quad [M] \longmapsto   \overline{\varrho}^{\mathbb{F}}_{n}(\xi_0,\xi_1)^{-1} ([M])  \notag
\end{align}
defines a complete invariant by \cref{classification bijection}. Since \begin{equation}
\mathbf{Grf}_n(A_n^{\mathbb{F}})/_{\cong}=\bigsqcup_{ (\xi_0,\xi_1) \in \mathcal{T}_{n}} I_n^{\mathbb{F}}(\xi_0,\xi_1) \notag
\end{equation}
we obtain a complete invariant
\begin{align} \mathfrak{J}^{\mathbb{F}}_{n}: \mathbf{Grf}_n(A_n^{\mathbb{F}})/_{\cong} & \xlongrightarrow{\sim}  \bigsqcup_{(\xi_0,\xi_1) \in \mathcal{T}_{n}} S_n^{\mathbb{F}}(\xi_0,\xi_1)/{\mathrm{Aut}(\mathcal{F}_n^{\mathbb{F}}(\xi_0))} \notag
\end{align}
for every $\mathbb{F}\in \mathcal{K}$. Thus, $\{ \mathfrak{J}^{\mathbb{F}}_{n}(\xi_0,\xi_1) \}_{\mathbb{F} \in \mathcal{K}} $ and $\{ \mathfrak{J}^{\mathbb{F}}_{n} \}_{\mathbb{F} \in \mathcal{K}} $ are complete.
\end{Defi}

One of our main goals for the rest of this article is to answer the question if there exists any class of invariants 
\begin{equation}\{f_n^{\mathbb{F}}: \mathbf{Grf}_n(A_n^{\mathbb{F}})/_{\cong} \longrightarrow Q_n^{\mathbb{F}}\}_{\mathbb{F}  \in \mathcal{K}} \notag
\end{equation}
which is complete and discrete. For $n=1$, the answer is given by the \textit{barcode} (see \cref{One dim pers sec}). In \cref{Chapter 2}, we will see that for $n \geq 2$ there exists no discrete and complete class of invariants. As any two complete classes of invariants are equivalent by \cref{Complete invariants equivalent}, it suffices to show that $\{\mathfrak{J}_n^{\mathbb{F}}\}_{\mathbb{F} \in \mathcal{K}}$ is not discrete for $n\geq 2$.

\subsection{Equivalent characterization of the tensor-condition and examples}\label{counterex}

The following proposition gives an equivalent characterization of the tensor-condition:
\begin{Prop}\label{Tensor condition} Recall that $\xi_0=(V_0, \mu_0)$ and $\xi_1=(V_1, \mu_1)$. Let $L \subseteq \mathcal{F}_n^{\mathbb{F}}(\xi_0)$ be a graded submodule with $\Xi_{n,0}^{\mathbb{F}}(L)=\xi_1$. Then $L$ satisfies the tensor-condition if and only if \begin{equation}\pi_v(L_v)=0 \notag
\end{equation}
for all $v \in V_0\cap V_1$, where ${\pi_v: \mathcal{F}_n^{\mathbb{F}}(\xi_0)_v  \to  A_n^{\mathbb{F}}(v)_v^{\mu_0(v)}}$ denotes the canonical projection. Note that ${A_n^{\mathbb{F}}(v)_v^{\mu_0(v)}=\mathbb{F}^{\mu_0(v)}}$. In particular, the tensor-condition is satisfied if ${\xi_1 \succ_D \xi_0}$.
\end{Prop}

This equivalent characterization of the tensor-condition is essential for our results in \cref{Chapter 2}.
Before we proceed with a proof of \cref{Tensor condition}, we give two examples which illustrate how this equivalence works. Moreover, \cref{Exa Dim  1  type} shows that the tensor-condition is necessary in the definition of $S_n^{\mathbb{F}}(\xi_0, \xi_1)$ if we want the map in \cref{classification surjective map} to be well-defined. 

 \begin{figure}[h!]
\centering
\includegraphics[scale=0.4]{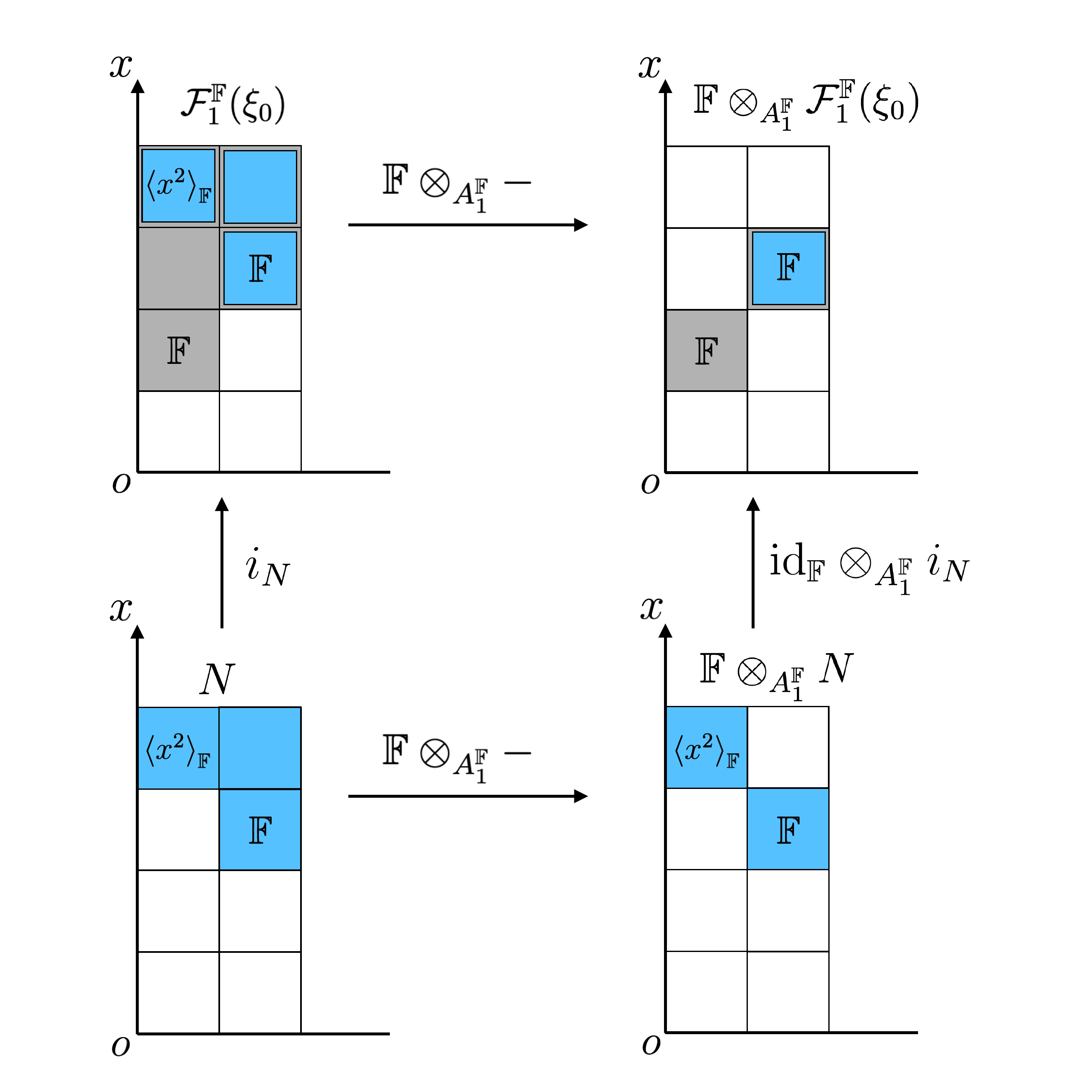} 
\caption{Illustration of \cref{Exa Dim  1  type}.}
\label{fig:Type Tensor dim1,N}
\end{figure}

\begin{Exa}\label{Exa Dim  1  type} This example is illustrated in Figures \ref{fig:Type Tensor dim1,N} and \ref{fig:Type Tensor dim1,L}. Let $n=1$. We have $A_1^{\mathbb{F}}=\mathbb{F}[x]$. Consider
\begin{align}
\xi_0&=\{ (1,1),(2,1)\}, \notag \\ \xi_1&=\{ (2,1),(3,1)\}  \notag
\end{align}
and define 
\begin{align}
L &:=  \braket{(x,0),(0,x)}_{\mathbb{F}[x]} \subseteq  \mathbb{F}[x](1) \oplus \mathbb{F}[x](2) =\mathcal{F}^\mathbb{F}_1(\xi_0), \notag \\
N &:=  \braket{(x^2,0), (0,1)}_{\mathbb{F}[x]} \subseteq  \mathbb{F}[x](1) \oplus \mathbb{F}[x](2)=\mathcal{F}^\mathbb{F}_1(\xi_0) \notag.
\end{align}
$L$ and $N$ are graded submodules of $\mathcal{F}_1^{\mathbb{F}}(\xi_0)$. Moreover, $L$ and $N$ are free as an $\mathbb{F}[x]$-modules since $\mathbb{F}[x]$ is principal ideal domain. Hence, they are graded free as $\mathbb{F}[x]$-modules by \cref{graded free mult}.
 We have $L_2=\braket{(x,0)}_{\mathbb{F}}$ and $L_3=  \braket{(x^2,0),(0,x) }_{\mathbb{F}} $ which means that the basis elements of $L$ are located in degree $2$ and $3$. Similary $N_2= \braket{(0, 1) }_{\mathbb{F}}$ and $N_3=  \braket{(x^2,0), (0,x)}_{\mathbb{F}}  $ and we see that the basis elements of $N$ are also located in degree $2$ and $3$. So, we obtain graded isomorphisms
 \begin{figure}[h!]
\centering
\advance\leftskip+1.5cm
\advance\rightskip-1.5cm

\includegraphics[scale=0.4]{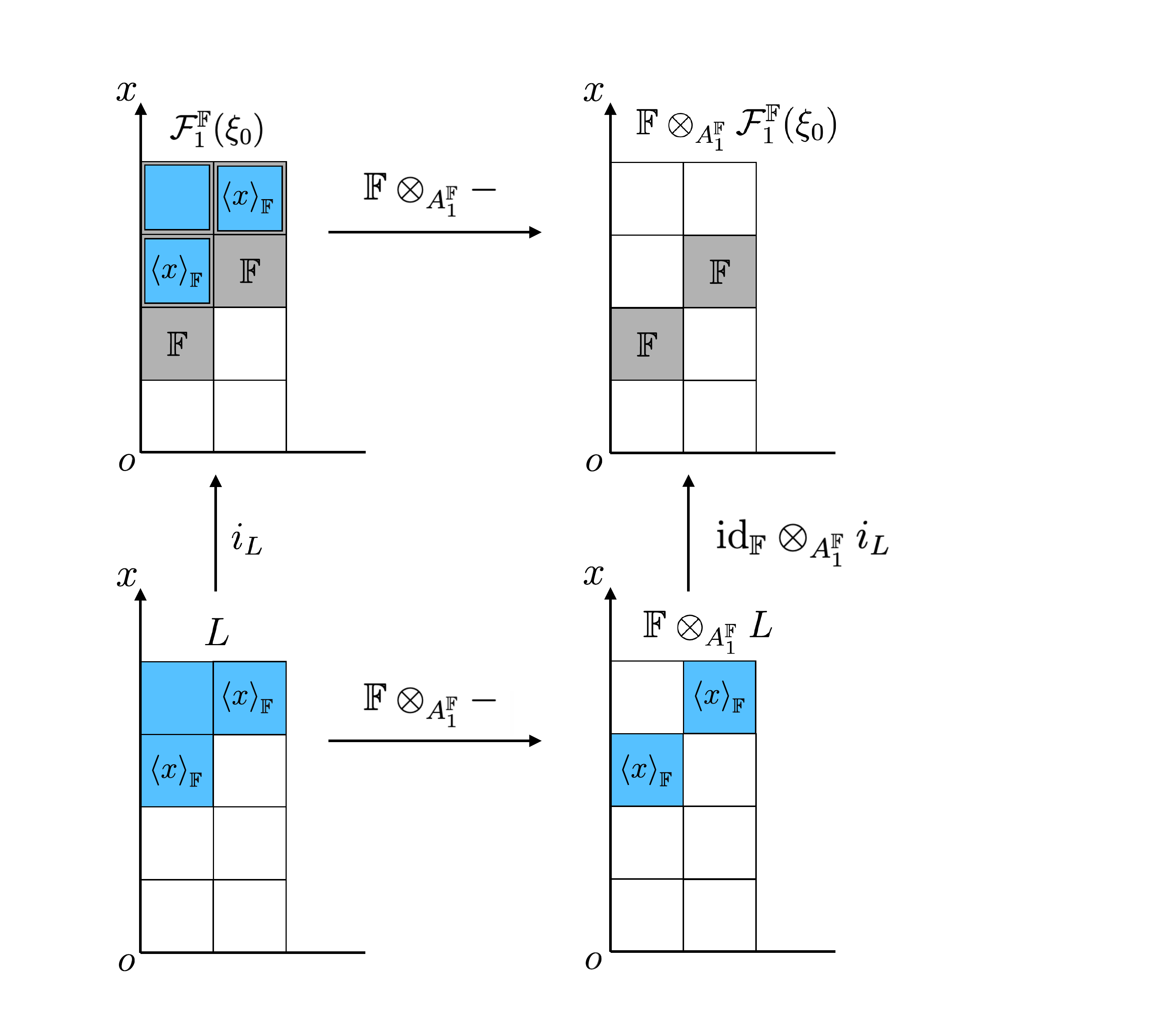} 
\caption{Illustration of \cref{Exa Dim  1  type}.}
\label{fig:Type Tensor dim1,L}
\end{figure}

\begin{align}
p_L: \mathbb{F}[x](2) \oplus \mathbb{F}[x](3) \longrightarrow L, &\quad (0,1) \longmapsto  (0,x) \notag \\
& \quad (1,0) \longmapsto  (x,0)  \notag \\ 
p_N: \mathbb{F}[x](2) \oplus \mathbb{F}[x](3) \longrightarrow N,  &\quad (1,0) \longmapsto  (0,1) \notag \\
& \quad (0,1) \longmapsto  (x^2,0)  \notag
\end{align}
which shows that $\Xi^{\mathbb{F}}_{1,0}(L)=\Xi^{\mathbb{F}}_{1,0}(N)=\xi_1$. As we can see, 
\begin{equation}
\mathcal{F}^\mathbb{F}_1(\xi_0)/L\cong \mathbb{F}(3) \oplus \mathbb{F}(2) \notag.
\end{equation}
Hence, $\Xi^{\mathbb{F}}_{1,0}(\mathcal{F}^\mathbb{F}_1(\xi_0)/L)=\xi_0$. But \begin{equation} \mathcal{F}^\mathbb{F}_1(\xi_0)/N \cong \mathbb{F}[x]/\braket{x^2}  (1) \notag.
\end{equation}
So, \begin{equation}\Xi^{\mathbb{F}}_{1,0} \left(\mathcal{F}^\mathbb{F}_1(\xi_0)/N \right) =\Xi^{\mathbb{F}}_{1,0}  \left(\mathbb{F}[x]/\braket{x^2}  (1) \right)=\{(1,1)\} \neq \xi_0 \notag.
\end{equation}
The reason is that $L$ satisfies the tensor-condition by \cref{Tensor condition} while $N$ does not. Moreover, we see that $\xi_1 \succ_D \xi_0$ does not hold but $S_{1}^{\mathbb{F}}(\xi_0,\xi_1)$ is still non-empty.
\end{Exa}

\begin{figure}[h]
\centering
\includegraphics[scale=0.29]{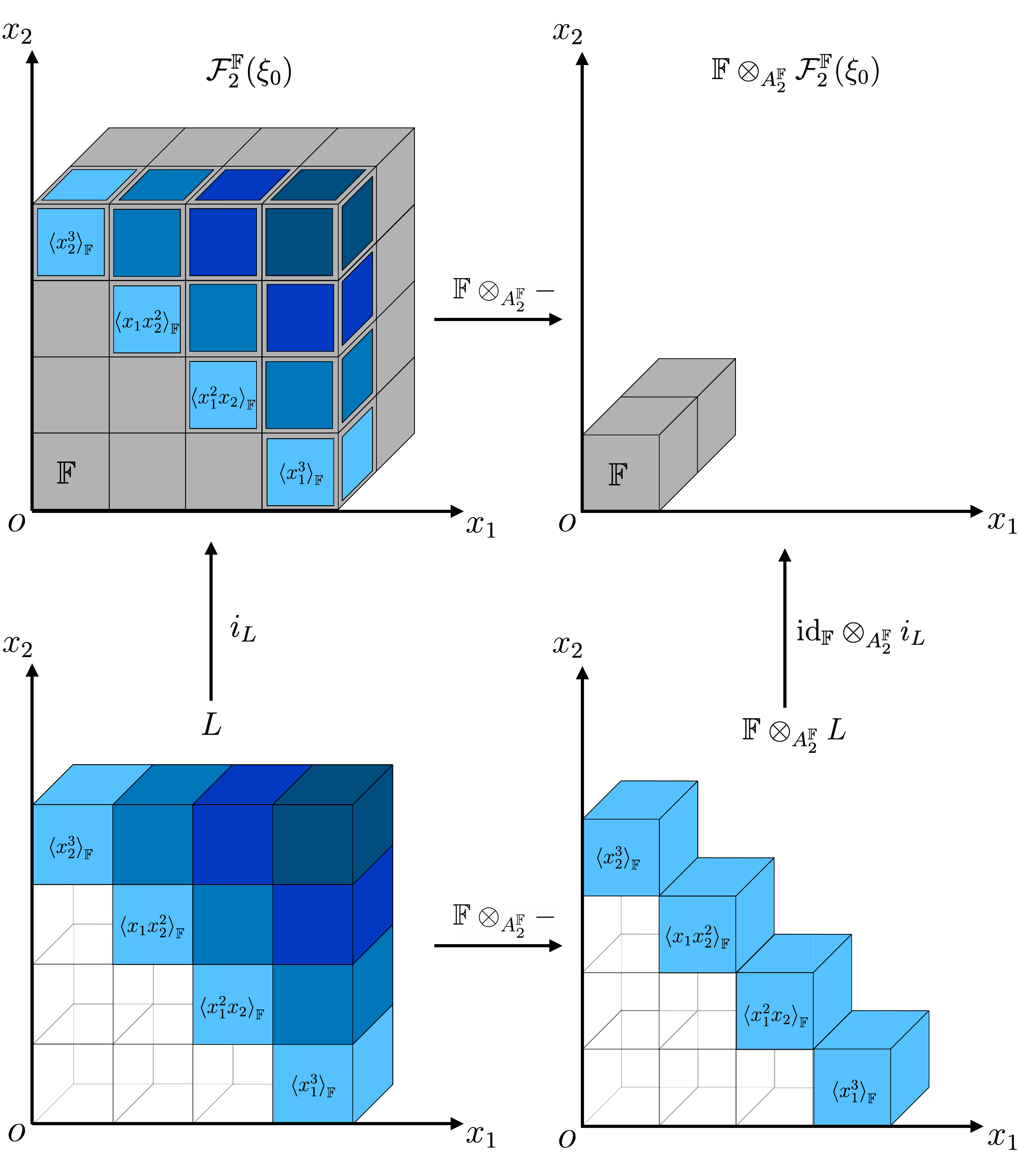} 
\caption{Illustration of \cref{Dim 2 Tensor Ex}.}
\label{fig:Type Tensor dim2.1}
\end{figure}

\begin{Exa}\label{Dim 2 Tensor Ex} This example is illustrated in Figure \ref{fig:Type Tensor dim2.1}.
Consider 
\begin{align}
\xi_0&=\{\textcolor{black}{((0,0),1), ((0,0),2)}\}, \notag \\
\xi_1 &=\{((3,0),1),((2,1),1),((1,2),1),((0,3),1)\} \notag 
\end{align}
and
\begin{align}
L&=\braket{(x_1^3,0), (x_1^2x_2,0),(x_1x_2^2,0),(x_2^3,0) }_{A^{\mathbb{F}}_2}\subseteq A^{\mathbb{F}}_2 \oplus A^{\mathbb{F}}_2 \notag,  \\
N&=\braket{(x_1^3,x_1^3), (x_1^2x_2,0) ,(x_1x_2^2,0) ,(0,x_2^3) }_{A^{\mathbb{F}}_2} \subseteq A^{\mathbb{F}}_2 \oplus A^{\mathbb{F}}_2 \notag.
\end{align}
We have $\xi_1 \succ_D \xi_0$. So, the tensor-condition is satisfied by \cref{Tensor condition} which is illustrated for $L$ in Figure \ref{fig:Type Tensor dim2.1}. We obtain
\begin{equation}
\boldsymbol{\Xi}_{2}^{\mathbb{F}}\left(\mathcal{F}_2^{\mathbb{F}}(\xi_0)/L \right)=\boldsymbol{\Xi}_{2}^{\mathbb{F}} \left(\mathcal{F}_2^{\mathbb{F}}(\xi_0)/N \right)=(\xi_0, \xi_1) \notag.
\end{equation}
Hence, $L,N \in S_{2}^{\mathbb{F}}(\xi_0, \xi_1)$. But $\mathcal{F}_2^{\mathbb{F}}(\xi_0)/L$ has elements without $A_2^{\mathbb{F}}$-torsion and in $\mathcal{F}_2^{\mathbb{F}}(\xi_0)/N$ every element has $A_2^{\mathbb{F}}$-torsion which shows that they are not isomorphic as $A_2^{\mathbb{F}}$-modules and therefore not as $n$-graded $A_2^{\mathbb{F}}$-modules. So, $\boldsymbol{\Xi}_{2}^{\mathbb{F}}$ is not complete. The multisets $\xi_0$ and $\xi_1$ will also appear later in \cref{cont inv} where we follow \cite[Sec.~5.2]{article} and \cite[Sec.~5.2]{Carlsson2009}.
\end{Exa}

\begin{proof}[Proof of \cref{Tensor condition}] Assume that for $v \in V_0\cap V_1$, $\pi_v(L_v)=0$. Let \begin{equation}
\sum_{j=1}^l a_j \otimes y_j \in \mathrm{im}\left(\mathrm{id}_{\mathbb{F}}\otimes_{A_n^{\mathbb{F}}} i_L \right) \notag.
\end{equation}
It suffices to show that $a_j \otimes y_j=0$ for all $j \in \{1, \dots, l\}$. So, let $j \in \{1, \dots, l\}$ and set $a:=a_j$ and $y:=y_j$. Then \begin{equation}y=\sum_{v \in {\mathbb{N}^n}} y_v \in \bigoplus_{v \in {\mathbb{N}^n}} L_v \subseteq \mathcal{F}_n^{\mathbb{F}}(\xi_0)  \notag.
\end{equation}
Recall that we have a graded isomorphism
\begin{equation} \mathbb{F}  \otimes_{A_n^{\mathbb{F}}} \mathcal{F}_n^{\mathbb{F}}(\xi_0)  \xlongrightarrow{\sim}  \mathcal{V}_n^{\mathbb{F}}(\xi_0)
\end{equation}
of $n$-graded $\mathbb{F}$-vector spaces. Hence, $a \otimes y_v =0$ for all $v \in {\mathbb{N}^n} \setminus V_0$. It remains to show that $a \otimes y_v =0  $ for all $v \in V_0$. Let $v \in V_0$. As the generators of $L$ are located in $\bigcup_{v \in V_1} L_v$, it suffices to show that $a \otimes y_v =0  $ for all $v \in V_0 \cap V_1$. Let $v \in V_0 \cap V_1$. We have 
\begin{equation}y_v \in L_v\subseteq \mathcal{F}_n^{\mathbb{F}}(\xi)_v=\bigoplus_{w \in V_0} A_n^{\mathbb{F}}(w)^{\mu(w)}_v \notag.
\end{equation}
$y_v$ is of the form
\begin{equation}
y_v=((z^{(w)}_1, \dots,  z^{(w)}_{\mu(w)}))_{w \in V_0} \notag
\end{equation}
where for $w \prec v$, $z^{(w)}_i=\lambda^{(w)}_{i}x^{v-w}$ for suitable $\lambda^{(w)}_{i} \in \mathbb{F}$ and $z^{(w)}_i=0$ if $w \npreceq v$. If $v=w$ we have $z^{(v)}_i=0$ since $\pi_v(L_v)=0$ by assumption. Enumerate $V_0=\{w_1, \dots, w_k\}$ and identify 
\begin{equation}
\mathcal{F}_n^{\mathbb{F}}(\xi)_v=\bigoplus_{w \in V_0} A_n^{\mathbb{F}}(w)^{\mu(w)}_v = \bigoplus_{j=1}^k A_n^{\mathbb{F}}(w_j)^{\mu(w_j)}_v \notag
\end{equation}
along this enumeration. Let $j \in \{1, \dots, k\}$. If $w_j \npreceq v$ set $\widetilde{z}^{(w_j)}_i:=z^{(w_j)}_i= 0$ for all $i \in \{1, \dots, \mu_0(w_j)\}$ and $x^{v-w_j}:=0$. If $w_j \preceq v$ set $\widetilde{z}^{(w_j)}_i:=\lambda^{(w_j)}_{i} $ for all $i \in \{1, \dots, \mu_0(w_j)\}$. By definition of the $A_n^{\mathbb{F}}$-module structure on $\mathbb{F}$, we obtain 
{\small \begin{align}
&a \otimes y_v =a \otimes ((z^{(w_1)}_1, \dots,  z^{(w_1)}_{\mu(w_1)}), \dots, (z^{(w_k)}_1, \dots,  z^{(w_k)}_{\mu(w_k)})) \notag \\
&= \sum_{j=1}^k a  \otimes ((0,\dots,0), \dots, (0, \dots, 0), \underset{j}{({z}^{(w_j)}_1, \dots, {z}^{(w_j)}_{\mu(w_j)})}, (0,\dots,0), \dots, (0, \dots, 0)) \notag \\
&= \sum_{j=1}^k a \cdot x^{v-w_j} \otimes ((0,\dots,0), \dots, (0, \dots, 0), \underset{j}{(\widetilde{z}^{(w_j)}_1, \dots, \widetilde{z}^{(w_j)}_{\mu(w_j)})}, (0,\dots,0), \dots, (0, \dots, 0)) \notag \\
&= \sum_{j=1}^k 0 \otimes ((0,\dots,0), \dots, (0, \dots, 0), \underset{j}{(\widetilde{z}^{(w_j)}_1, \dots, \widetilde{z}^{(w_j)}_{\mu(w_j)})}, (0,\dots,0), \dots, (0, \dots, 0))=0 \notag .
\end{align}}

Assume that $ \mathrm{id}_{\mathbb{F}}\otimes_{A_n^{\mathbb{F}}} i_L : \mathbb{F} \otimes_{A_n^{\mathbb{F}}} L \to \mathbb{F} \otimes_{A_n^{\mathbb{F}}} \mathcal{F}_n^{\mathbb{F}}(\xi_0) $ is the zero map. Let $v \in V_0\cap V_1$ and $y \in L_v$. We can write 
\begin{equation}
y=((z^{(w)}_{1} , \dots, z^{(w)}_{\mu_0(w)} ))_{w \in V_0}  \in \mathcal{F}_n^{\mathbb{F}}(\xi_0)_v= \bigoplus_{w \in V_0} A_n^{\mathbb{F}}(w)_v^{\mu_0(w)}\notag
\end{equation}
where $z^{(w)}_i \in \mathbb{F}$ if $w=v$. We have $1 \otimes y=0$ in $\mathbb{F} \otimes_{A_n^{\mathbb{F}}}\mathcal{F}_n^{\mathbb{F}}(\xi_0)_v$ by assumption. So, if $v=w$, then necessarily $z^{(w)}_i=0$ for all $i \in \{1, \dots, \mu_0(w)\}$.
\end{proof}

\subsection{One-dimensional persistence}\label{One dim pers sec}
Let $n=1$. Recall that $A_1^{\mathbb{F}} =\mathbb{F}[x]$. For the following, let $\xi_0=(V_0, \mu_0)$ and ${\xi_1=(V_1, \mu_1)} $ be two one-dimensional multisets. The following material about one-dimensional persistence is very common (see for example \cite{onedimpers}). Let $[M]  \in  I_{1}^{\mathbb{F}}(\xi_0,\xi_1)$. The \textit{the structure theorem} (see for example \cite[Thm.~2.1]{onedimpers} (our notation is different)), states that
\begin{align}\label{struc thm}
M  \cong  \bigoplus_{i=1}^{m} \bigoplus_{j=1}^{d_i}\mathbb{F}[x]/ \braket{x^{t_{i,j}}}_{\mathbb{F}[x]} (v_{i,j}) \oplus \bigoplus_{i=m+1}^{r} \bigoplus_{j=1}^{d_i} \mathbb{F}[x](v_{i,j})   
\end{align}
as one-graded $\mathbb{F}[x]$-modules for unique $r,m \in \mathbb{N}$ with $r \leq m$ and unique ${d_i, t_{i,j},v_{i,j}\in \mathbb{N}}$ with $v_{i,j}=v_{i,k}$ for all $j,k \in \{1, \dots, d_i\}$. Note that we generally do not have $t_{i,j}=t_{i,k}$ for all $j,k \in \{1, \dots, d_i\}$. Using the graded isomorphism in (\ref{struc thm}), we may define a finite multiset $B_1^{\mathbb{F}}(\xi_0, \xi_1)([M])$ via 
\begin{equation}
B_1^{\mathbb{F}}(\xi_0, \xi_1)([M]):=\bigcup_{i=1}^{m} \bigcup_{j=1}^{d_i} \{([v_{i,j},v_{i,j}+ t_{i,j}),j) \} \cup \bigcup_{i=m+1}^{r} \bigcup_{j=1}^{d_i} \{([v_{i,j},\infty),j) \} \notag
\end{equation}
where for $v \in \mathbb{R}$ and $t \in \mathbb{R} \cup \{\infty \}$ with $v < t$, $[v,t) \subseteq \mathbb{R}$ denotes a half-open interval.
$B_1^{\mathbb{F}}(\xi_0, \xi_1)([M])$ is called the \textit{barcode of} $[M]$ \textit{over} $(\xi_0, \xi_1)$. We have 
\begin{align}
\xi_0=\bigcup_{i=1}^{r}  \{ (v_{i,1},1), \dots, (v_{i,d_i},d_i) \} \notag 
\end{align}
and
\begin{align}
 \xi_1=\bigcup_{i=1}^{m'} \{(v'_{i,1}+ t'_{i,1},1), \dots, (v'_{i,d'_i}+ t'_{i,d'_i},d'_i) \} \notag.
\end{align}
where $v'_{i,j}=v_{l,k}$ and $t'_{i,j}=t_{l,k}$ after renumbering the $t_{i,j}$ and $v_{i,j}$ according to the multiplicity of the $t_{i,j}$, such that $t_{i,j}=t_{i,k}$ for all $j,k \in \{1, \dots, d'_i \}$ and ${\sum_{i=1}^{m'} d'_i=\sum_{i=1}^{m} d_i}$.
As we can see, $\xi_0$ and $\xi_1$ correspond to the start and end points of elements in $B_1^{\mathbb{F}}(\xi_0, \xi_1)([M])$. The intuitive interpretation of the barcode is that it captures the lifetime-length of homogeneous generators of $M$.

\begin{Exa}\label{Illustration of Barcodes} Figure \ref{fig:Barcode3} illustrates the barcode
\begin{align}
\{([0,\infty),1), ([0,2),1), ([0,1),1), ([1, \infty),1),([1,3),1),([1,3),2), ([2,3),1),([2,3),2) \}  \notag.
\end{align}
\end{Exa}

\begin{figure}
\centering
\includegraphics[scale=0.3]{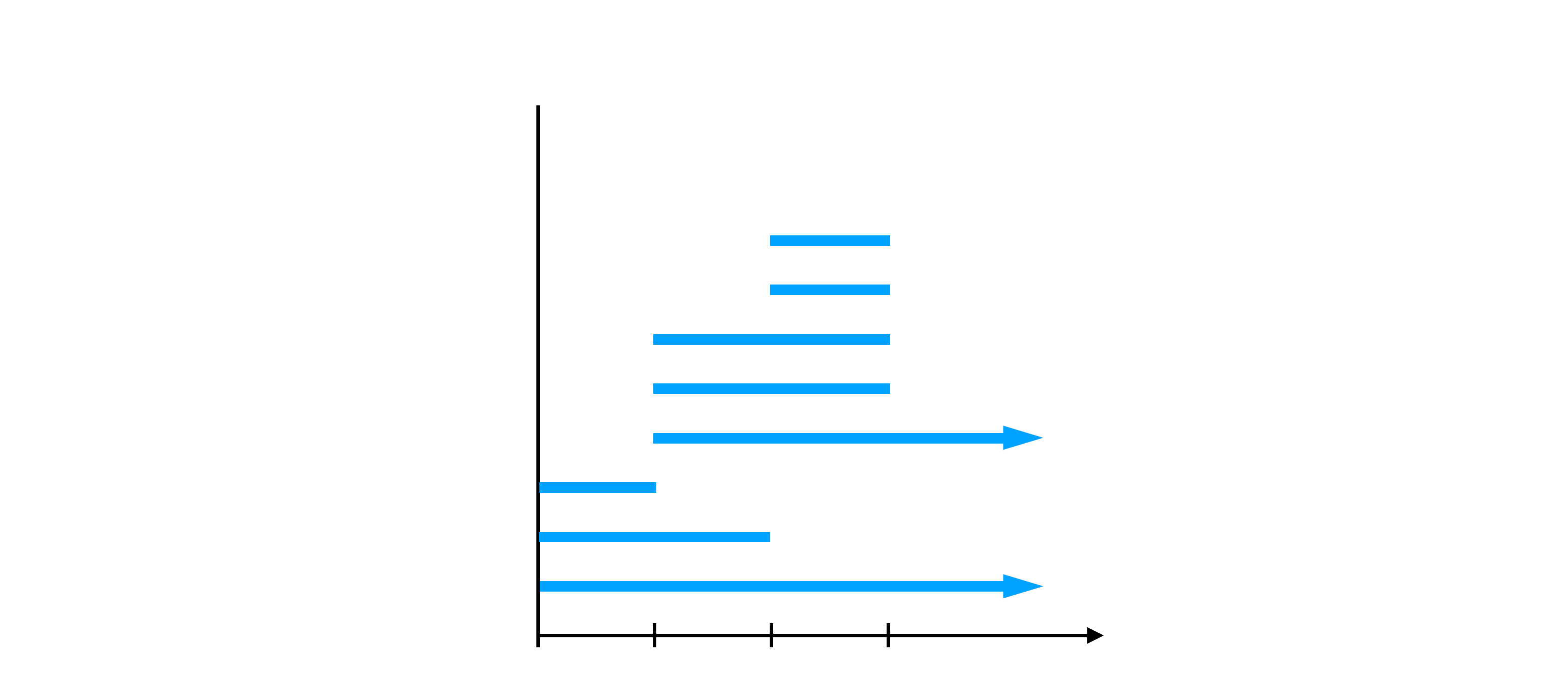} 
\caption{barcode diagram (illustration of \cref{Illustration of Barcodes}).}
\label{fig:Barcode3}
\end{figure}
Now define $\mathcal{B}(\xi_0, \xi_1)$ as the set of all mulitsets 
\begin{equation}
\bigcup_{i=1}^{m} \bigcup_{j=1}^{d_i} \{([v_{i,j},w_{i,j}),j) \} \cup \bigcup_{i=m+1}^{r} \bigcup_{j=1}^{d_i} \{([v_{i,j},\infty),j) \} \notag
\end{equation}
such that $v_{i,j} < w_{i,j}$,
\begin{align}
\xi_0=\bigcup_{i=1}^{r}  \{ (v_{i,1},1), \dots, (v_{i,d_i},d_i) \} \notag
\end{align}
and
\begin{align}
 \xi_1=\bigcup_{i=1}^{m'} \{(w'_{i,1},1), \dots, (w'_{i,d'_i},d'_i) \} \notag
\end{align}
where $w'_{i,j}=w_{l,k}$ after renumbering the $w_{l,k}$ according to their multiplicity, such that $w'_{i,j}=w'_{i,k}$ for all $j,k \in \{1, \dots, d'_i \}$ and $\sum_{i=1}^{m'} d'_i=\sum_{i=1}^{m} d_i$.
$\mathcal{B}(\xi_0, \xi_1)$ is called \textit{the space of barcodes over} $(\xi_0, \xi_1)$ and
\begin{equation}
\mathcal{B}:= \bigsqcup_{(\xi_0, \xi_1) \in \mathcal{T}_1} \mathcal{B}(\xi_0, \xi_1) \notag
\end{equation}
is called \textit{the space of barcodes}. We obtain a set theoretic bijection
\begin{equation}
B_1^{\mathbb{F}}(\xi_0, \xi_1):   I_{1}^{\mathbb{F}}(\xi_0,\xi_1) \xlongrightarrow{\sim} \mathcal{B}(\xi_0,  \xi_1), \quad [M] \longmapsto B_1^{\mathbb{F}}(\xi_0, \xi_1)([M]) \notag
\end{equation}
and $\{B_1^{\mathbb{F}}(\xi_0, \xi_1) \}_{\mathbb{F} \in \mathcal{K}}$ defines a discrete class of complete invariants. Since \begin{equation}
\mathbf{Grf}_1(A_1^{\mathbb{F}})/_{\cong}=\bigsqcup_{ (\xi_0,\xi_1) \in \mathcal{T}_{1}} I_1^{\mathbb{F}}(\xi_0,\xi_1) \notag
\end{equation}
we obtain a set theoretic bijection
\begin{equation} B^{\mathbb{F}}_1: \mathbf{Grf}_1(\mathbb{F}[x])/_{\cong} \xlongrightarrow{\sim}  \mathcal{B}   \notag
\end{equation}
and $\{B_1^{\mathbb{F}} \}_{\mathbb{F} \in \mathcal{K}}$ defines a discrete class of complete invariants, the so-called \textit{barcode}.  We have $\{B_1^{\mathbb{F}}\}_{\mathbb{F} \in \mathcal{K}}  \cong \{\mathfrak{J}_{1}^{\mathbb{F}} \}_{\mathbb{F} \in \mathcal{K}}$ by \cref{Complete invariants equivalent}. So, $ \{\mathfrak{J}_{1}^{\mathbb{F}} \}_{\mathbb{F} \in \mathcal{K}}$ is a complete class of discrete invariants.

Now one might ask the question if there are barcode-like invariants $B^{\mathbb{F}}_{n}$ for $n\geq 2 $ that are equivalent to $\mathfrak{J}^{\mathbb{F}}_{n}$. If $n \geq 2$ it is unfortunately not possible to find a class of invariants $\{f_n^{\mathbb{F}}: \mathbf{Grf}_n(A_n^{\mathbb{F}}) \to Q_n^{\mathbb{F}}\}_{\mathbb{F} \in \mathcal{K}}$, which is discrete and complete. This is one of the main results in \cref{Chapter 2}. In the next section, we present a discrete generalization of the barcode to dimension $n\geq 2$, the so called \textit{rank invariant}.
\subsection{The rank invariant}\label{Rank invariant section}
This section is based on \cite[Sec.~6]{article} and \cite[Sec.~6]{Carlsson2009}. Let \begin{equation}
\mathbb{D}^n:=\{(u,v)\mid u\in \mathbb{N}^n, \, v\in {\mathbb{N}^n},\, u\preceq v \} \subseteq \mathbb{N}^n \times {\mathbb{N}^n}  \notag
\end{equation}
be the subset above the diagonal.

\begin{Defi}[Rank invariant] The assignment 
\begin{align}
\rho_{n}^{\mathbb{F}}: \mathbf{Grf}_n(A_n^{\mathbb{F}})/_{\cong} & \longrightarrow  \mathrm{Hom}_{\mathrm{Sets}}(\mathbb{D}^n, \mathbb{N}^2), & \notag \\
[M] & \longmapsto \left(\mathbb{D}^n \longrightarrow \mathbb{N}^2,\, \,
(u,v) \longmapsto \left( \mathrm{rank}(x^{v-u}:M_u \to M_v), \mathrm{dim}_{\mathbb{F}}(M_u)  \right)\right) \notag
\end{align}
defines a discrete class of invariants $\{\rho_{n}^{\mathbb{F}}\}_{\mathbb{F} \in \mathcal{K}}$, the so-called \textit{rank invariant}. Note that $\mathrm{Hom}_{\mathrm{Sets}}(\mathbb{D}^n, \mathbb{N}^2)$ is uncountable, but $\mathrm{im}(\rho_{n}^{\mathbb{F}})$ is countable.
\end{Defi}
Note that in comparison to \cite{Carlsson2009} and \cite{article} our definition of the rank invariant does not map the value $(u, \infty)$. This is not a problem as our definition of the rank invariant already captures all of the information we want to capture. In addition to \cite{Carlsson2009} and \cite{article}, we also include the dimension in each degree. The next theorem is from \cite[Thm.~12]{Carlsson2009} and \cite[Thm.~5]{article} where we use our definition of the rank invariant.

\begin{Thm}[{{\cite[Thm.~12]{Carlsson2009}, \cite[Thm.~5]{article}}}]\label{Barcode rank equiv} We have $\{\rho_{1}^{\mathbb{F}}\}_{\mathbb{F} \in \mathcal{K}} \cong \{B_{1}^{\mathbb{F}}\}_{\mathbb{F} \in \mathcal{K}}$ which shows that $\{\rho_{1}^{\mathbb{F}}\}_{\mathbb{F} \in \mathcal{K}} $ is complete.
\end{Thm}
\begin{proof}[Sketch of the proof]
As in \cite[Thm.~5]{article} and \cite[Thm.~12]{Carlsson2009}, the idea is to show that the assignment
\begin{align}
\mathcal{B}   \longrightarrow \mathrm{im}\left(\rho_1^{\mathbb{F}} \right), \notag
\end{align}
which sends a barcode $B \in \mathcal{B}$ to the function
\begin{align}\label{Barcode Rank bijection}
 \mathbb{D}^1 &\longrightarrow \mathbb{N}^2, \notag \\
(u,v) & \longmapsto \left( \left \vert \left \{ ([t,s),i) \in B \mid  [u,v) \subseteq [t,s) \right \} \right \vert, \left \vert \left \{ ([t,s),i) \in B \mid u \in [t,s) \right \} \right \vert   \right)  ,
\end{align}
is a bijection. Our proof goes a different way than the proof in \cite{article} and \cite{Carlsson2009}. 

The idea is as follows: for $u,v \in {\mathbb{N}}$ with $u \leq v$, we have
\begin{equation}\label{dimensionformula}
\mathrm{dim}_{\mathbb{F}}(\mathrm{ker}(M_u \to M_v))= \mathrm{dim}_{\mathbb{F}}(M_u) - \mathrm{rank}(M_u \to M_v)  
\end{equation}
which corresponds to the number of homogeneous generators in $M_u$ whose lifetime end in $M_v$. Using this formula, we may simply read off the rank invariant from a given barcode as described in (\ref{Barcode Rank bijection}) and the assignment is clearly well-defined and injective.
The intuition for surjectivity is as follows: let $C \in \mathrm{im}(\rho_1^{\mathbb{F}})$. Then $C=\rho_1^{\mathbb{F}}([M])$ for some $[M] \in \mathbf{Grf}_1(\mathbb{F}[x])/_{\cong}$. Let $(u,v) \in \mathbb{D}^1= \{ (u,v) \in {\mathbb{N}}^2 \mid u \leq v\}$. We know that $B(\psi_M)$ exists. The question is how  $B(\psi_M)$ is determined by $\rho_1^{\mathbb{F}}([M])$. Using (\ref{dimensionformula}), we see that $\rho_1^{\mathbb{F}}([M])$ gives us all the information about the lifetime length of homogeneous generators for each $M_v$ with $v \in {\mathbb{N}}$. So, the shape of $B(\psi_M)$ is determined as follows: we start in $M_0$ and determine the lifetime length of each homogeneous generator via (\ref{dimensionformula}). Then we do the same for $M_1$ and so on. This procedure terminates, since for some $\alpha \in {\mathbb{N}}$, $M_{\alpha} \to M_v$ has to be an isomorphism of $\mathbb{F}$-vector spaces for all $v \in \mathbb{N}_{\geq \alpha}$ (see \cite[Def.~3.3, Thm.~3.1]{onedimpers}).
\end{proof}
So, $\{\rho_{n}^{\mathbb{F}}\}_{\mathbb{F} \in \mathcal{K}}$ is a discrete generalization of $\{B_{1}^{\mathbb{F}}\}_{\mathbb{F} \in \mathcal{K}}$ to dimension $n \geq 2$.

\section{Parameterization}\label{Chapter 2}
 
The first goal of this section is to show that for $n\geq 2$, there exists no class \begin{equation}\left\{f_n^{\mathbb{F}}: \mathbf{Grf}_n(A_n^{\mathbb{F}})/_{\cong} \longrightarrow Q_n^{\mathbb{F}} \right \}_{\mathbb{F}  \in \mathcal{K}} \notag
\end{equation}
of invariants which is complete and discrete. As $\{\mathfrak{J}_{n}^{\mathbb{F}}\}_{\mathbb{F} \in \mathcal{K}}$ is complete, it suffices to show that $\{\mathfrak{J}_{n}^{\mathbb{F}}\}_{\mathbb{F} \in \mathcal{K}}$ is not discrete for $n\geq 2$ by \cref{Complete invariants equivalent}. Let $\xi_0$ and $\xi_1$ be two finite $n$-dimensional multisets. In Sections \ref{Relation Families section}-\ref{Framed relation families section}, we parameterize $S_{n}^{\mathbb{F}}(\xi_0, \xi_1)$ as a subset of a product of Grassmannians together with a group action of ${\mathrm{Aut}(\mathcal{F}_n^{\mathbb{F}}(\xi_0))} $. This approach follows \cite[Sec.~4.3]{article} and \cite[Sec.~5]{Carlsson2009}. However, the authors ommited proofs and details. We fill in the gaps and give detailed proofs. Using this parameterization, we show the non-existence of a discrete and complete class of invariants in \cref{cont inv}.

In \cref{The Moduli space}, we investigate the second goal of this section: that this subset corresponds to the set of $\mathbb{F}$-points of an algebraic variety together with an algebraic group action of ${\mathrm{Aut}(\mathcal{F}_n^{\mathbb{F}}(\xi_0))} $. In other words, we want to give an explicit parameterization of the moduli space of our classification problem. Here we follow \cite[Sec.~5]{article} and \cite[Sec.~5]{Carlsson2009}.
In \cref{param as subprevariety}, we construct a subprevariety $\mathcal{Y}_n^{\mathbb{F}}(\xi_0,\xi_1)$ of a product of Grassmann varieties (over $\mathbb{F}$) such that \begin{equation}
\mathcal{Y}_n^{\mathbb{F}}(\xi_0,\xi_1)(\mathbb{F})\cong S_{n}^{\mathbb{F}}(\xi_0, \xi_1) \notag.
\end{equation}
$\mathrm{GL}_m(\mathbb{F})$ can be realized as the set of $\mathbb{F}$-points of an affine algebraic group $\mathcal{GL}_m^{\mathbb{F}}$ over $\mathbb{F}$. Now it is possible to show that
$\mathrm{Aut}(\mathcal{F}_n^{\mathbb{F}}(\xi_0))$ corresponds to the set of $\mathbb{F}$-points of a closed subscheme 
\begin{equation}
\mathcal{GL}_{\preceq}^{\mathbb{F}}(\xi_0) \subseteq \mathcal{GL}_{m_0}^{\mathbb{F}} \notag
\end{equation}
where $m_0:=\abs{\xi_0}$. Assuming that $\mathcal{GL}_{\preceq}^{\mathbb{F}}(\xi_0)$ is  an algebraic group, the idea is now (using the functor of points approach) to define for every Noetherian $\mathbb{F}$-algebra $T$ a group action
\begin{equation}
\mathcal{GL}_{\preceq}^{\mathbb{F}}(\xi_0)(T) \times  \mathcal{Y}_n^{\mathbb{F}}(\xi_0,\xi_1)(T) \longrightarrow \mathcal{Y}_n^{\mathbb{F}}(\xi_0,\xi_1)(T) \notag
\end{equation}
which is natural in $T$. For this, we have to determine the $T$-points $\mathcal{GL}_{\preceq}^{\mathbb{F}}(\xi_0)(T) $ and $\mathcal{Y}_n^{\mathbb{F}}(\xi_0,\xi_1)(T)$. For $\mathcal{GL}_{\preceq}^{\mathbb{F}}(\xi_0)(T) $ this is not difficult, but for $\mathcal{Y}_n^{\mathbb{F}}(\xi_0,\xi_1)(T)$ it is. In \cref{Functor of points}, we give some ideas that provide evidence for how this can be done. For Grassmann varieties (and therefore also for products of Grassmann varieties) this problem is also technical but solved (see for instance \cite[(8.4)]{GW}). Since $\mathcal{Y}_n^{\mathbb{F}}(\xi_0,\xi_1)$ is constructed as a subprevariety of a product of Grassmann varieties, we use the approach for Grassmann varieties. In \cite[Sec.~5]{Carlsson2009}, the authors give a proof for the subprevariety-part which does not contain an interpretation of the tensor-condition in terms of polynomial equations. In \cite[Thm.~4]{article}, the subprevariety-part is stated where in comparison to \cite[Sec.~5]{Carlsson2009} a proof and conditions are missing. In \cite[Thm.~4]{article} and \cite[Sec.~5]{Carlsson2009}, the authors claim (without a proof) that the action is algebraic.

\subsection{Relation families}\label{Relation Families section}
For the  following let $\xi_0=(V_0, \mu_0)$ and $\xi_1=(V_1, \mu_1)$ be  two  finite $n$-dimensional multisets such that $S_{n}^{\mathbb{F}}(\xi_0,\xi_1)$ is non-empty. In this section, we start with the parameterization of $S_n^{\mathbb{F}}(\xi_0, \xi_1)/\mathrm{Aut}(\mathcal{F}_n^{\mathbb{F}}(\xi_0))$. The idea is to map $L \in S_n^{\mathbb{F}}(\xi_0, \xi_1)$ to the corresponding familiy of $\mathbb{F}$-vector spaces $(L_w)_{w \in V_1}$ and to determine the conditions under which such a family of $\mathbb{F}$-vector spaces yields an element in $S_n^{\mathbb{F}}(\xi_0, \xi_1)$. These conditions are given in the next definition.

\begin{Defi}[$\mathbb{F}$-relation family]\label{Defi Rel fam} An $\mathbb{F}$-\textit{relation family with respect to} $(\xi_0, \xi_1)$ is a family $(L_w  )_{w\in V_1}$ such that for all $w \in V_1$:
\begin{enumerate}
\item $L_w\subseteq \mathcal{F}_n^{\mathbb{F}}(\xi_0)_w$ is an $\mathbb{F}$-linear subspace.
\item $\mathrm{dim}_\mathbb{F}(L_w)=\mathrm{dim}_\mathbb{F} \left( \mathcal{F}_n^{\mathbb{F}}(\xi_1)_{w} \right)$.
\item if $v \in V_1$ with $v\prec w$, then $x^{w-v}L_{v}\subseteq L_w$.
\item if $w \in V_0 \cap V_1$, then we have  $\pi_w (L_w)=0$ where $\pi_w: \mathcal{F}_n^{\mathbb{F}}(\xi_0)_w \to A_n^{\mathbb{F}}(w)^{\mu_0(w)}_w $ denotes the canonical projection. Note that $A_n^{\mathbb{F}}(w)^{\mu_0(w)}_w=\mathbb{F}^{\mu_0(w)}$.
\item $\mathrm{dim}_\mathbb{F} \left( L_w / \sum_{v \prec w} x^{w-v} L_v \right) = \mu_1(w)$.
\end{enumerate}
$R_n^{\mathbb{F}}(\xi_0,\xi_1)$ denotes the set of all $\mathbb{F}$-\textit{relation families with respect to} $(\xi_0, \xi_1)$.
\end{Defi}

Our definition of $\mathbb{F}$-relation families is based on \cite[Def.~9] {article} and \cite[Sec.~5, p.~85]{Carlsson2009}. \cite{Carlsson2009} is more general than \cite[Def.~9]{article} and has $5$ as an extra condition, but condition $4$ is missing. In \cite{article}, conditions $4$ and $5$ are missing. The next theorem is based on and inspired by \cite[Thm.~3]{article} and \cite[Sec.~5]{Carlsson2009}. Both articles do not give a proof.

\begin{Thm}\label{Paramet thm1} The map
\begin{align}
\varphi_n^{\mathbb{F}}(\xi_0, \xi_1): S_{n}^{\mathbb{F}}(\xi_0,\xi_1) \xlongrightarrow{\sim} R_n^{\mathbb{F}}(\xi_0,\xi_1), \quad
L  \longmapsto (L_w)_{w\in V_1}\notag
\end{align}
is a set theoretic bijection with  inverse
\begin{align}
\phi_n^{\mathbb{F}}(\xi_0, \xi_1):R_n^{\mathbb{F}}(\xi_0,\xi_1) \xlongrightarrow{\sim}  S_{n}^{\mathbb{F}}(\xi_0,\xi_1) , \quad (L_w )_{w\in V_1} \longmapsto \left\langle \bigcup_{w\in V_1} L_w  \right\rangle _{A_n^{\mathbb{F}}}  \notag .
\end{align}
\end{Thm}
Before we proceed with a proof of \cref{Paramet thm1}, we give two examples for $\mathbb{F}$-relation families:
\begin{figure}[h]
\centering
\includegraphics[scale=0.6]{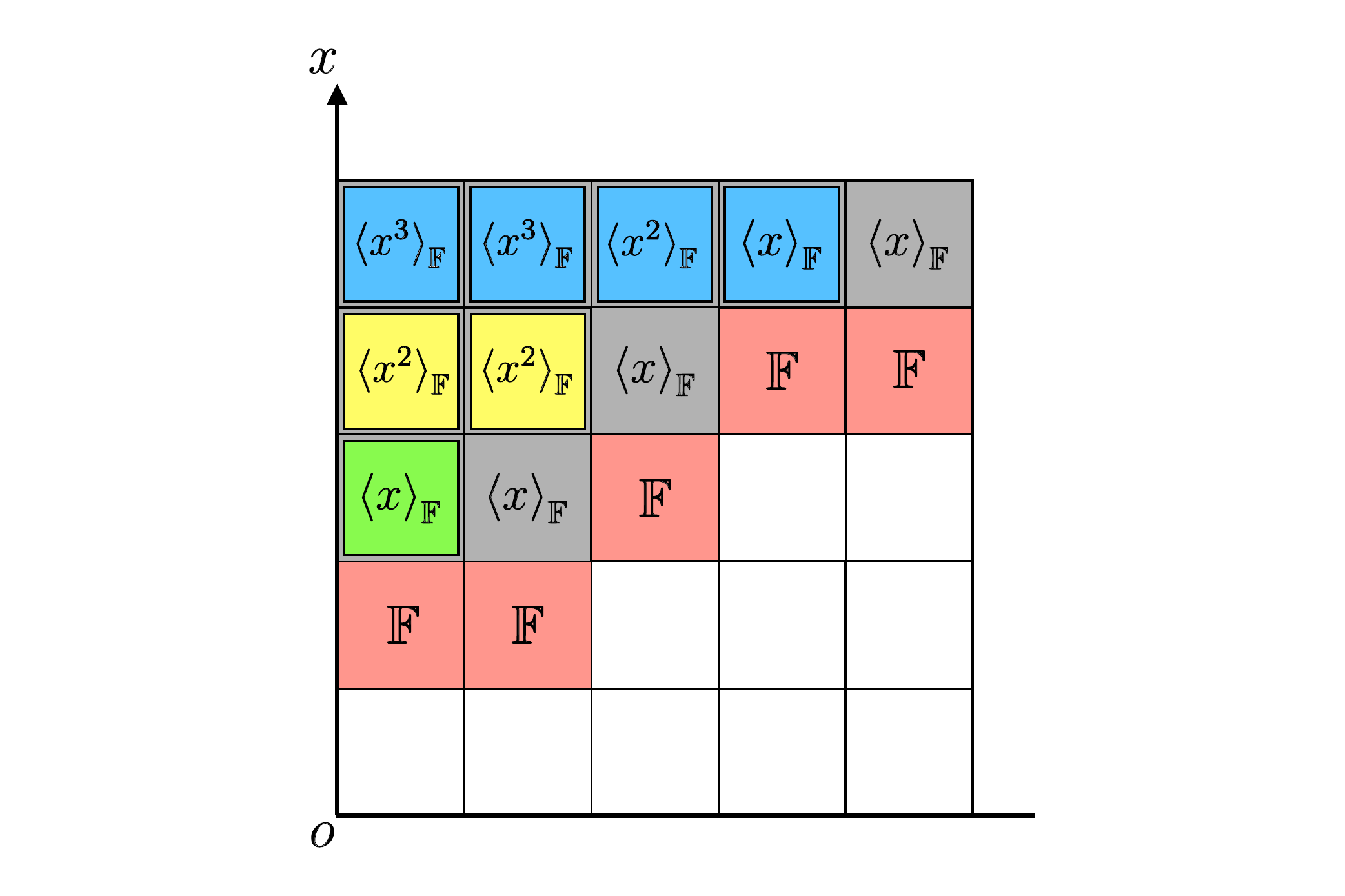} 
\caption{$\mathbb{F}$-relation family in dimension $n=1$ (illustration of \cref{Exa Rel fam1}).}
\label{fig:Relation Families1}
\end{figure}
\begin{Exa}\label{Exa Rel fam1} Let $n=1$. Consider \begin{align}
\xi_0 & := \{(1,1), (1,2), (2,1) ,(3,1), (3,2) \} \notag \\
\xi_1 &:= \{(2,1), (3,1), (4,1),(4,2)\} \notag
\end{align}
and let
\begin{align}
L_{4} & := \left\langle (x^3,0,0,0,0), (0,x^3,0,0,0),(0,0,x^2,0,0), (0,0,0,x,0)\right\rangle_{\mathbb{F}} \subseteq \mathcal{F}_1^{\mathbb{F}}(\xi_0)_{4} \notag, \\
L_{3} & := \left\langle (x^2,0,0,0,0), (0,x^2,0,0,0) \right\rangle_{\mathbb{F}} \subseteq \mathcal{F}_1^{\mathbb{F}}(\xi_0)_{3}, \notag \\
L_{2} & := \left\langle (x,0,0,0,0) \right\rangle_{\mathbb{F}} \subseteq \mathcal{F}_1^{\mathbb{F}}(\xi_0)_{2}  \notag.
\end{align}
Then $(L_j)_{j \in \{2,3,4\}} \in R_1^{\mathbb{F}}(\xi_0,\xi_1)$.
Figure \ref{fig:Relation Families1} illustrates $(L_j)_{j \in \{2,3,4\}}$. The blue blocks correspond to $L_4$, the yellow blocks correspond to $L_3$ and the green block corresponds to $L_2$. The red blocks are the "forbidden" blocks that we get from condition $4$ in $R_1^{\mathbb{F}}(\xi_0, \xi_1)$. By using Theorems \ref{classification bijection} and \ref{Paramet thm1}, we obtain a set theoretic bijection
\begin{equation}
R_1^{\mathbb{F}}(\xi_0, \xi_1)/\mathrm{Aut}(\mathcal{F}_1^{\mathbb{F}}(\xi_0) \cong I_1^{\mathbb{F}}(\xi_0, \xi_1)  \notag.
\end{equation}
\end{Exa}

\begin{figure}[h]
\centering
\includegraphics[scale=0.4]{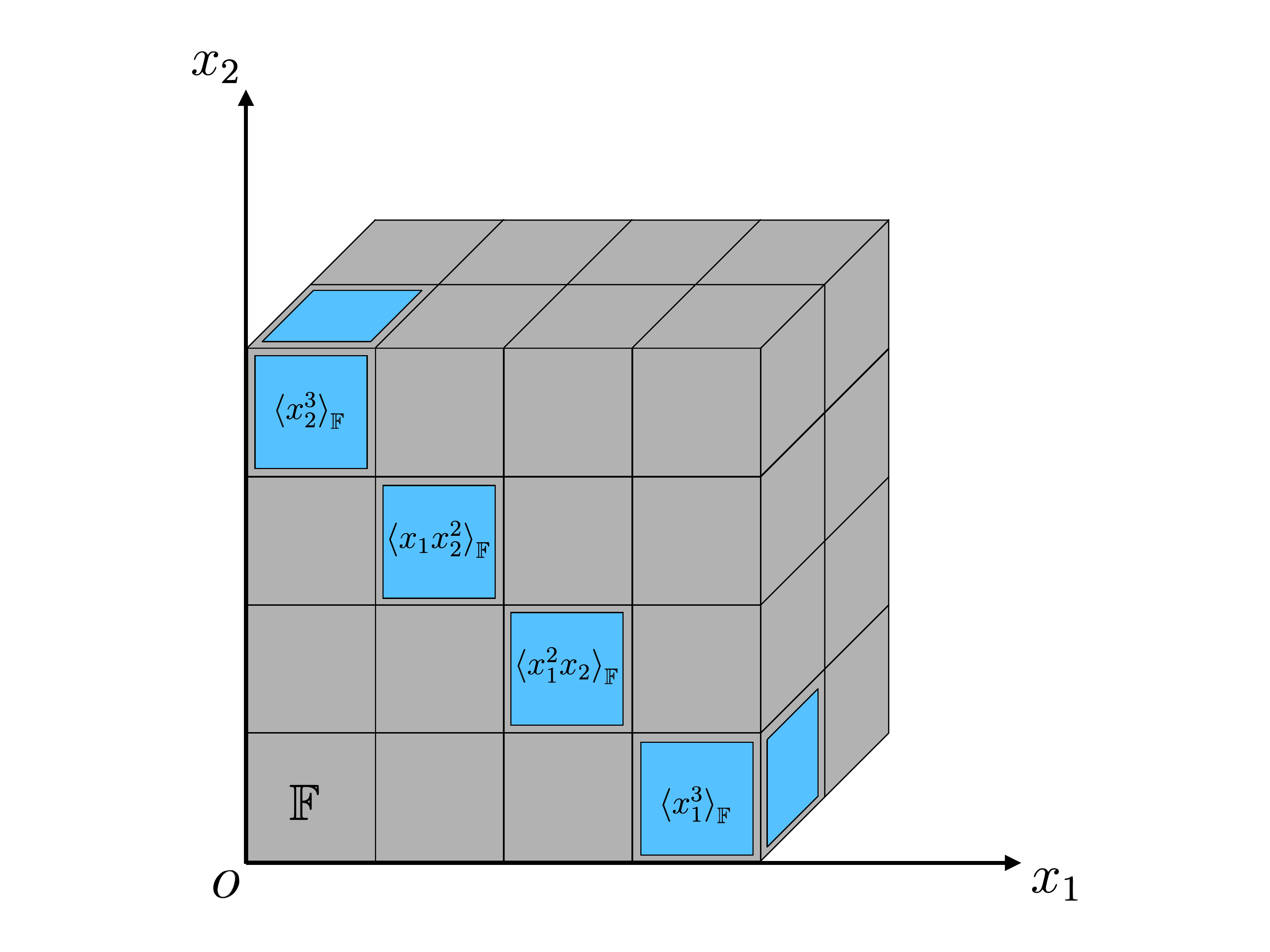} 
\caption{$\mathbb{F}$-relation family in dimension $n=2$ (illustration of \cref{Exa Rel fam2}).}
\label{fig:Relation Families2}
\end{figure}
\newpage
\begin{Exa}\label{Exa Rel fam2} This example can be found in \cite[Sec.~5.2]{article}, \cite[Sec.~5.2]{Carlsson2009} and will be also discussed later in \cref{cont inv}. Let $n=2$. As in \cref{Dim 2 Tensor Ex}, consider 
\begin{align}
\xi_0&:=\{\textcolor{black}{((0,0),1), ((0,0),2)}\}, \notag \\
\xi_1 &:=\{((3,0),1),((2,1),1),((1,2),1),((0,3),1)\}. \notag 
\end{align}
Then conditions 3-5 in $R_2^{\mathbb{F}}(\xi_0, \xi_1)$ are trivial and we obtain
\begin{equation}
R_2^{\mathbb{F}}(\xi_0, \xi_1)=\prod_{w \in V_1} \left \{L_w\subseteq \mathcal{F}_2^{\mathbb{F}}(\xi_0)_w \mid \mathrm{dim}_\mathbb{F}(L_w)=\mathrm{dim}_\mathbb{F} \left( \mathcal{F}_2^{\mathbb{F}}(\xi_1)_{w} \right)=1 \right\} \notag
\end{equation}
Since $\mathcal{F}_2^{\mathbb{F}}(\xi_0)_w \cong \mathbb{F}^2$ for all $w \in V_1$, we have
\begin{equation}
R_2^{\mathbb{F}}(\xi_0, \xi_1) \cong \mathbb{P}_1(\mathbb{F})^4 \notag
\end{equation}
as sets, where $\mathbb{P}_1(\mathbb{F})$ denotes the projective line over $\mathbb{F}$ (i.e. the set of all one-dimensional subspaces of $\mathbb{F}^2$). Figure \ref{fig:Relation Families2} illustrates $(L_w)_{w \in V_1} \in R_2^{\mathbb{F}}(\xi_0, \xi_1)$ where 
\begin{align}
&L_{(3,0)}:=\braket{(x_1^3,0)}_{\mathbb{F}} \subseteq \mathcal{F}_2^{\mathbb{F}}(\xi_0)_{(3,0)}, \notag \\  &L_{(2,1)}:=\braket{(x_1^2x_2,0)}_{\mathbb{F}} \subseteq \mathcal{F}_2^{\mathbb{F}}(\xi_0)_{(2,1)}, \notag \\
&L_{(1,2)}:=\braket{(x_1x_2^2,0)}_{\mathbb{F}} \subseteq \mathcal{F}_2^{\mathbb{F}}(\xi_0)_{(1,2)}, \notag \\ &L_{(0,3)}:=\braket{(x_2^3,0)}_{\mathbb{F}} \subseteq \mathcal{F}_2^{\mathbb{F}}(\xi_0)_{(0,3)} \notag.
\end{align}
\end{Exa}

\begin{proof}[Proof of \cref{Paramet thm1}] $\varphi_n^{\mathbb{F}}(\xi_0, \xi_1)$ is well-defined: let $L \in S_{n}^{\mathbb{F}}(\xi_0, \xi_1)$. Since $L$ is graded, condition $1$ and $3$ are satisfied. Condition $4$ is equivalent to the tensor-condition by \cref{Tensor condition}. It remains to show that condition $2$ and $5$ are satisfied. Let ${p_1: \mathcal{F}_n^{\mathbb{F}}(\xi_1) \to L}$ be a free hull of $L$ and $E_{\xi_1}:=\bigcup_{v \in V_1}E_{v}$ be the standard basis of $\mathcal{F}_n^{\mathbb{F}}(\xi_1)$ as $A_n^{\mathbb{F}}$-module where for $v \in V_1$, $E_v$ corresponds to the block $A_n^{\mathbb{F}}(v)^{\mu_1(v)}_v=\mathbb{F}^{\mu_0(v)}$.  Let $w \in V_1$ and $E_{v \preceq w}:=\bigcup_{v \preceq w }E_{v}$. Then $\mathcal{F}_n^{\mathbb{F}}(\xi_1)_w= \braket{E_{v \preceq w}}_{\mathbb{F}}$ and $p_1(E_{\xi_1})$ is a minimal set of homogeneous generators of $L$ by \cref{equiv char free hulls} and $p_1(E_{v \preceq w})$ is an $\mathbb{F}$-basis of $L_w$. We have \begin{equation}
p_1(\braket{E_{v \preceq w}}_{\mathbb{F}})=p_1(\mathcal{F}_n^{\mathbb{F}}(\xi_1)_{w})=L_w  \notag
\end{equation}
which shows that $\mathrm{dim}_{\mathbb{F}}(\mathcal{F}_n^{\mathbb{F}}(\xi_1)_w)\geq \mathrm{dim}_{\mathbb{F}}(L_w)$ and equality must hold due to minimality. Thus, condition $2$ is satisfied. It remains to show that condition $5$ is satisfied: since $\mathcal{F}_n^{\mathbb{F}}(\xi_0)$ is graded free and $L \subseteq \mathcal{F}_n^{\mathbb{F}}(\xi_0)$ is graded, multiplication by $x^{w-v} $ is an $\mathbb{F}$-vector space monomorphism from $L_v$ to $L_w$ for all $v,w \in {\mathbb{N}^n}$ with $v \preceq w$. Therefore,
\begin{align}
\mathrm{dim}_{\mathbb{F}}(L_w/x^{w-v} L_v)&=\mathrm{dim}_{\mathbb{F}}(L_w)- \mathrm{dim}_{\mathbb{F}}(L_v)=\abs{E_{u \preceq w}}- \abs{E_{u \preceq v}} \notag \\
&= \sum_{u \preceq w} \mu_1(u)- \sum_{u \preceq v} \mu_1(u) = \sum_{v \prec u \preceq w} \mu_1(u)\notag.
\end{align}
If we procced inductively we see that condition $5$ is satisfied. So, we conlude that $\varphi_n^{\mathbb{F}}(\xi_0, \xi_1)$ is well-defined. 

$\phi_n^{\mathbb{F}}(\xi_0, \xi_1)$ is well-defined: let $(L_w)_{ w\in V_1} \in R_n^{\mathbb{F}}(\xi_0, \xi_1)$ and \begin{equation}
L:= \left\langle \bigcup_{w \in V_1} L_w  \right\rangle _{A_n^{\mathbb{F}}} \subseteq \mathcal{F}_n^{\mathbb{F}}(\xi_0) \notag.
\end{equation}
Then $L \subseteq \mathcal{F}_n^{\mathbb{F}}(\xi_0)$ is a graded $A_n^{\mathbb{F}}$-submodule since it is generated by homogeneous elements (see \cref{basics about graded}).
Condition $4$ is equivalent to the tensor-condition by \cref{Tensor condition}. So, it remains to show that $L$ has type $\xi_1$:
\begin{enumerate}[leftmargin=2cm]
\item[Step $1$:] let $V_{1,1}^{\mathrm{max}} \subseteq V_1$ be the set of all maximal $v \in V_1$ with respect to $\preceq$. Then 
 $\mathrm{dim}_\mathbb{F} \left( L_w / \sum_{v \prec w} x^{w-v} L_v \right) = \mu_1(w)$ for all $w \in V_{1,1}^{\mathrm{max}}$. For all $w \in V_{1,1}^{\mathrm{max}}$, let $B_{1,w} \subseteq L_w$ with $\abs{B_{w}}=\mu_1(w)$ such that 
\begin{equation} \left \langle B_{w} \cup \sum_{v \prec w} x^{w-v} L_v \right \rangle_{\mathbb{F}}=L_w  \notag
\end{equation}
As  $\mathrm{dim}_\mathbb{F}(L_w)=\mathrm{dim}_\mathbb{F} \left( \mathcal{F}_n^{\mathbb{F}}(\xi_1)_{w} \right)$ we see that \begin{equation}\mathrm{dim}_{\mathbb{F}} \left(\sum_{v \prec w} x^{w-v} L_v \right)= \mathrm{dim}_\mathbb{F} \left( \mathcal{F}_n^{\mathbb{F}}(\xi_1)_{w} \right)- \mu_1(w) \notag
\end{equation}
for all $w \in V_{1,1}^{\mathrm{max}}$.
\item[Step $2$:] let $V_{2,1}^{\mathrm{max}} \subseteq V_1 \setminus V_{1,1}^{\mathrm{max}} $ be the set of all maximal $v \in V_1 \setminus V_{1,1}^{\mathrm{max}}$ with respect to $\preceq$. Then 
 $\mathrm{dim}_\mathbb{F} \left( L_w / \sum_{v \prec w} x^{w-v} L_v \right) = \mu_1(w)$ for all $w \in V_{2,1}^{\mathrm{max}}$. For all $w \in V_{2,1}^{\mathrm{max}}$, let $B_{w} \subseteq L_w$ with $\abs{B_{w}}=\mu_1(w)$ such that 
\begin{equation} \left \langle B_{w} \cup \sum_{v \prec w} x^{w-v} L_v \right \rangle_{\mathbb{F}}=L_w  \notag
\end{equation}
As  $\mathrm{dim}_\mathbb{F}(L_w)=\mathrm{dim}_\mathbb{F} \left( \mathcal{F}_n^{\mathbb{F}}(\xi_1)_{w} \right)$ we see that \begin{equation}\mathrm{dim}_{\mathbb{F}} \left(\sum_{v \prec w} x^{w-v} L_v \right)= \mathrm{dim}_\mathbb{F} \left( \mathcal{F}_n^{\mathbb{F}}(\xi_1)_{w} \right)- \mu_1(w) \notag
\end{equation}
for all $w \in V_{2,1}^{\mathrm{max}}$.
\end{enumerate}
If we proceed inducutively this way, there exists an $m \in \mathbb{N}$ such that $V_{m,1}^{\mathrm{max}} \neq \emptyset$ and $V_{m+1,1}^{\mathrm{max}} =\emptyset$
\begin{enumerate}[leftmargin=2cm]
\item[Step $m$:] as $V_{m+1,1}^{\mathrm{max}} =\emptyset$, we have $\mathrm{dim}_{\mathbb{F}}(L_w)=\mu_1(w)$ for all $w \in V_{m,1}^{\mathrm{max}} $. Now let $B_{w} \subseteq L_w$ be an $\mathbb{F}$-basis. Then $\abs{B_{w}}=\mu_1(w)=\mathrm{dim}_{\mathbb{F}}(L_w)$.
\end{enumerate}
By construction, we have 
 \begin{equation}\label{has type eq 1}
L_w=\left \langle \bigcup_{v \preceq w} x^{w-v} B_{v} \right \rangle_{\mathbb{F}} .
\end{equation} 
Since multiplication by $x^{w-v}$ is an $\mathbb{F}$-vector space monomorphism from $\mathcal{F}_n^{\mathbb{F}}(\xi_0)_v$ to $\mathcal{F}_n^{\mathbb{F}}(\xi_0)_w$, it holds that \begin{equation}\abs{x^{w-v} B_v}=\abs{B_v} \notag.
\end{equation}
So, we have
\begin{align}\label{has type eq 2}
\abs{\bigcup_{v \preceq w} x^{w-v}B_{v}}& \leq \sum_{v \preceq w} \abs{x^{w-v}B_{v}} = \sum_{v \preceq w} \abs{B_{v}} \notag \\
&=\sum_{v \preceq w} \mu_1(v)=\mathrm{dim}_{\mathbb{F}} \left( \mathcal{F}_n^{\mathbb{F}}(\xi_1)_{w} \right)=\mathrm{dim}_{\mathbb{F}}(L_w).
\end{align}
Now (\ref{has type eq 1}) and (\ref{has type eq 2}) imply that $\bigcup_{v \preceq w} x^{w-v}B_{v}$ is an $\mathbb{F}$-basis of $L_w$. Therefore,\begin{equation}
B:=\bigcup_{v \in V_1} B_v \notag
\end{equation}
is a finite system of homogeneous generators of $L$. Note that the union $\bigcup_{v \in V_1} B_v$ is disjoint since the $B_v$ live in different degrees.  Since \begin{equation}
\abs{B}= \sum_{v \in V_1} \mu_1(v)=\abs{\xi_1}=\mathrm{dim}_{\mathbb{F}}(\mathbb{F} \otimes_{A_n^{\mathbb{F}}} L)=\mathrm{dim}_{\mathbb{F}}(L/\mathfrak{m}_n^{\mathbb{F}} L) \notag
\end{equation}
and since $\pi(B)$ generates $L/\mathfrak{m}_n^{\mathbb{F}} L$ (where $\pi: L \to L/\mathfrak{m}_n^{\mathbb{F}} L$ denotes the canonical projection), we conclude that $\pi(B)$ is an $\mathbb{F}$-basis of $L/\mathfrak{m}_n^{\mathbb{F}} L$. Hence, $B$ is minimal by \cref{graded Naka}.

Recall that $E_{\xi_1}=\bigcup_{v \in V_1} E_v$ is the standard basis of $\mathcal{F}_n^{\mathbb{F}}(\xi_1)$ as an $n$-graded $A_n^{\mathbb{F}}$-module. We have $\abs{E_v}=\mu_1(v)=\abs{B_v}$. So, there are set theoretic bijections $\iota_v: E_v \xrightarrow{\sim} B_v$. Thus, 
\begin{equation}
p_1: \mathcal{F}_n^{\mathbb{F}}(\xi_1) \longrightarrow L, \quad  E_v \ni b \longmapsto \iota_v(b) \notag
\end{equation}
defines a surjective morphism of $n$-graded $A_n^{\mathbb{F}}$-modules. Since $B$ is minimal, $(\mathcal{F}_n^{\mathbb{F}}(\xi_1),p_1)$ is a free hull of $L$ by \cref{equiv char free hulls} and $L$ has type $\xi_1$.

$\varphi_n^{\mathbb{F}}(\xi_0, \xi_1)$ and $\phi_n^{\mathbb{F}}(\xi_0, \xi_1)$ are inverse to each other:
Let $L \in S_n^{\mathbb{F}}(\xi_0, \xi_1)$. Since, $L$ has type $\xi_1$ the generators of $L$ are located in $\bigcup_{w \in V_1} L_w$ and we obtain
\begin{align}
\phi_n^{\mathbb{F}}(\xi_0, \xi_1)(\varphi_n^{\mathbb{F}} (\xi_0, \xi_1)(L))=\phi_n^{\mathbb{F}}(\xi_0, \xi_1)((L_w)_{w \in V_1})= \left \langle \bigcup_{w \in V_1} L_w  \right \rangle_{A_n^{\mathbb{F}}} =L \notag.
\end{align}
Now let $(L_w)_{w \in V_1} \in R_n^{\mathbb{F}} (\xi_0, \xi_1)$ and let
\begin{equation}
N:=\left \langle \bigcup_{w \in V_1} L_w  \right \rangle_{A_n^{\mathbb{F}}}  \notag.
\end{equation}
Let $w \in V_1$. Then $L_w \subseteq N_w$. Since $\phi_n^{\mathbb{F}}(\xi_0, \xi_1)$ is well-defined, we know that $N$ has type $\xi_1$. So, there exists a surjective graded $A_n^{\mathbb{F}}$-module homomorphism $p_1: \mathcal{F}_n^{\mathbb{F}}(\xi_1) \to N$. Since $(L_w)_{w \in V_1} \in R_n^{\mathbb{F}}(\xi_0, \xi_1)$, we have $\mathrm{dim}_{\mathbb{F}} (\mathcal{F}_n^{\mathbb{F}}(\xi_1)_w)= \mathrm{dim}_{\mathbb{F}} (L_w)$. Thus, we obtain
\begin{equation}
\mathrm{dim}_{\mathbb{F}} (L_w) \leq \mathrm{dim}_{\mathbb{F}} (N_w) \leq  \mathrm{dim}_{\mathbb{F}} (\mathcal{F}_n^{\mathbb{F}}(\xi_1)_w)= \mathrm{dim}_{\mathbb{F}} (L_w) \notag
\end{equation}
which shows that $\mathrm{dim}_{\mathbb{F}} (L_w)= \mathrm{dim}_{\mathbb{F}} (N_w)$. Since $L_w \subseteq N_w$, we have $L_w=N_w$. Therefore,
\begin{equation}
\varphi_n^{\mathbb{F}}(\xi_0, \xi_1)(\phi_n^{\mathbb{F}}(\xi_0, \xi_1)((L_w)_{w \in V_1}))=\varphi_n^{\mathbb{F}}(\xi_0, \xi_1)(N)=(L_w)_{w \in V_1} \notag.
\end{equation}
So, $\varphi_n^{\mathbb{F}}(\xi_0, \xi_1)$ and $\phi_n^{\mathbb{F}}(\xi_0, \xi_1)$ are inverse to each other
\end{proof}

\begin{Defiprop} $\mathrm{Aut}(\mathcal{F}_n^{\mathbb{F}}(\xi_0))$ acts as a group on $R_n^{\mathbb{F}}(\xi_0,\xi_1)$ where for $(L_w  )_{w\in V_1} \in R_n^{\mathbb{F}}(\xi_0,\xi_1)$ and $f \in \mathrm{Aut}(\mathcal{F}_n^{\mathbb{F}}(\xi_0))$, \begin{equation} f \cdot (L_w )_{w\in V_1}:=\left  ( f (L_w) \right )_{w\in V_1} \notag.
\end{equation}
\end{Defiprop}
\begin{proof}
We have to check that conditions $1$-$5$ in $R_n^{\mathbb{F}}(\xi_0, \xi_1)$ are preserverd under this action. So, let $(L_w)_{w \in V_1} \in R_n^{\mathbb{F}}(\xi_0, \xi_1)$ and $f \in \mathrm{Aut}(\mathcal{F}_n^{\mathbb{F}}(\xi_0)$. Then 
\begin{equation}
L:= \left\langle \bigcup_{w \in V_1} L_w  \right\rangle _{A_n^{\mathbb{F}}}  = \varphi_n^{\mathbb{F}}(\xi_0, \xi_1) ((L_w)_{w \in V_1})\in S_{n}^{\mathbb{F}}(\xi_0, \xi_1) \notag
\end{equation}
by \cref{Paramet thm1}. We have $f(L) \in S_{n}^{\mathbb{F}}(\xi_0, \xi_1)$ and since $f$ is a graded automorphism, we have $f(L)_w=f(L_w)$ for all $w \in {\mathbb{N}^n}$. Using \cref{Paramet thm1}, we therefore conclude that 
\begin{equation}
(f(L_w))_{w \in V_1}=(f(L)_w)_{w \in V_1}=\phi_n^{\mathbb{F}}(\xi_0, \xi_1)(f(L)) \in  R_n^{\mathbb{F}}(\xi_0, \xi_1) \notag.
\end{equation}
\end{proof}
\begin{Thm}\label{Paramet thm} $\varphi_n^{\mathbb{F}}(\xi_0, \xi_1)$ and $\phi^{\mathbb{F}}(\xi_0, \xi_1)$ are $\mathrm{Aut}(\mathcal{F}_n^{\mathbb{F}}(\xi_0))$-equivariant and they induce mutually inverse set theoretic bijections
\begin{equation}
\overline{\varphi}_n^{\mathbb{F}}(\xi_0, \xi_1): S_{n}^{\mathbb{F}}(\xi_0,\xi_1)/\mathrm{Aut}(\mathcal{F}_n^{\mathbb{F}}(\xi_0)) \longrightarrow R_n^{\mathbb{F}}(\xi_0,\xi_1)/\mathrm{Aut}(\mathcal{F}_n^{\mathbb{F}}(\xi_0)) \notag
\end{equation}
and
\begin{align}
\overline{\phi}_n^{\mathbb{F}}(\xi_0, \xi_1):R_n^{\mathbb{F}}(\xi_0,\xi_1)/\mathrm{Aut}(\mathcal{F}_n^{\mathbb{F}}(\xi_0)) \longrightarrow  S_{n}^{\mathbb{F}}(\xi_0,\xi_1)/\mathrm{Aut}(\mathcal{F}_n^{\mathbb{F}}(\xi_0)) \notag.
\end{align}
Thus, 
\begin{equation} I_{n}^{\mathbb{F}}(\xi_0,\xi_1)  \cong R_{n}^{\mathbb{F}}(\xi_0,\xi_1)/  \mathrm{Aut}(\mathcal{F}_n^{\mathbb{F}}(\xi_0)) \notag
\end{equation}
as sets.
\end{Thm}
\begin{proof} Both $\varphi_n^{\mathbb{F}}(\xi_0, \xi_1)$ and $\phi_n^{\mathbb{F}}(\xi_0, \xi_1)$ are $\mathrm{Aut}(\mathcal{F}_n^{\mathbb{F}}(\xi_0))$-equivariant by definition of the $\mathrm{Aut}(\mathcal{F}_n^{\mathbb{F}}(\xi_0))$-action. The rest is clear by \cref{Paramet thm1} and \cref{classification bijection}.
\end{proof}

\subsection{The automorphism group}
Recall that $\xi_0=(V_0, \mu_0)$ is a finite $n$-dimensional mulitset. In this section, we give an explicit description of $\mathrm{Aut}(\mathcal{F}_n^{\mathbb{F}}(\xi_0))$ in terms of transformation matrices.

Our next definition, proposition is from \cite[Lem.~3]{article} (see also \cite[Sec.~4.3]{Carlsson2009}).
\begin{Defiprop} Define 
\begin{equation}
\mathrm{Aut}_{\preceq}^{\mathbb{F}}(\xi_0)\subseteq \mathrm{Aut}\left(\bigoplus_{v \in V_0} \mathbb{F}^{\mu_0(v)} \right) \notag
\end{equation}
as the subgroup of all $\mathbb{F}$-vector space automorphisms \begin{equation}
f:\bigoplus_{v \in V_0} \mathbb{F}^{\mu_0(v)} \xlongrightarrow{\sim} \bigoplus_{v \in V_0} \mathbb{F}^{\mu_0(v)} \notag
\end{equation}
such that for all $v \in V_0$, 
\begin{equation}
f \left( \mathbb{F}^{\mu_0(v)} \right) \subseteq \bigoplus_{w \preceq v} \mathbb{F}^{\mu_0(w)} \notag.
\end{equation}
\end{Defiprop}

\begin{proof} We have to show that $\mathrm{Aut}_{\preceq}^{\mathbb{F}}(\xi_0)\subseteq \mathrm{Aut}\left(\bigoplus_{v \in V_0} \mathbb{F}^{\mu_0(v)} \right)$ is a subgroup. We clearly have $ \mathrm{id} \in\mathrm{Aut}_{\preceq}^{\mathbb{F}}(\xi_0)$ and for $f,g \in \mathrm{Aut}_{\preceq}^{\mathbb{F}}(\xi_0)$, $f \circ g \in \mathrm{Aut}_{\preceq}^{\mathbb{F}}(\xi_0)$. It remains to show that $f^{-1} \in \mathrm{Aut}_{\preceq}^{\mathbb{F}}(\xi_0)$. Let $v \in V_0$. It holds that \begin{equation}f \left( \mathbb{F}^{\mu_0(v)} \right) \subseteq \bigoplus_{w \preceq v} \mathbb{F}^{\mu_0(w)} \notag. 
\end{equation}
Hence, \begin{equation}
f \left( \bigoplus_{w \preceq v}  \mathbb{F}^{\mu_0(w)} \right) \subseteq \bigoplus_{w \preceq v}  \mathbb{F}^{\mu_0(w)} \notag.
\end{equation}
Since $f$ is an automorphism of $\mathbb{F}$-vector spaces, this implies that \begin{equation}
f \left( \bigoplus_{w \preceq v}  \mathbb{F}^{\mu_0(w)}  \right) = \bigoplus_{w \preceq v}  \mathbb{F}^{\mu_0(w)} \notag. 
\end{equation}
Thus, \begin{equation}
f^{-1} \left(\bigoplus_{w \preceq v} \mathbb{F}^{\mu_0(w)} \right) =  \bigoplus_{w \preceq v} \mathbb{F}^{\mu_0(w)} \notag
\end{equation} 
and in particular 
\begin{equation}
f^{-1} \left(\mathbb{F}^{\mu_0(v)} \right) \subseteq  \bigoplus_{w \preceq v} \mathbb{F}^{\mu_0(w)} \notag.
\end{equation}
\end{proof}

 Before we proceed, we need to introduce further notations:
\begin{Defi}\label{def of taus} Recall that $\xi_0=(V_0, \mu_0)$. Let $v \in {\mathbb{N}^n}$ and $w \in V_0$:
\begin{enumerate}
\item if $v \succeq w$, we have $A_n^{\mathbb{F}}(w)^{\mu_0(w)}_v= \braket{x^{v-w}}_{\mathbb{F}}^{\mu_0(w)}$. So, 
\begin{equation}
 \{(x^{v-w},0, \dots, 0), (0,x^{v-w},0, \dots, 0), \dots, (0, \dots, 0,x^{v-w}) \} \notag
\end{equation}
is an $\mathbb{F}$-basis of $A_n^{\mathbb{F}}(w)^{\mu_0(w)}_v$. Now let
\begin{equation}
\tau(w)_v:   A_n^{\mathbb{F}}(w)^{\mu_0(w)}_v \xlongrightarrow{\sim}  \mathbb{F}^{\mu_0(w)} \notag
\end{equation}
be the $\mathbb{F}$-linear isomorphism which maps the basis element \begin{equation}
(0, \dots ,0,\underset{j}{x^{v-w}},0, \dots,0) \in A_n^{\mathbb{F}}(w)^{\mu_0(w)}_v \notag
\end{equation}
to the $j$-th standard basis vector \begin{equation}(0, \dots, 0,\underset{j}{1},0, \dots, 0) \in \mathbb{F}^{\mu_0(w)} \notag.
\end{equation}
\item if $v \nsucceq w$, let 
\begin{equation}\tau(w)_v: 0= A_n^{\mathbb{F}}(w)_v^{\mu_0(w)} \longrightarrow \mathbb{F}^{\mu_0(w)} \notag
\end{equation}
be the zero map. 
\item For $v \in \mathbb{N}^n$, we obtain an $\mathbb{F}$-vector space monomorphism \begin{equation}\tau_v: \mathcal{F}_n^{\mathbb{F}}(\xi_0)_v \longrightarrow \bigoplus_{w \in V_0} \mathbb{F}^{\mu_0(w)}, \quad x \longmapsto (\tau(w)_v(x))_{w\in V_0} \notag.
\end{equation}
\end{enumerate}
\end{Defi}
\begin{Exa} \, \begin{enumerate}
\item Consider $\xi_0=\{(0,1), (1,1)\}$. Let
\begin{equation}
\mathcal{F}:=\mathcal{F}_1^{\mathbb{F}}(\xi_0)=A_1^{\mathbb{F}} \oplus A_1^{\mathbb{F}}(1) \notag.
\end{equation}
Then $(x,1) \in \mathcal{F}_1$ and \begin{equation}
\tau_1(x,1)=(1,1)\in \mathbb{F}^2 \notag
\end{equation}
\item Consider $\xi_0=\{((0,0),1), ((1,2),1), ((1,2),2)\}$. Let
\begin{equation}\mathcal{F}:=\mathcal{F}_2^{\mathbb{F}}(\xi_0)=A_2^{\mathbb{F}} \oplus A_2^{\mathbb{F}}(1,2)\oplus A_2^{\mathbb{F}}(1,2) \notag.
\end{equation}
Then $(x_1 x_2^2,0,1) \in \mathcal{F}_{(1,2)}$ and \begin{equation}
\tau_{(1,2)}(x_1 x_2^2,0,1)=(1,0,1) \in \mathbb{F}^3 \notag.
\end{equation}
\end{enumerate}
\end{Exa}

\begin{Prop}[{{\cite[Lem.~3]{article}, \cite[Thm.~5]{Carlsson2009}}}]\label{first Autgroup isom}  There exists an isomorphism of groups 
\begin{equation} \nu_n^{\mathbb{F}}(\xi_0):\mathrm{Aut}(\mathcal{F}_n^{\mathbb{F}}(\xi_0)) \xlongrightarrow{\sim} \mathrm{Aut}_{\preceq}^{\mathbb{F}}(\xi_0)  \notag.
\end{equation} 
\end{Prop}
This result is stated (without a proof) in \cite[Lem.~3]{article}. In \cite[Thm.~5]{Carlsson2009}, the authors give a proof where they also try to construct homomorphisms that are inverse to each other. But their homomorphism from $\mathrm{Aut}(\mathcal{F}_n^{\mathbb{F}}(\xi_0))$ to $\mathrm{Aut}_{\preceq}^{\mathbb{F}}(\xi_0)$ which sends $\psi$ to $\mathrm{id}_{\mathbb{F}} \otimes_{A_n^{\mathbb{F}}} \psi$ has image contained in $\mathrm{Aut}(\mathcal{V}_n^{\mathbb{F}}(\xi_0))$. $\mathrm{Aut}(\mathcal{V}_n^{\mathbb{F}}(\xi_0))\subseteq \mathrm{Aut}_{\preceq}^{\mathbb{F}}(\xi_0)$ is a subgroup and the inclusion is proper in many cases as we will see later (see \cref{proper incl}). Moreover, \cref{Exa tensoring dangerous} shows that \begin{equation}\mathrm{id}_{\mathbb{F}} \otimes_{A_n^{\mathbb{F}}} (-):\mathrm{Aut}(\mathcal{F}_n^{\mathbb{F}}(\xi_0)) \to \mathrm{Aut}_{\preceq}^{\mathbb{F}}(\xi_0) \notag
\end{equation}
is generally not injective. 

Our proof is inspired by \cite{Carlsson2009} but goes a different way for the inverse homomorphism.
\begin{proof} Let $\mathcal{F}:=\mathcal{F}_n^{\mathbb{F}}(\xi_0)$. Let $g \in \mathrm{Aut}(\mathcal{F})$. For every $v \in V_0$, 
\begin{equation}
 g_{v}:\mathcal{F}_{v} \xlongrightarrow{\sim} \mathcal{F}_{v}, \quad x \longmapsto g(x) \notag
\end{equation}
is an automorphism of $\mathbb{F}$-vector spaces. The map $\nu=\nu_n^{\mathbb{F}}(\xi_0)$ is constructed as follows. Let $v \in V_0$. We have
 \begin{equation}
\mathcal{F}_{v}=   \bigoplus_{w \in V_0} A_n^{\mathbb{F}}(w)_v^{\mu_0(w)}  \notag.
\end{equation}
If $v \succeq w$, we have $A_n^{\mathbb{F}}(w)_v^{\mu_0(w)}=\braket{x^{v-w}}_{\mathbb{F}}^{\mu_0(w)}$ and if $v \nsucceq w$, we have $A_n^{\mathbb{F}}(w)_v^{\mu_0(w)}=0$. So, we may identify
\begin{equation}
\mathcal{F}_{v}=   \bigoplus_{w \preceq v} \braket{x^{v-w}}_{\mathbb{F}}^{\mu_0(w)}  \notag.
\end{equation}
Now
\begin{equation}
\nu_{v}(g):=\tau_{v} \circ g_{v}\circ \tau_{v}^{-1}: \bigoplus_{w \preceq v} \mathbb{F}^{\mu_0(w)} \xlongrightarrow{\sim} \bigoplus_{w \preceq v} \mathbb{F}^{\mu_0(w)} \notag
\end{equation}
defines an automorphism of $\mathbb{F}$-vector spaces for all $v \in V_0$. Define \begin{equation}\nu(g):=\bigoplus_{v \in V_0} \nu_v(g) \notag.
\end{equation}
Then $\nu(g)\left( \bigoplus_{w \preceq v} \mathbb{F}^{\mu_0(v)} \right) = \bigoplus_{w \preceq v} \mathbb{F}^{\mu_0(w)} $
and in particular, $\nu(g)( \mathbb{F}^{\mu_0(w)}) \subseteq \bigoplus_{w \preceq v} \mathbb{F}^{\mu_0(w)} $.
$\nu$ is clearly injective.

Conversely, if $f \in \mathrm{Aut}_{\preceq}^{\mathbb{F}}(\xi_0)$, then $f \left( \mathbb{F}^{\mu_0(v)} \right) \subseteq \bigoplus_{w \preceq v}\mathbb{F}^{\mu_0(w)}$ for all $v \in V_0$ implies that $f$ induces $\mathbb{F}$-vector space automorphisms
\begin{equation}
f_v:  \bigoplus_{w \preceq v} \mathbb{F}^{\mu_0(w)} \xlongrightarrow{\sim} \bigoplus_{w \preceq v} \mathbb{F}^{\mu_0(w)}, \quad x \longmapsto f(x) \notag
\end{equation}
for all $v \in V_0$. Thus, we  obtain $\mathbb{F}$-vector  space automorphisms
\begin{equation}
\widetilde{f}_v:= \tau_v^{-1} \circ f_v \circ \tau_v: \mathcal{F}_{v} \xlongrightarrow{\sim} \mathcal{F}_{v} \notag.
\end{equation}
for all $v \in V_0$. For $y=\sum_{w \in {\mathbb{N}^n}} y_w \in \mathcal{F}$, we can write $y_w$ uniquely as \begin{equation} y_w= \left( \sum_{i=1}^{\mu_0(v)} \lambda_{i,w,v} x^{w-v} e_{i,v} \right)_{v \in V_0} \notag
\end{equation}
where $\lambda_{i,w,v} \in \mathbb{F}$, where we set $x^{w-v}=0$ if $w \nsucceq v$ and where $e_{i,v}$ denotes the $i$-th standard basis vector
\begin{equation}
e_{i,v}=(0, \dots, 0, \underset{i}{1}, 0, \dots, 0 ) \notag.
\end{equation}
of length $\mu_0(v)$. Then
\begin{align}
\widetilde{f}: \mathcal{F} & \xlongrightarrow{\sim} \mathcal{F}, \notag \\
 \sum_{w \in {\mathbb{N}^n}}\left( \sum_{i=1}^{\mu_0(v)} \lambda_{i,w,v} x^{w-v} e_{i,v} \right)_{v \in V_0} &\longmapsto \sum_{w \in {\mathbb{N}^n}}\left( \sum_{i=1}^{\mu_0(v)} \lambda_{i,w,v} x^{w-v} f(e_{i,v}) \right)_{v \in V_0} \notag
\end{align}
is a graded automorphism and we have $\nu(\widetilde{f})=f$ by construction. This shows that $\nu$ is surjective. That $\nu$ is a group homomorphism is due to the fact that for $f,g \in \mathrm{Aut}(F)$, 
\begin{align}
\nu_{v}(g  \circ f)&=\tau_{v} \circ (g\circ f)_v \circ \tau_{v}^{-1} \notag \\
&=\tau_{v} \circ g_v \circ f_v \circ \tau_{v}^{-1} \notag \\
&=\tau_{v} \circ g_v \circ \tau_{v}^{-1} \circ \tau_{v} \circ f_v \circ \tau_{v}^{-1} \notag \\
&=\tau_{v} \circ g \circ \tau_{v}^{-1} \circ \tau_{v} \circ f \circ \tau_{v}^{-1} \notag \\ 
&= \nu_{v}(g) \circ \nu_{v}(f) \notag
\end{align}
for all $v \in V_0$.
\end{proof}

In \cite{Carlsson2009} and \cite{article}, the authors do not give an explicit description of the transformation matrices of elements in $\mathrm{Aut}_{\preceq}^{\mathbb{F}}(\xi_0)$. The goal for the rest of this subsection is to give such an explicit description. For technical reasons, we fix an enumeration of \begin{equation}
V_0=\{v_{1}, \dots, v_{s_0}\} \notag.
\end{equation}
Note that this enumeration has nothing to do with the ordering $\preceq$ on $V_0$ except only if $n=1$. In this case, we may assume that $v_{i} < v_{j}$ if $i<  j$. For the following, let \begin{equation}m_0:=\abs{\xi_0}=\sum_{v \in V_0} \mu_0(v) \notag.
\end{equation}
We identify
\begin{equation}
\bigoplus_{v \in V_0} \mathbb{F}^{\mu_0(v)}=\mathbb{F}^{m_0} \notag
\end{equation}
along our enumeration of $V_0=\{v_{1}, \dots, v_{s_0}\}$. Now $\tau_v$ (see \cref{def of taus}) defines an $\mathbb{F}$-vector space monomorphism
\begin{equation}
\tau_v: \mathcal{F}_n^{\mathbb{F}}(\xi_0)_v \longhookrightarrow \mathbb{F}^{m_0} \notag
\end{equation}
for all $v \in \mathbb{N}^n$. 

\begin{Defi}\label{transformation matrix defi} Any $M \in \mathrm{GL}_{m_0}(\mathbb{F})$ can be written as
\begin{equation}
M=\begin{pmatrix}
M(v_{1},v_{1}) & \dots &M(v_{1},v_{s_0}) \\
\vdots & & \vdots \\
M(v_{s_0},v_{1}) & \dots & M(v_{s_0},v_{s_0}) \notag
\end{pmatrix}
\end{equation}
where $M(v_{i},v_{j}) \in M_{\mu_0(v_{i}) \times \mu_0(v_{j})}(\mathbb{F})$ for $1\leq i,j \leq s_0$.
Let \begin{equation}
Z(\xi_0):=\{(i,j) \in \{1, \dots, s_0\}^2 \mid v_i \npreceq v_j \} \notag 
\end{equation}
and define
\begin{equation}
\mathrm{GL}^{\mathbb{F}}_{\preceq}(\xi_0):= \left \{ M\in \mathrm{GL}_{m_0}(\mathbb{F}) \mid M(v_{i},v_{j})=0 \, \forall (i,j) \in Z(\xi_0)  \right \} \notag.
\end{equation} 
\end{Defi}

\begin{Prop}\label{second Autgroup isom}
The assignment
\begin{equation} \omega_n^{\mathbb{F}}(\xi_0):\mathrm{Aut}_{\preceq}^{\mathbb{F}}(\xi_0)  \xlongrightarrow{\sim} \mathrm{GL}_{m_0}(\mathbb{F}), \quad  f \longmapsto \mathrm{Mat}_{E_{m_0}}(f) \notag
\end{equation} 
is an injective group homomorphism where $E_{m_0}$ denotes the standard basis of $\mathbb{F}^{m_0}$ and $\mathrm{Mat}_{E_{m_0}}(f)$ the transformation matrix of $f$ with respect to $E_{m_0}$. We have \begin{equation}
\mathrm{im}(\omega_n^{\mathbb{F}}(\xi_0))=\mathrm{GL}^{\mathbb{F}}_{\preceq}(\xi_0) \notag.
\end{equation}
So, $\mathrm{GL}^{\mathbb{F}}_{\preceq}(\xi_0)$ 
is a group in particular and $\psi_n^{\mathbb{F}}(\xi_0):= \omega_n^{\mathbb{F}}(\xi_0) \circ \nu_n^{\mathbb{F}}(\xi_0)$ defines an isomorphism of groups
\begin{equation}
\psi_n^{\mathbb{F}}(\xi_0): \mathrm{Aut}(\mathcal{F}_n^{\mathbb{F}}(\xi_0)) \xlongrightarrow{\sim} \mathrm{GL}_{\preceq}^{\mathbb{F}}(\xi_0) \notag.
\end{equation}
\end{Prop} 

\begin{Rem}\label{proper incl}
Note that 
\begin{align}
\mathrm{Aut}(\mathcal{V}_n^{\mathbb{F}}(\xi_0)) \cong \prod_{v \in V_0} \mathrm{GL}_{\mu_0(v)} (\mathbb{F}) \subseteq \mathrm{GL}_{\preceq}^{\mathbb{F}}(\xi_0) \notag
\end{align}
and that $\prod_{v \in V_0} \mathrm{GL}_{\mu_0(v)} (\mathbb{F}) \subseteq \mathrm{GL}_{\preceq}^{\mathbb{F}}(\xi_0) $ is a proper inclusion in many cases which is illustrated in Examples \ref{Exa Aut 1dim}, \ref{dim one Aut Exa}, \ref{dim 2 Aut Exa1} and \ref{dim 2 Aut Exa2}.
\end{Rem}

Before we proceed with a proof of \cref{second Autgroup isom}, we give some examples.

\begin{Exa}\label{Exa Aut 1dim} Let $n=1$ and consider $\xi_0 = \{(1,1), (2,1) \} $. Then $m_0=2$. Let $v_{1}:=1$ and $v_{2}:=2$. So, we have $Z(\xi_0)=\{(2,1)\}$. Hence, $M(v_{2}, v_{1})=0$ and we therefore obtain
\begin{align}
\mathrm{GL}_{\preceq}^{\mathbb{F}}(\xi_0)= \left \{  \begin{pmatrix}
a & b \\ 0 & c \end{pmatrix}   \in \mathrm{GL}_2(\mathbb{F} )
 \right \} \notag.
\end{align}
\end{Exa}

\begin{Exa}\label{dim one Aut Exa} Let $n=1$ and consider \begin{equation}\xi_0 := \{(1,1), (1,2),(2,1),(3,1),(3,2)
 \} \notag.
\end{equation}
Then $m_0=5$. Let $v_{1}=1$, $v_{2}=2$ and $v_{3}=3$. Then 
\begin{equation} Z(\xi_0)=\{(2,1), (3,1), (3,2)\} \notag.
\end{equation}
So, we have \begin{equation}
M(v_{2}, v_{1})= \begin{pmatrix} \textcolor{cyan}{0} & \textcolor{cyan}{0}
 \end{pmatrix}, \quad M(v_{3}, v_{1})=\begin{pmatrix} \textcolor{Blue}{0} & \textcolor{Blue}{0} \\ \textcolor{Blue}{0} & \textcolor{Blue}{0}
 \end{pmatrix}, \quad M(v_{3}, v_{2})=\begin{pmatrix} \textcolor{RedOrange}{0}  \\ \textcolor{RedOrange}{0}
 \end{pmatrix} \notag
\end{equation}
and we therefore obtain
\begin{equation}
\mathrm{GL}_{\preceq}^{\mathbb{F}}(\xi_0)= \left \{ \begin{pmatrix}
\textcolor{Black}{a_{11}} & \textcolor{Black}{a_{12}} & \textcolor{ForestGreen}{a_{13}} & \textcolor{Sepia}{a_{14}} & \textcolor{Sepia}{a_{15}} \\
\textcolor{Black}{a_{21}} & \textcolor{Black}{a_{22}} & \textcolor{ForestGreen}{a_{23}} & \textcolor{Sepia}{a_{24}} & \textcolor{Sepia}{a_{25}} \\
\textcolor{Cyan}{0} & \textcolor{Cyan}{0} & \textcolor{Fuchsia}{a_{33}} & \textcolor{YellowOrange}{a_{34}}& \textcolor{YellowOrange}{a_{35}} \\
\textcolor{Blue}{0} & \textcolor{Blue}{0} & \textcolor{RedOrange}{0} & \textcolor{Gray}{a_{44}} & \textcolor{Gray}{a_{45}} \\
\textcolor{Blue}{0} & \textcolor{Blue}{0} & \textcolor{RedOrange}{0} & \textcolor{Gray}{a_{54}} & \textcolor{Gray}{a_{55}}
\end{pmatrix} \in \mathrm{GL}_{5}(\mathbb{F}) \right \} \notag
\end{equation}
\end{Exa}

\begin{Exa}\label{Aut dim 2} Consider $\xi_0 = \{(v,1)\} $ with $v \in \mathbb{N}^n$. Then $m_0=\mu_0(v)$. Let $v_{1}:=v$. Then $Z(\xi_0)=\emptyset$ and therefore
\begin{align}
\mathrm{GL}_{\preceq}^{\mathbb{F}}(\xi_0)=\mathrm{GL}_{\mu_0(v)} (\mathbb{F}) \notag.
\end{align}
\end{Exa}
\newpage
\begin{Exa}\label{dim 2 Aut Exa1} Let $n=2$ and consider $\xi_0 = \{((1,0), 1), ((0,1), 2)  \} $. Then $m_0=2$. Let $v_{1}:=(1,0)$ and $v_2=(0,1)$. Then
\begin{equation} Z(\xi_0)=\{(1,2), (2,1)\} \notag
\end{equation} 
and we therefore obtain $M(v_1, v_2)=M(v_2, v_1)=0$ which shows that
\begin{align}
\mathrm{GL}_{\preceq}^{\mathbb{F}}(\xi_0)=\left\{ \begin{pmatrix}
a & 0 \\
0 & b
\end{pmatrix} \in
\mathrm{GL}_2 (\mathbb{F}) \right\} \notag.
\end{align}
\end{Exa}

\begin{Exa}\label{dim 2 Aut Exa2} Let $n=2$ and consider
\begin{equation}
\xi_0:=\{((0,3),1),((1,3),1), ((1,3),2), ((0,4),1) \} \notag.
\end{equation}
Then $m_0=4$. Let \begin{align}
 v_{1}:=(0,3), \quad v_{2}:=(1,3), \quad 
v_{3}:=(0,4) \notag .
\end{align}
Then
\begin{equation} \{ (i,j) \in \{1,2,3\}^2 \mid v_i \npreceq v_j \}=\{(2,1), (2,3), (3,1), (3,2)\} \notag.
\end{equation}
Hence, 
\begin{equation}  M(v_{2}, v_{1})= \begin{pmatrix} \textcolor{Cyan}{0} \\ \textcolor{Cyan}{0} \end{pmatrix}, \quad M(v_{2}, v_{3})= \begin{pmatrix} \textcolor{YellowOrange}{0} \\ \textcolor{YellowOrange}{0} \end{pmatrix}, \quad  M(v_{3}, v_{1})=\begin{pmatrix} \textcolor{Fuchsia}{0}\end{pmatrix}, \quad  M(v_{3},v_{2})=\begin{pmatrix}
\textcolor{Blue}{0} & \textcolor{Blue}{0}
\end{pmatrix} \notag
\end{equation}
 and we obtain
\begin{equation}
\mathrm{GL}_{\preceq}^{\mathbb{F}}(\xi_0)= \left \{ \begin{pmatrix}
a_{11} & \textcolor{ForestGreen}{a_{12}} & \textcolor{ForestGreen}{a_{13}} & \textcolor{Sepia}{a_{14}}   \\
\textcolor{Cyan}{0} & \textcolor{Gray}{a_{22}} & \textcolor{Gray}{a_{23}} & \textcolor{YellowOrange}{0}  \\
\textcolor{Cyan}{0} & \textcolor{Gray}{a_{32}} & \textcolor{Gray}{a_{33}}& \textcolor{YellowOrange}{0}  \\
\textcolor{Fuchsia}{0} & \textcolor{Blue}{0} & \textcolor{Blue}{0} & \textcolor{RedOrange}{a_{44}}
\end{pmatrix} \in \mathrm{GL}_{4}(\mathbb{F}) \right \} \notag
\end{equation}
\end{Exa}

\begin{proof}[Proof of \cref{second Autgroup isom}] We clearly have $\mathrm{Mat}_{E_{m_0}}(f \circ g)=\mathrm{Mat}_{E_{m_0}}(f) \cdot \mathrm{Mat}_{E_{m_0}}(g)$ for all ${f,g \in \mathrm{Aut}_{\preceq}^{\mathbb{F}}(\xi_0)}$ and $f=g$ if $\mathrm{Mat}_{E_{m_0}}(f)=\mathrm{Mat}_{E_{m_0}}(g)$ which shows that $\omega_n^{\mathbb{F}}(\xi_0)$ is an injective group homomorphism. The next step is to show that $\mathrm{im}(\omega_n^{\mathbb{F}}(\xi_0))=\mathrm{GL}_{\preceq}^{\mathbb{F}}(\xi_0)$. Let $f \in \mathrm{Aut}_{\preceq}^{\mathbb{F}}(\xi_0)$ and $M:=\mathrm{Mat}_{E_{m_0}}(f)$. We can write $E_{m_0}=\bigcup_{i=1}^{s_0} E_{m_0,i}$ with
\begin{equation}
E_{m_0,i}=\left \{e_j : \sum_{k=1}^{i-1} \mu_0(v_{k}) < j \leq  \sum_{k=1}^{i} \mu_0(v_{k}) \right \} \notag.
\end{equation}
for $i \in \{1, \dots, s_0\}$ where \begin{equation}e_j=(0,\dots,0,\underset{j}{1},0,\dots,0)^T \notag
\end{equation}
denotes the $j$-th standard column vector in $\mathbb{F}^{m_0}$. By construction, $\mathbb{F}^{\mu_0(v_{j})}$ corresponds to $E_{m_0,j}$. Now
\begin{equation}
M \cdot \braket{E_{m_0,j}}_{\mathbb{F}} \subseteq \bigoplus_{ v_{i} \preceq v_{j} } \braket{E_{m_0,i} }_{\mathbb{F}} \notag
\end{equation}
is equivalent to that 
\begin{equation}
M(v_{i},v_{j})=0 \notag
\end{equation}
for all $i \in \{1, \dots, s_0 \}$ such that $v_i \npreceq v_j$.
This shows that the transformation matrix $M=\mathrm{Mat}_{E_{m_0}}(f)$ of $f$ with respect to $E_{m_0}$ has the desired form. So, $\omega_n^{\mathbb{F}}(\xi_0)$ is well-defined. On the other hand this equivalence shows that for $M \in \mathrm{GL}_{\preceq}^{\mathbb{F}}(\xi_0)$,
\begin{equation}
M \cdot \braket{E_{m_0,j}}_{\mathbb{F}} \subseteq \bigoplus_{ v_{i} \preceq v_{j} } \braket{E_{m_0,i} }_{\mathbb{F}} \notag
\end{equation}
for all $j \in \{1, \dots, s_0\}$.
As $\mathbb{F}^{\mu_0(v_{i})}$ corresponds to $E_{m_0,i}$ this implies that the $\mathbb{F}$-vector space automorphism $f: \mathbb{F}^{m_0} \xrightarrow{\sim} \mathbb{F}^{m_0}$, which is determined by $M$, satifies 
\begin{equation}
f\left( \mathbb{F}^{\mu_0(v)}\right) \subseteq \bigoplus_{w \preceq v} \mathbb{F}^{\mu_0(w)} \notag
\end{equation}
for all $v \in V_1$. So, we have $f \in \mathrm{Aut}_{\preceq}^{\mathbb{F}}(\xi_0)$ and $M= \mathrm{Mat}_{E_{m_0}}(f)$ by construction.
\end{proof}

\subsection{Framed relation families}\label{Framed relation families section}
Recall that $\xi_0=(V_0, \mu_0)$ and $\xi_1=(V_1, \mu_1)$ are  two  finite $n$-dimensional multisets such that $S_{n}^{\mathbb{F}}(\xi_0,\xi_1)$ is non-empty.
In the last section, we have explicitly determined the automorphism group of $\mathcal{F}_n^{\mathbb{F}}(\xi_0)$ in terms of transformation matrices. The idea is now to proceed with the parameterization of $S_n^{\mathbb{F}}(\xi_0, \xi_1)$ by choosing $\mathbb{F}$-bases for the set of all $\mathbb{F}$-relation families $R_n^{\mathbb{F}}(\xi_0, \xi_1)$ in order to \textit{frame} $R_n^{\mathbb{F}}(\xi_0, \xi_1)$ as a subset $Y_n^{\mathbb{F}}(\xi_0, \xi_1)$ of a product of \textit{Grassmannians} together with a group action of $\mathrm{GL}_{\preceq}^{\mathbb{F}}(\xi_0)$, such that 
\begin{equation}
R_n^{\mathbb{F}}(\xi_0, \xi_1)/\mathrm{Aut}(\mathcal{F}_n^{\mathbb{F}}(\xi_0)) \cong Y_n^{\mathbb{F}}(\xi_0, \xi_1)/\mathrm{GL}_{\preceq}^{\mathbb{F}}(\xi_0) \notag
\end{equation}
as sets.
\begin{Defi}[Grassmannian] Let $d,m \in \mathbb{N}$ with $d \leq m$. Denote by
\begin{equation}
\mathrm{G}_{\mathbb{F}}(d,m):=\{ N \subseteq \mathbb{F}^m \mid \mathrm{dim}_{\mathbb{F}}(N)=d \} \notag
\end{equation}
the \textit{Grassmannian (with respect to} $d,m$ \textit{and} $\mathbb{F}$).
\end{Defi}

Recall that \begin{equation} m_0=\abs{\xi_0}=\sum_{v \in V_0} \mu_0(v) \notag
\end{equation}
and that we identify \begin{equation}\bigoplus_{w \in V_0} \mathbb{F}^{\mu_0(w)}=\mathbb{F}^{m_0} \notag
\end{equation}
along our enumeration of $V_0=\{v_{1}, \dots, v_{s_0}\}$. The next definition is based on and inspired by \cite[Sec.~5, p.~86]{Carlsson2009}. The definition  in \cite{Carlsson2009} is more general, but condition $1$ is missing. 

\begin{Defi}[Framed $\mathbb{F}$-relation family]\label{framed Relfam} Let \begin{equation}
\delta_1: {\mathbb{N}^n} \longrightarrow {\mathbb{N}}, \quad w \longmapsto \mathrm{dim}_{\mathbb{F}} (\mathcal{F}_n^{\mathbb{F}}(\xi_1)_w) \notag .
\end{equation}
A \textit{framed} $\mathbb{F}$-\textit{relation familiy with respect to} $(\xi_0, \xi_1)$ is a family
\begin{equation}
(L_w)_{w  \in  V_1} \in \prod_{w  \in V_1} \mathrm{G}_{\mathbb{F}}(\delta_1(w), m_0) \notag
\end{equation}
such that for all $w \in V_1$:
\begin{enumerate}
\item $\pi_v(L_w)=0$ for all $v \in V_0$ with $v \nprec w$  where $\pi_v: \mathbb{F}^{m_0} \to  \mathbb{F}^{\mu_0(v)}$ denotes the canonical projection.
\item if $v \in V_1$ with $v\prec w$, then $L_{v}\subseteq L_w$. 
\item $\mathrm{dim}_{\mathbb{F}}\left( L_w / \sum_{v \prec w} L_v \right)=\mu_1(w)$.
\end{enumerate} 
$Y_{n}^{\mathbb{F}}(\xi_0, \xi_1) \subseteq \prod_{w \in V_1} \mathrm{G}_{\mathbb{F}}(\delta_1(w), m_0)$ denotes the set of all \textit{framed} $\mathbb{F}$-\textit{relation families with respect to} $(\xi_0,\xi_1)$.
\end{Defi}

\begin{Defi} $\mathrm{GL}_{\preceq}^{\mathbb{F}}(\xi_0)$ acts as a group on $\prod_{w\in V_1} \mathrm{G}_{\mathbb{F}}(\delta_1(w), m_0)$ where for \begin{equation}{(L_w)_{w\in V_1} \in \prod_{w\in V_1} \mathrm{G}_{\mathbb{F}}(\delta_1(w), m_0)} \notag
\end{equation}
and $M \in \mathrm{GL}_{\preceq}^{\mathbb{F}}(\xi_0)$,
\begin{equation}
M\cdot (L_w)_{w\in V_1}:=\left(  M\cdot  L_w \right)_{w\in V_1} \notag.
\end{equation}
\end{Defi}
Recall that we have $\mathbb{F}$-vector space monomorphisms
\begin{equation}
\tau_v: \mathcal{F}_n^{\mathbb{F}}(\xi_0)_v \longhookrightarrow \mathbb{F}^{m_0} \notag
\end{equation}
for all $v \in \mathbb{N}^n$ (see \cref{def of taus}), where we use the identification \begin{equation}\bigoplus_{v \in V_0} \mathbb{F}^{\mu_0(v)}=\mathbb{F}^{m_0} \notag
\end{equation}
along the enumeration of $V_0=\{v_1, \dots, v_{s_0}\}$. 
The next theorem is stated implicitly in \cite[p.~86]{Carlsson2009}, where condition 1 in $Y_n^{\mathbb{F}}(\xi_0, \xi_1)$ is missing.

\begin{Thm}\label{all on sets prop} 
The assignment
\begin{equation}
\theta_n^{\mathbb{F}}(\xi_0, \xi_1): R_n^{\mathbb{F}}(\xi_0,\xi_1) \longrightarrow \prod_{w\in V_1} \mathrm{G}_{\mathbb{F}}(\delta_1(w), m_0), \quad (L_w  )_{w \in V_1 } \longmapsto  (\tau_w(  L_w)  )_{w \in V_1} \notag
\end{equation}
is a well-defined set theoretic injection, such that for all ${(L_w)_{w \in V_1} \in R_n^{\mathbb{F}}(\xi_0,\xi_1) }$ and all $f \in \mathrm{Aut}(\mathcal{F}^{\mathbb{F}}_n(\xi_0))$,
\begin{equation} \theta^{\mathbb{F}}_n(\xi_0, \xi_1)(f \cdot (L_w )_{w \in V_1})=\psi_n^{\mathbb{F}}(\xi_0)(f) \cdot \theta^{\mathbb{F}}_n(\xi_0, \xi_1)( (L_w)_{w \in V_1})  \notag
\end{equation} 
with $\psi_n^{\mathbb{F}}(\xi_0)$ from \cref{second Autgroup isom}. We have \begin{equation}
\mathrm{im}(\theta_n^{\mathbb{F}}(\xi_0, \xi_1))=Y_{n}^{\mathbb{F}}(\xi_0,\xi_1) \notag.
\end{equation}
Thus, $Y_{n}^{\mathbb{F}}(\xi_0,\xi_1)\subseteq \prod_{w\in V_1} \mathrm{G}_{\mathbb{F}}(\delta_1(w), m_0)$ is $\mathrm{GL}_{\preceq}^{\mathbb{F}}(\xi_0)$-invariant in particular. Therefore, $\theta_n^{\mathbb{F}}(\xi_0, \xi_1)$ induces a set theoretic bijection
\begin{equation}
 \overline{\theta}_n^{\mathbb{F}}(\xi_0, \xi_1): R_n^{\mathbb{F}}(\xi_0,\xi_1)/ \mathrm{Aut}(\mathcal{F}^{\mathbb{F}}_n(\xi_0)) \xlongrightarrow{\sim} Y_{n}^{\mathbb{F}}(\xi_0,\xi_1)/ \mathrm{GL}_{\preceq}^{\mathbb{F}}(\xi_0) \notag
\end{equation}
and we have \begin{equation}
I^\mathbb{F}_n(\xi_0, \xi_1) \cong  Y_{n}^{\mathbb{F}}(\xi_0,\xi_1)/ \mathrm{GL}_{\preceq}^{\mathbb{F}}(\xi_0)\notag
\end{equation}
as sets.
\end{Thm}

Before we proceed with a proof of \cref{all on sets prop} we give two examples:

\begin{Exa}\label{Exa Rel fam framed} Let $n=1$. This example is illustrated in Figure \ref{fig:Relation Families}. As in \cref{Exa Rel fam2}, let \begin{align}\xi_0 &:= \{(1,1), (1,2), (2,1) ,(3,1), (3,2) \}, \notag \\
\xi_1 & := \{(2,1), (3,1), (4,1),(4,2)\} \notag 
\end{align} 
and
\begin{align}
L_{4} & := \left\langle (x^3,0,0,0,0), (0,x^3,0,0,0),(0,0,x^2,0,0), (0,0,0,x,0)\right\rangle_{\mathbb{F}} \subseteq \mathcal{F}_1^{\mathbb{F}}(\xi_0)_{4} \notag, \\
L_{3} & := \left\langle (x^2,0,0,0,0), (0,x^2,0,0,0) \right\rangle_{\mathbb{F}} \subseteq \mathcal{F}_1^{\mathbb{F}}(\xi_0)_{3}, \notag \\
L_{2} & := \left\langle (x,0,0,0,0) \right\rangle_{\mathbb{F}} \subseteq \mathcal{F}_1^{\mathbb{F}}(\xi_0)_{2}  \notag.
\end{align}
Then $(L_2,L_3,L_4)\in R_1^{\mathbb{F}}(\xi_0,\xi_1)$. We have \begin{equation}\theta_1^{\mathbb{F}}(\xi_0, \xi_1)(L_2,L_3,L_4)=\left( \tau_2(L_2),\tau_3(L_3),\tau_4(L_4)\right) \notag
\end{equation}
\begin{figure}
\centering
\includegraphics[scale=0.5]{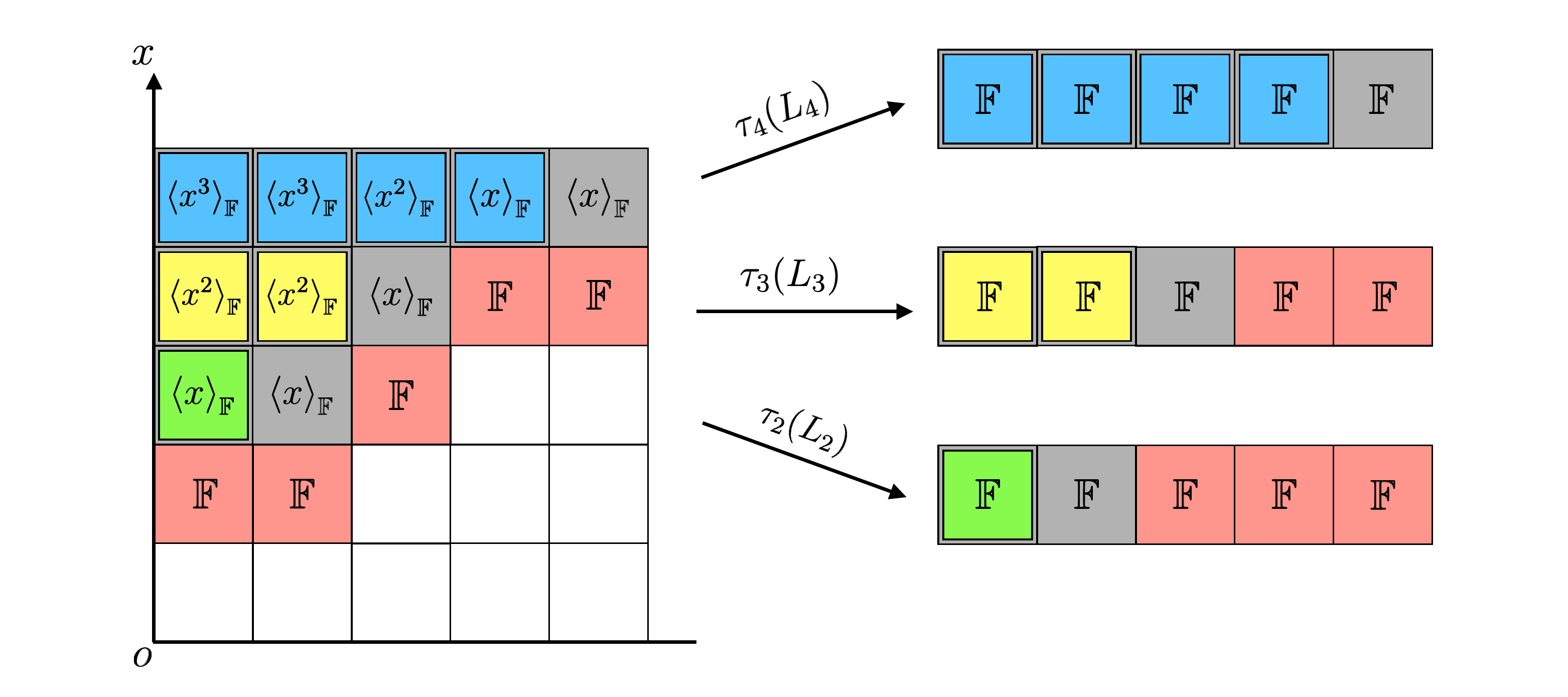} 
\caption{ $\mathbb{F}$-relation families and $\mathbb{F}$-framed relation families in dimension $n=1$ (illustration of \cref{Exa Rel fam framed}).}
\label{fig:Relation Families}
\end{figure}
where
\begin{align}
\tau_2(L_2)&= \left\langle \tau_2(x,0,0,0,0) \right\rangle_{\mathbb{F}} \notag \\ & =\left\langle (1,0,0,0,0) \right \rangle_{\mathbb{F}} \subseteq \mathbb{F}^5 ,  \notag \\
\tau_3(L_3)&=\left\langle  \tau_3(x^2,0,0,0,0), \tau_3 (0,x^2,0,0,0) \right\rangle_{\mathbb{F}} \notag \\&=\left\langle (1,0,0,0,0), (0,1,0,0,0) \right \rangle_{\mathbb{F}} \subseteq \mathbb{F}^5 , \notag \\
\tau_4(L_4)&= \left\langle \tau_4(x^3,0,0,0,0), \tau_4(0,x^3,0,0,0),\tau_4(0,0,x^2,0,0), \tau_4(0,0,0,x,0)\right\rangle_{\mathbb{F}}\notag \\ & = \left\langle (1,0,0,0,0), (0,1,0,0,0),(0,0,1,0,0), (0,0,0,1,0)\right\rangle_{\mathbb{F}} \subseteq \mathbb{F}^5 \notag
\end{align}
Figure \ref{fig:Relation Families} illustrates $\theta_1^{\mathbb{F}}(\xi_0, \xi_1)((L_j)_{j \in \{2,3,4\}})$. On the left: the blue blocks correspond to $L_4$, the yellow blocks correspond to $L_3$ and the green block corresponds to $L_2$. The red blocks are the "forbidden" blocks that we get from condition $4$ in $R_1^{\mathbb{F}}(\xi_0, \xi_1)$. On the right: the blue blocks correspond to the image of $L_4$ under $\tau_4$, the yellow blocks correspond to the image of $L_3$ under $\tau_3$ and the green block corresponds to the image of $L_2$ under $\tau_2$. The red books are the "forbidden" blocks that we  get from condition $1$ in $Y_1^{\mathbb{F}}(\xi_0, \xi_1)$.
\end{Exa}

\begin{Exa}\label{Exa Rel fam2 vs framed} 
This example can be found in \cite[Sec.~5.2]{article}, \cite[Sec.~5.2]{Carlsson2009} and will be also discussed later in \cref{cont inv}. Let $n=2$. As in \cref{Exa Rel fam2}, consider 
\begin{align}
\xi_0&:=\{\textcolor{black}{((0,0),1), ((0,0),2)}\}, \notag \\
\xi_1 &:=\{((3,0),1),((2,1),1),((1,2),1),((0,3),1)\}. \notag 
\end{align}
Figure \ref{fig:Relfamvis3 vs framed} illustrates 
\begin{align}&\theta_2^{\mathbb{F}}(\xi_0, \xi_1)(L_{(3,0)}, L_{(2,1)}, L_{(1,2)}, L_{(3,0)}) \notag \\
&=(\tau_{(3,0)}(L_{(3,0)}), \tau_{(2,1)}(L_{(2,1)}),\tau_{(1,2)}( L_{(1,2)}), \tau_{(3,0)}(L_{(3,0)})) \notag
\end{align}
\begin{figure}
\centering
\includegraphics[scale=0.34]{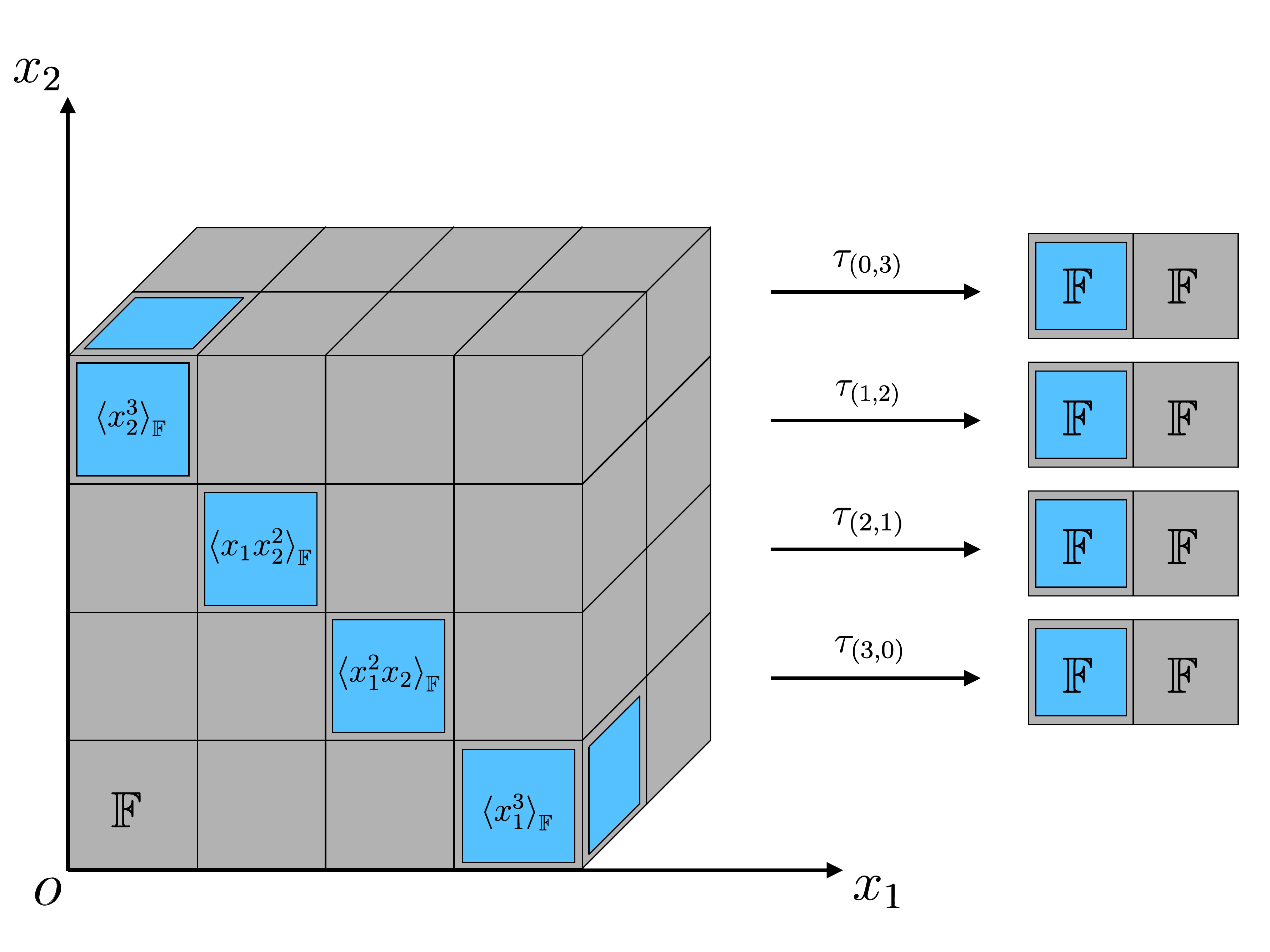} 
\caption{ $\mathbb{F}$-relation families and $\mathbb{F}$-framed relation families in dimension $n=2$ (illustration of \cref{Exa Rel fam2 vs framed}).}
\label{fig:Relfamvis3 vs framed}
\end{figure}
where 
\begin{align}
&L_{(3,0)}:=\braket{(x_1^3,0)}_{\mathbb{F}} \subseteq \mathcal{F}_2^{\mathbb{F}}(\xi_0)_{(3,0)}, \notag \\  &L_{(2,1)}:=\braket{(x_1^2x_2,0)}_{\mathbb{F}} \subseteq \mathcal{F}_2^{\mathbb{F}}(\xi_0)_{(2,1)}, \notag \\
&L_{(1,2)}:=\braket{(x_1x_2^2,0)}_{\mathbb{F}} \subseteq \mathcal{F}_2^{\mathbb{F}}(\xi_0)_{(1,2)}, \notag \\ &L_{(0,3)}:=\braket{(x_2^3,0)}_{\mathbb{F}} \subseteq \mathcal{F}_2^{\mathbb{F}}(\xi_0)_{(0,3)} \notag.
\end{align}
We have 
\begin{align}
&\tau_{(3,0)}(L_{(3,0)})=\braket{\tau_{(3,0)}(x_1^3,0)}_{\mathbb{F}} =\braket{(1,0)}_{\mathbb{F}} \subseteq \mathbb{F}^2, \notag \\  &\tau_{(2,1)}(L_{(2,1)})=\braket{\tau_{(2,1)}(x_1^2x_2,0)}_{\mathbb{F}} =\braket{(1,0)}_{\mathbb{F}} \subseteq \mathbb{F}^2, \notag \\
&\tau_{(1,2)}(L_{(1,2)})=\braket{\tau_{(1,2)}(x_1x_2^2,0)}_{\mathbb{F}}= \braket{(1,0)}_{\mathbb{F}}\subseteq \mathbb{F}^2, \notag \\ &\tau_{(0,3)}(L_{(0,3)})=\braket{\tau_{(0,3)}(x_2^3,0)}_{\mathbb{F}}= \braket{(1,0)}_{\mathbb{F}} \subseteq \mathbb{F}^2\notag.
\end{align}
Conditions 1-3 in $Y_2^{\mathbb{F}}(\xi_0, \xi_1)$ are trivial and we obtain
\begin{align}
 Y_{2}^{\mathbb{F}}(\xi_0, \xi_1) =\mathrm{G}_{\mathbb{F}}(1,2)^4 = \mathbb{P}_1 (\mathbb{F})^4 \notag
\end{align}
where $\mathbb{P}_1(\mathbb{F})$ denotes the projective line over $\mathbb{F}$ (the set of all one-dimensional subspaces of $\mathbb{F}^{2}$). Since $\mathrm{GL}_{\preceq}^{\mathbb{F}}(\xi_0)=\mathrm{GL}_2 (\mathbb{F})$
(see \cref{Aut dim 2}), we have
\begin{equation}
 I_2^{\mathbb{F}}(\xi_0, \xi_1)\cong Y_{2}^{\mathbb{F}}(\xi_0, \xi_1)/\mathrm{GL}_{\preceq}^{\mathbb{F}}(\xi_0) =  \mathbb{P}_1 (\mathbb{F})^4/\mathrm{GL}_2(\mathbb{F}) \notag
\end{equation}
as sets.
\end{Exa}

\begin{proof}[Proof of \cref{all on sets prop}]  Since $\tau_w$ is injective for all $w\in V_1$  by construction, $\theta_n^{\mathbb{F}}(\xi_0, \xi_1)$ is well-defined and injective. So, let us show that $\mathrm{im}(\theta_n^{\mathbb{F}}(\xi_0, \xi_1))=Y_{n}^{\mathbb{F}}(\xi_0,\xi_1)$. 

$\mathrm{im}(\theta_n^{\mathbb{F}}(\xi_0, \xi_1))\subseteq Y_{n}^{\mathbb{F}}(\xi_0,\xi_1)$: for this, let $(L_w)_{w \in V_1}\in R_n^{\mathbb{F}}(\xi_0, \xi_1)$. We have to show that $(\tau_w(L_w))_{w \in V_1}$ satisfies conditions 1-3 in $Y_n^{\mathbb{F}}(\xi_0, \xi_1)$. Let $w \in V_1$.
\begin{enumerate}
\item By construction of the $\tau_w$, it is clear that  $\pi_v(\tau_w(L_w))=0$ for all $v \in V_0$ with $v \nprec w$.
\item By construction of the $\tau_v$, we have $\tau_v(L_v)=\tau_w(x^{w-v} L_v)$ for all $v \in V_1$ with $v \prec w$. Hence,
\begin{align}
\tau_v(L_v)=\tau_{w} (x^{w-v} L_v )\subseteq \tau_{w} (L_w ) \notag
\end{align}
for all $v \in V_1$ with $v \prec w$.
\item We have  $ \tau_w (L_w) \cong L_w$ as $\mathbb{F}$-vector spaces for all $w \in V_1$ and therefore, \begin{equation}{\sum_{v \prec w} \tau_v (L_v) \cong \sum_{v \prec w} L_v } \notag
\end{equation}
as $\mathbb{F}$-vector spaces. We have $\sum_{v \prec w} \tau_v (L_v) \subseteq \tau_w (L_w)$ by condition 2. Thus, we conclude that \begin{equation}
\mu_1(w)=\mathrm{dim}_{\mathbb{F}}\left(L_w /\sum_{v \prec w} L_v \right)=\mathrm{dim}_{\mathbb{F}} \left(\tau_w (L_w) /\sum_{v \prec w} \tau_v (L_v) \right) \notag.
\end{equation}
\end{enumerate}

$Y_{n}^{\mathbb{F}}(\xi_0,\xi_1) \subseteq \mathrm{im}(\theta_n^{\mathbb{F}}(\xi_0, \xi_1))$: let $(L_w)_{w \in V_1} \in Y_{n}^{\mathbb{F}}(\xi_0,\xi_1)$. Let $w \in V_1$ and $z \in L_w$. Then 
\begin{equation}
z=(z_{1,1}, \dots , z_{1, \mu_0(v_1)}, \dots, z_{s_0,1}, \dots , z_{s_0, \mu_0(v_{s_0})}) \notag
\end{equation}
for suitable $z_{i,j} \in \mathbb{F}$ where for all $i \in \{1, \dots, s_0\}$, 
\begin{align} 
 v_i \nprec w \Longrightarrow \forall \, j \in \{1, \dots, \mu_0(v_i) \}: z_{i,j}=0   \notag.
\end{align}
Let $i \in \{1, \dots, s_0\} $.  If $ v_i \prec w$, define $\widetilde{z}_{i,j}:=x^{w-v_j}z_{i,j}$ for all $j \in \{1, \dots, \mu_0(v_i)\}$ and if $ v_i \nprec w$, define $\widetilde{z}_{i,j}:=0$ for all $j \in \{1, \dots, \mu_0(v_i)\}$. Let 
\begin{align}
\widetilde{z} :=(\widetilde{ z}_{1,1}, \dots , \widetilde{z}_{1, \mu_0(v_1)}, \dots, \widetilde{z}_{s_0,1}, \dots , \widetilde{z}_{s_0, \mu_0(v_{s_0})}) \notag
\end{align}
and define $\widetilde{ L}_w:=\{\widetilde{ z} \mid z \in L_w \}$. Then $(\widetilde{L}_w)_{w \in V_1} \in R_n^{\mathbb{F}}(\xi_0, \xi_1)$ and 
\begin{equation}
\theta_n^{\mathbb{F}}(\xi_0, \xi_1)((\widetilde{L}_w)_{w \in V_1} )=(L_w)_{w \in V_1} \notag.
\end{equation}
Thus, we conclude that $Y_{n}^{\mathbb{F}}(\xi_0,\xi_1) \subseteq \mathrm{im}(\theta_n^{\mathbb{F}}(\xi_0, \xi_1))$. The rest is clear by construction of $\psi_n^{\mathbb{F}}(\xi_0)$.
\end{proof}

\subsection{The non-existence of a discrete and complete invariant}\label{cont inv}
Using the parameterization result of the last section, we are now able to give a universal counterexample for any $n \geq 2$, which proves the non-existence of a discrete class of complete invariants \begin{equation}\left\{f_n^{\mathbb{F}}: \mathbf{Grf}_n(A_n^{\mathbb{F}})/_{\cong} \longrightarrow Q_n^{\mathbb{F}} \right\}_{\mathbb{F}  \in \mathcal{K}} \notag.
\end{equation}

The following counter example can be found in \cite[Sec.~5.2]{article} and\cite[Sec.~5.3]{Carlsson2009} (see also \cref{Exa Rel fam2 vs framed}). We just generalize it in a very natural way from $n=2$ to $n\geq 2$. For $n=2$, define  \begin{align}
v_{2,1}:=(3,0) ,\quad  v_{2,2}:=(2,1),  \quad v_{2,3}:=(1,2),\quad
v_{2,4}:=(0,3) \notag
\end{align}
and for $n\geq 3$, define
\begin{align}
v_{n,1}&:=(3,0,0, \dots,0), \quad v_{n,2} :=(2,1,0, \dots, 0), \notag \\ v_{n,3}&:=(1,2,0, \dots, 0),\quad v_{n,4}:=(0,3,0,\dots, 0) \notag.
\end{align}
Now let 
\begin{align}
\xi_{n,0}&:=\{((0, \dots, 0),1),((0, \dots, 0),2) \}, \notag\\
\xi_{n,1}&:=\{(v_{n,1},1),(v_{n,2},1),(v_{n,3},1),(v_{n,4},1) \} \notag.
\end{align}
\begin{Prop}\label{hilfsprop} If $\mathbb{F}$ is finite we have
\begin{align}
\infty >\abs{I_{n}^{\mathbb{F}}(\xi_{n,0}, \xi_{n,1})} \geq \abs{\mathbb{F}}-2 \notag.
\end{align}
If  $\mathbb{F}$ is countable $I_{n}^{\mathbb{F}}(\xi_{n,0}, \xi_{n,1})$ is countable and if $\mathbb{F}$ is uncountable $I_{n}^{\mathbb{F}}(\xi_{n,0}, \xi_{n,1})$ is uncountable.
\end{Prop}
\begin{proof}
We have $m_0=2$, $\mathrm{GL}_{\preceq}^{\mathbb{F}}(\xi_{n,0})=\mathrm{GL}_{2}(\mathbb{F})$ (see \cref{Aut dim 2}) and  $\delta_1(v_{n,i})=1$  for all $i\in \{1, \dots, 4\}$. Conditions $1$-$3$ in \cref{framed Relfam} are all trivial. Therefore,
\begin{align}
 Y_{n}^{\mathbb{F}}(\xi_0, \xi_1) =\mathrm{G}_{\mathbb{F}}(1,2)^4 = \mathbb{P}_1 (\mathbb{F})^4 \notag
\end{align}
where $\mathbb{P}_1(\mathbb{F})$ denotes the projective line over $\mathbb{F}$ (the set of all one-dimensional subspaces of $\mathbb{F}^{2}$). By \cref{all on sets prop}, we obtain a set theoretic bijection
\begin{align}
I_{n}^{\mathbb{F}}(\xi_{n,0}, \xi_{n,1}) \cong {\mathbb{P}_1(\mathbb{F})^4}/{\mathrm{GL}_2(\mathbb{F})}  \notag.
\end{align}
Now let \begin{equation}
\Omega:=\{(l_1,l_2,l_2,l_4)\in \mathbb{P}_1(\mathbb{F})^4 \mid l_i\neq l_j \, \mathrm{for} \, i\neq j \}\subseteq \mathbb{P}_1(\mathbb{F})^4 \notag
\end{equation}
be the subspace of pairwise distinct lines. $\Omega$ is clearly $\mathrm{GL}_2(\mathbb{F})$ invariant and we have $ {\mathbb{P}_1(\mathbb{F})^4}/{\mathrm{GL}_2(\mathbb{F})}\supseteq {\Omega}/{\mathrm{GL}_2(\mathbb{F})}$. 

We show that $\Omega/\mathrm{GL}_2(\mathbb{F}) \cong  \mathbb{F}\setminus \{0,1\}$ as sets: let \begin{equation}
((a_1:a_2),(b_1:b_2),(c_1:c_2),(d_1:d_2)) \in \Omega \notag.
\end{equation}
As $(a_1:a_2) \neq (b_1:b_2)$, we have 
$\mathbb{F}^2=\braket{(a_1,a_2), (b_1,b_2)}_{\mathbb{F}}$ which shows that 
\begin{equation}
M:=\begin{pmatrix}
a_1 & b_1 \\
a_2 & b_2
\end{pmatrix} \in \mathrm{GL}_2(\mathbb{F}) \notag .
\end{equation}
Hence, \begin{align}
&(M^{-1}(a_1:a_2),M^{-1}(b_1:b_2),M^{-1}(c_1:c_2),M^{-1}(d_1:d_2)) \notag \\
=&((1:0),(0:1),M^{-1}(c_1:c_2),M^{-1}(d_1:d_2)) \notag.
\end{align}
Now let $(\widetilde{c_1}:\widetilde{c_2}):=M^{-1} (c_1:c_2)$. As $(\widetilde{c_1}:\widetilde{c_2}) \neq (1:0), (0:1)$, 
\begin{equation}
D:=\begin{pmatrix}
\frac{1}{c_1} & 0 \\
0 & \frac{1}{c_2}
\end{pmatrix} \in \mathrm{GL}_2(\mathbb{F}) \notag
\end{equation}
is well-defined and we obtain
\begin{align}
&(D(1:0),D(0:1),D(\widetilde{c_1}:\widetilde{c_2}),DM^{-1}(d_1:d_2)) \notag \\
=&((1:0),(0:1),(1:1),DM^{-1}(d_1:d_2)) \notag.
\end{align}
Thus, we have
\begin{align}
\Omega/\mathrm{GL}_2(\mathbb{F})=\{ & \mathrm{GL}_2(\mathbb{F}) \cdot (l_1,l_2,l_3,l_4)  \mid \notag \\
&  l_1=(1:0), l_2=(0:1), l_3=(1:1), l_4=(x:y) \neq l_1,l_2,l_3 \} \notag
\end{align}
which shows that the assignment
\begin{align}
\Omega/\mathrm{GL}_2(\mathbb{F})   \xlongrightarrow{\sim} \mathbb{P}_1(\mathbb{F})\setminus \{(1:0), (0:1), (1:1) \} , \quad
 \mathrm{GL}_2(\mathbb{F}) \cdot (l_1,l_2,l_3,l_4)  \longmapsto l_4 \notag
\end{align}
is a bijection of sets. Since
\begin{ceqn}
\begin{align}
{\mathbb{P}_1(\mathbb{F})^4}/{\mathrm{GL}_2(\mathbb{F})}\supseteq {\Omega}/{\mathrm{GL}_2(\mathbb{F})} \cong \mathbb{P}_1(\mathbb{F})
\setminus \{(1:0), (0:1), (1:1) \}  \cong \mathbb{F} \setminus \{0,1\} \notag
\end{align}
as sets, the statement follows.
\end{ceqn}
\end{proof}
This shows that for $n\geq 2$, there exists no discrete class of complete invariants 
\begin{equation}\{f_{\mathbb{F}}: \mathbf{Grf}_n(A_n^{\mathbb{F}})/_{\cong} \longrightarrow Q_n^{\mathbb{F}} \}_{\mathbb{F}  \in \mathcal{K}} \notag.
\end{equation}

\subsection{The moduli space}\label{The Moduli space}
Let us start with recalling some standard material: Let 
$\mathcal{D}$ be a category, \begin{equation}
F:\mathcal{D}^\mathrm{opp} \longrightarrow \mathrm{Sets} \notag
\end{equation}
a functor and $(S, \beta)$ a pair with $S\in \mathcal{D}$ and $\beta \in F(S)$. If the assignment\begin{equation}
\mathrm{Hom}_{\mathcal{D}}(T, S) \longrightarrow F(T), \quad f \longmapsto F(f)(\beta) \notag
\end{equation}
is a bijection for all $T \in \mathcal{D}$, then $F$ is called \textit{representable} by $(S, \beta)$. Note that if there exists an isomorphism $g: \mathrm{Hom}_{\mathcal{D}}(-, S) \to F$, then $F$ is representable by $(S, g(\mathrm{id}_S))$ and we just say that $F$ is representable by $S$ if there exists such an isomorphism. Let $S$ be a scheme. An $S$-\textit{scheme} $X$ is a scheme $X$ together with a morphism of schemes $X \to S$. A \textit{morphism} of $S$-schemes $X \to Y$ is a morphism of schemes such that the diagram
\begin{equation}
\begin{tikzcd}
X \arrow[rd] \arrow[rr] && Y\arrow[ld] \\
&S&
\end{tikzcd} \notag
\end{equation}
commutes. Denote by 
\begin{enumerate}
\item $\mathrm{LNSch}_{\mathbb{F}}$ the category of locally Noetherian $\mathbb{F}$-schemes,
\item $\mathrm{NAff}_{\mathbb{F}}$ the category of Noetherian affine $\mathbb{F}$-schemes and
\item $\mathrm{NAlg}_{\mathbb{F}}$ the category of Noetherian $\mathbb{F}$-algebras.
\end{enumerate}
Any $X \in \mathrm{LNSch}_{\mathbb{F}}$ defines a functor (the so-called functor of points) \begin{equation}
h_X: \mathrm{LNSch}_{\mathbb{F}}^{\mathrm{opp}} \longrightarrow \mathrm{Sets}, \quad Y \longmapsto \mathrm{Hom}_{\mathbb{F}}(Y, X) \notag
\end{equation}
We often use the notation $h_X(-):=X(-)$. By \textit{Yoneda‘s Lemma}, the scheme $X$ is completely determined by its functor of points $h_X$. Moreover, $h_X$ is completey determined by its restriction to $\mathrm{NAff}_{\mathbb{F}}^{\mathrm{opp}}$ (or $\mathrm{NAlg}_{\mathbb{F}}$ since $\mathrm{NAff}_{\mathbb{F}}^{\mathrm{opp}}$ is equivalent to $\mathrm{NAlg}_{\mathbb{F}}$) (see \cite[Prop.~VI-2, Ex.~VI-3]{eisenbud2006geometry}).

\begin{Defi} Let $S$ be a scheme and $X,Y$ be $S$-schemes. $Y$ is called an \textit{open/closed subscheme} of $X$ if  there exists a morphism of $S$-schemes $Y \to X$ which is an open/closed immersion. In this case, we often simply say that $Y \subseteq X$ is an open/closed subscheme.
\end{Defi}

The following definition is standard material about schemes. See for example \cite[Def.~3.9]{GW} where it is formulated for $\mathrm{Spec}({\mathbb{Z}})$-schemes (for $S$-schemes everything is analogous).

\begin{Defi}[Gluing Datum] Let $S$ be a scheme. A \textit{gluing datum} of $S$-schemes consists of the following data:
\begin{enumerate}
\item an index set $I$,
\item for all $i \in I$, an $S$-scheme $U_i$,
\item for all $i,j \in I$ an open subset $U_{i,j} \subseteq U_i$ (we consider $U_{ij}$ as open subschemes of $U_i$),
\item for all $i,j \in I$ an isomorphism $\varphi_{ji}: U_{ij} \to U_{ji}$ of $S$-schemes,
\end{enumerate}
such that 
\begin{enumerate}
\item $U_{ii}=U_i$ for all $i \in I$
\item the \textit{cocycle condition} holds: $\varphi_{kj} \circ \varphi_{ji}=\varphi_{ki}$ on $U_{ij} \cap U_{ik}$ for all $i,j,k \in I$.
\end{enumerate}
Note that we implicitly assume in the cocycle condition that $\varphi_{ji}(U_{ij} \cap U_{ik}) \subseteq U_{jk}$ (such that the composition is meaningful). For $i = j = k$, the cocycle condition implies that
$\varphi_{ii} =\mathrm{id}_{U_i} $ and (for $i=k$), $\varphi_{ji}:U_{ij}  \xrightarrow{\sim} U_{ji}$ is an isomorphism of $S$-schemes with  $\varphi_{ij}^{-1} =\varphi_{ji}$.
\end{Defi}

The following proposition is standard material about schemes. See for example \cite[Prop.~3.10]{GW} where it is formulated for $\mathrm{Spec}({\mathbb{Z}})$-schemes (for $S$-schemes everything is analogous).

\begin{Prop}[Gluing Schemes]\label{Gluing schemes} Let $S$ be a scheme and \begin{equation}
((U_i)_{i\in I}, (U_{ij})_{i,j\in I}, (\varphi_{ij})_{i,j \in I}) \notag
\end{equation}
a gluing datum of $S$-schemes. Then there exists an $S$-scheme $X$ together with $S$-morphisms $\psi_i: U_i \to X$ for all $i \in I$ such that 
\begin{enumerate}
\item for all $i \in I$, the map $\psi_i$ yields an $S$-isomorphism from $U_i$ onto an open subscheme of $X$,
\item $\psi_j \circ \varphi_{ji} = \psi_i$ for all $i,j \in I$,
\item $X=\bigcup_{i \in I} \psi_i(U_i)$,
\item $\psi_i(U_i) \cap \psi_j(U_j)=\psi_i(U_{ij})=\psi_j(U_{ji}) $ for all $i,j \in I$.
\end{enumerate}
Furthermore, $X$ together with the $\psi_i$ is uniquely determined up to unique $S$-isomorphism. We usually identify $U_i=\psi_i(U_i) \subseteq X$ and write  $X=\bigcup_{i \in I} U_i$ as well as \begin{equation}
U_i \cap U_j= \psi_i(U_i) \cap \psi_j(U_j)=\psi_{i}(U_{ij})=\psi_{j}(U_{ij})=U_{ij} \notag.
\end{equation}
\end{Prop}

\begin{Defi} Let $m \in \mathbb{N}$. Denote by \begin{enumerate}
\item  $\mathbb{A}_\mathbb{F}^m= \mathrm{Spec}(\mathbb{F}[x_1, \dots, x_m])$ the affine space of dimension $m$ as $\mathbb{F}$-scheme,
\item $\mathbb{P}_\mathbb{F}^m= \mathrm{Proj}(\mathbb{F}[x_0, \dots, x_m])$ the projective space of dimension $m$ as $\mathbb{F}$-scheme and
\item $\mathbb{P}_m(\mathbb{F})$ the projective space of dimension $m$ as the set of all one-dimensional subspaces of $\mathbb{F}^{m+1}$.
\end{enumerate}
Note that we have $\mathbb{A}_\mathbb{F}^m(\mathbb{F})=\mathbb{F}^m$ and $\mathbb{P}_\mathbb{F}^m(\mathbb{F})=\mathbb{P}_m(\mathbb{F})$.
\end{Defi}

\begin{Defi}[Projective $\mathbb{F}$-scheme] An $\mathbb{F}$-scheme $X$ is called a \textit{projective} $\mathbb{F}$-\textit{scheme} if there exists a natural number $m \in \mathbb{N}$ such that $X \subseteq \mathbb{P}^m_\mathbb{F}$ is a closed subscheme.
\end{Defi}

\begin{Defi} Let $X$ be an $\mathbb{F}$-scheme.
\begin{enumerate}
\item $X$ is called an \textit{affine} $\mathbb{F}$-\textit{variety} if $X=\mathrm{Spec}(R)$ for a finitely generated $\mathbb{F}$-alegbra $R$.
\item $X$ is called an $\mathbb{F}$-\textit{variety} if $X= \bigcup_{i=1}^m X_i$ where $X_i \subseteq X$ is open and an affine variety for all $i\in \{1, \dots, m\}$.
\item $X$ is called a \textit{projective} $\mathbb{F}$-variety if $X$ is an $\mathbb{F}$-variety and a projective $\mathbb{F}$-scheme. 
\item Let $X,Y$ be $\mathbb{F}$-varieties. $Y$ is called an \textit{open/closed subvariety} of $X$ if $Y$ is an open/closed subscheme of $X$.
\item Let $X,Y$ be $\mathbb{F}$-varieties. Then $Y$ is called a \textit{subprevariety} of $X$ if there exists a closed subvariety $Z \subseteq X$ such that $Y \subseteq Z$ is an open subvariety.
\end{enumerate}
\end{Defi}

\subsubsection{Parameterization as a subprevariety}\label{param as subprevariety}

We begin this section with the realization of $\mathrm{G}_{\mathbb{F}}(d,m)$ as the set of $\mathbb{F}$-points of a projective $\mathbb{F}$-scheme $\mathcal{G}_{\mathbb{F}}(d,m)$, the so-called \textit{Grassmann scheme}. There are different ways to construct $\mathcal{G}_{\mathbb{F}}(d,m)$. Our construction follows that in \cite[III.2.7]{eisenbud2006geometry} and focuses on the technical details that are missing in \cite{eisenbud2006geometry}.

\begin{Thm}\label{Grassmann scheme}  Let $d,m \in \mathbb{N}$ with $d \leq m$. Then there exists a projective $\mathbb{F}$-scheme $\mathcal{G}_{\mathbb{F}}(d,m)$ (the \textit{Grassmann scheme}) together with a set theoretic bijection
\begin{equation}
 \mathcal{G}_{\mathbb{F}}(d,m)(\mathbb{F}) \xlongrightarrow{\sim}  \mathrm{G}_{\mathbb{F}}(d,m) \notag .
\end{equation}
If $d > m$, we set $\mathcal{G}_{\mathbb{F}}(d,m):= \emptyset$. 
\end{Thm}

\begin{proof} Let $L \in \mathrm{G}_{\mathbb{F}}(d,m)$. Pick an $\mathbb{F}$-basis $B_L$ of $L$ and define $M(B_L)\in M_{d \times m}(\mathbb{F})$ as the matrix which has the transposed basis vectors as entries. Let \begin{equation}
M_{d \times m}(\mathbb{F})_d:= \{ M \in M_{d \times m}(\mathbb{F}) \mid \mathrm{rank}(M)=d \} \notag.
\end{equation}
For $M,N \in M_{d \times m}(\mathbb{F})_d$, define $M \sim N$ if and only if $A \cdot M = N$ for some $A \in \mathrm{GL}_d(\mathbb{F})$. Then 
\begin{equation}
\eta: \mathrm{G}_{\mathbb{F}}(d,m) \longrightarrow M_{d \times m } (\mathbb{F})_d/_\sim, \quad L \longmapsto [M(B_L)] \notag
\end{equation}
is a well-defined bijection. Define $\mathcal{I}:=\{(i_1< \dots < i_d) \mid 1 \leq i_j \leq m\}$. Let \begin{equation}[M] \in M_{d \times m } (\mathbb{F})_d/_\sim \notag.
\end{equation}
Since $\mathrm{rk}(M)=d$ there exists an $I \in \mathcal{I}$ such that 
$\mathrm{det}(M_{I})\neq 0$. Hence, $(M_{I}^{-1} M)_I=1_d$. Here $M_I$ denotes the $I$-th submatrix of $M$. $\mathrm{det}(M_I)$ is also called the $I$-th minor of $M$. For $I \in \mathcal{I}$, let 
\begin{align}
U(I)&:= \{ [M] \in M_{d \times m}(\mathbb{F})/_{\sim} \mid  \mathrm{det}(M_I) \neq 0 \} \notag, \\
V(I)&:= \{ L \in \mathrm{G}_{\mathbb{F}}(d,m) \mid L \oplus \braket{\{e_i\}_{i \in I^{c}}} = \mathbb{F}^m \} \notag, \\
W(I) & :=\{ M \in M_{d \times m}(\mathbb{F})\mid M_I= {1}_d \} \notag.
\end{align}
Then $\bigcup_{I \in \mathcal{I}} V(I)= \mathrm{G}_{\mathbb{F}}(d,m)$, $\bigcup_{I \in \mathcal{I}} U(I)= M_{d \times m}(\mathbb{F})/_{\sim}$ and  $\eta|_{V(I)}=U(I)$. We obtain a well-defined bijection
\begin{equation}
\lambda(I): U(I) \longrightarrow W(I), \quad [M] \longmapsto M_I^{-1} M \notag
\end{equation}
Note that $W(I) \cong \mathbb{A}^{m(m-d)}_\mathbb{F} (\mathbb{F})$ as sets. For $I,J \in \mathcal{I}$, define
\begin{equation}
W(I,J):=\{ M \in W(I) \mid \mathrm{det}(M_J)\neq 0 \} \notag
\end{equation}
Then we have $W(I,I)=W(I)$ and set theoretic bijections
\begin{equation}
\varphi_{IJ}: W(I,J) \longrightarrow W(J,I), \quad M \longmapsto M_J^{-1} M \notag.
\end{equation}
Now let
$X:=\bigcup_{I \in \mathcal{I}} W(I)/_{\sim}$ where for $x \in W(I)$ and $y \in W(J)$, $x \sim y$ if $x \in W(I,J)$, $y \in W(J,I)$ and $\varphi_{IJ}(x)=y$. We obtain a well-defined set theoretic bijection
\begin{equation}
\lambda: \bigcup_{I \in \mathcal{I}} U(I) \longrightarrow X, \quad U(I) \ni x \longmapsto \lambda(I)(x) \notag
\end{equation}
and hence a set theoretic bijection
\begin{equation}
\lambda \circ \eta: \mathrm{G}_{\mathbb{F}}(d,m) \longrightarrow X \notag.
\end{equation}
The next step is to translate everything into polynomial equations. Let ${R:=\mathbb{F}[(x_{ij})]}$ be the polynomial ring in $d m$ variables. Then $\mathcal{W}:=\mathrm{Spec}(R) \cong \mathbb{A}_\mathbb{F}^{dm}$. Let ${K(I):=\braket{(x_{i,j})_I- 1_d }}$ and $R(I):=R / K(I)$.
Now
\begin{equation}
\mathcal{W}(I):=\mathrm{Spec}(R(I)) \notag
\end{equation}
is a closed subscheme of $\mathcal{W}$ with $\mathcal{W}(I)(\mathbb{F})=W(I)$. For $J\in \mathcal{I}$, define 
\begin{equation}
R(I,J):=R(I) \left[\mathrm{det}((x_{ij})_J)^{-1} \right] \notag
\end{equation}
and
\begin{equation}\mathcal{W}(I,J):=
\mathrm{Spec}(R(I,J)) \notag.
\end{equation}
$\mathcal{W}(I,J)$ is an open  subscheme of $\mathcal{W}(I)$. We have $\mathcal{W}(I,I)=\mathcal{W}(I)$
and for any $T \in \mathrm{NAlg}_{\mathbb{F}}$, we obtain
\begin{align}
\mathcal{W}(I)(T)=\mathrm{Hom}_\mathbb{F}( R(I), T) =\{ M \in M_{d \times m}(T) \mid M_I=1_d \} \notag
\end{align}
and 
\begin{equation}
\mathcal{W}(I,J)(T)=\mathrm{Hom}_\mathbb{F}( R(I,J), T)= \{M \in \mathcal{W}(I)(S) \mid \mathrm{det}(M_J) \in T^\times \} \notag.
\end{equation}
The assignment 
\begin{equation}
\varphi_{JI}(T): \mathcal{W}(I,J)(T) \longrightarrow \mathcal{W}(J,I)(T), \quad M \longmapsto M_J^{-1} M \notag.
\end{equation}
is a set theoretic bijection and natural in $T$. Hence, we obtain an isomorphism of $\mathbb{F}$-schemes
\begin{equation}
\varphi_{JI}: \mathcal{W}(I,J) \longrightarrow \mathcal{W}(J,I)\notag.
\end{equation}
We have for all $M \in \mathcal{W}(I,J)(T) \cap \mathcal{W}(I,K)(T) $, 
\begin{align}
(\varphi_{KJ}(T) \circ \varphi_{JI}(T)) (M)&=\varphi_{KJ}(T)(M_I^{-1} M)= M_K^{-1} M_I^{-1} M \notag \\
&\overset{M_I=1_d}{=} M_K^{-1} M=\varphi_{KI}(T)(M) \notag
\end{align}
which shows that the cocycle condition is satisfied. Now use \cref{Gluing schemes} to define $\mathcal{G}_{\mathbb{F}}(d,m)$ as the $\mathbb{F}$-scheme that one obtains by gluing the $\mathcal{W}(I)$ along the $\varphi_{JI}$. By Construction, we have $\mathcal{G}_{\mathbb{F}}(d,m)(\mathbb{F})=X$ 
and $(\eta \circ \lambda)^{-1}$ is a set theoretic bijection from $\mathcal{G}_{\mathbb{F}}(d,m)(\mathbb{F})$ to $\mathrm{G}_{\mathbb{F}}(d,m)$. So, from an intuitive point of view, this construction is the correct one if we consider $\mathbb{F}$-points. That $\mathcal{G}_{\mathbb{F}}(d,m)$ is a projective variety follows from \cite[Ex.~III-49] {eisenbud2006geometry} or \cite[(8.10)]{GW}.
\end{proof}

The next theorem is based on and inspired by \cite[Sec.~5]{Carlsson2009} and \cite[Thm.~4]{article}. In order to realize $Y_{n}^{\mathbb{F}}(\xi_0, \xi_1)$ as a subprevariety of a product of Grassmann varieties, we have to translate conditions 1-3 in the definition of $Y_{n}^{\mathbb{F}}(\xi_0, \xi_1)$ (see \cref{framed Relfam}) into polynomial equations. This translation is spelled out in \cite[Sec.~5]{Carlsson2009} (without the first condition in the definition of $Y_{n}^{\mathbb{F}}(\xi_0, \xi_1)$) while in \cite{article} the authors only state the theorem (here the first and third condition in $Y_{n}^{\mathbb{F}}(\xi_0, \xi_1)$ are missing). Our proof follows \cite[Sec.~5]{Carlsson2009} with a focus on the technical details and on the translation of the first condition in the definition of $Y_{n}^{\mathbb{F}}(\xi_0, \xi_1)$ into polynomial equations.

\begin{Thm}\label{Paramet Thm on F points} There exists a subprevariety \begin{equation}\mathcal{Y}_n^{\mathbb{F}}(\xi_0, \xi_1) \subseteq \prod_{w\in V_1} \mathcal{G}_{\mathbb{F}}(\delta_1(w), m_0) \notag
\end{equation}
and a set theoretic bijection
\begin{equation}
\mathcal{Y}_n^{\mathbb{F}}(\xi_0, \xi_1)(\mathbb{F}) \xlongrightarrow{\sim} Y_{n}^{\mathbb{F}}(\xi_0, \xi_1)\notag.
\end{equation}
\end{Thm}

\begin{proof} 
Before we start with the proof, we introduce some notations:
\begin{enumerate}
\item recall that $\xi_0=(V_0, \mu_0)$, $\xi_1=(V_1, \mu_1)$, ${\delta_1(w)=\mathrm{dim}_{\mathbb{F}} (\mathcal{F}_n^{\mathbb{F}}(\xi_1)_w)}$ for all $w \in \mathbb{N}^n$, ${m_0=\abs{\xi_0}= \sum_{v \in V_0}\mu_0(v)}$ and that we indentify $\mathbb{F}^{m_0}= \bigoplus_{v \in V_0} \mathbb{F}^{\mu_0(v)}$ along our enumeration of ${V_0=\{v_1, \dots, v_{s_0}\}}$.  For $w \in V_1$, let \begin{equation}
I_{w}=\left(i^{(w)}_1 < \dots < i^{(w)}_{\delta_1(w)} \right) \notag
\end{equation}
with $1 \leq i^{(w)}_t \leq m_0$. Denote by $\mathcal{I}(w)$ the set of all such sets $I_{w}$ and define \begin{equation}
{\mathcal{J}:=\prod_{w \in V_1} \mathcal{I}(w)} \notag.
\end{equation}
We have an open covering of $\prod_{w\in V_1} \mathcal{G}_{\mathbb{F}}(\delta_1(w), m_0)$ by affines that are isomorphic to $\mathcal{W}(I)$
where $\mathcal{W}(I):=\prod_{w \in V_1} \mathcal{W}(I_{w})$ with ${I=(I_{w})_{w \in V_1} \in \mathcal{J}}$ (see proof of \cref{Grassmann scheme}). We have \begin{equation}
\mathcal{W}(I)(\mathbb{F})=\prod_{w \in V_1} \mathcal{W}(I_w)(\mathbb{F})=\prod_{w \in V_1} W(I_w) \notag.
\end{equation}
Define $W(I):=\prod_{w \in V_1} W(I_w)$.
\item for $w\in  V_1$, let $\underline{x}(w):=\left(x^{(w)}_{ij} \right)$ where $1 \leq i \leq \delta_1(w)$, $1 \leq j \leq m_0$. By adapting the notations of the proof of \cref{Grassmann scheme}, define \begin{equation}
R:=\mathbb{F}\left[(\underline{x}({w}))_{w \in V_1} \right], \quad K\left(I \right):=  \left \langle\bigcup_{w \in  V_1} K\left(I_{w} \right) \right \rangle, \quad R(I):=R/K(I) \notag
\end{equation}
where $I=(I_w)_{w \in V_1} \in \mathcal{J}$. Then $\mathcal{W}\left(I \right)= \mathrm{Spec}(R(I))$. For $J \in \mathcal{J}$, $\mathcal{W}(I,J)$ is defined analogously and we obtain canonical isomorphisms 
\begin{equation}\varphi_{JI}:\mathcal{W}(I,J) \to \mathcal{W}(J,I) \notag
\end{equation}
such that $\prod_{w \in V_1} \mathcal{G}_{\mathbb{F}}(\delta_1(w), m_0)$ results by gluing the $\mathcal{W}(I)$ along the $\varphi_{JI}$. 
\item for $w \in V_1$, let 
\begin{equation}
\eta_w: \mathrm{G}_{\mathbb{F}}(\delta_1(w),m_0) \longrightarrow M_{\delta_1(w) \times m_0 } (\mathbb{F})_{\delta_1(w)}/_\sim \notag
\end{equation}
and 
\begin{equation}
\lambda_{I_w}: U(I_w) \longrightarrow W(I_w) \notag
\end{equation}
as before (see proof of \cref{Grassmann scheme}). We have a set theoretic bijection
\begin{equation}
\gamma_{I_w}:=\lambda_{I_w} \circ \eta_w|_{V(I_w)}: V(I_w) \longrightarrow W(I_w) \notag.
\end{equation}
\end{enumerate}
To show the existence of a subprevariety \begin{equation}\mathcal{Y}_n^{\mathbb{F}}(\xi_0, \xi_1) \subseteq \prod_{w\in V_1} \mathcal{G}_{\mathbb{F}}(\delta_1(w), m_0) \quad \textrm{with} \quad {Y}_n^{\mathbb{F}}(\xi_0, \xi_1) \cong \mathcal{Y}_n^{\mathbb{F}}(\xi_0, \xi_1)(\mathbb{F}) \notag
\end{equation}
we have to interprete conditions 1-3 in \cref{framed  Relfam} via polynomial equations. Recall that $(L_w)_{w  \in  V_1} \in \prod_{w\in V_1} \mathrm{G}_{\mathbb{F}}(\delta_1(w), m_0)$ is in $Y_{n}^{\mathbb{F}}(\xi_0, \xi_1)$ if and only if for all $w \in V_1$:
\begin{enumerate}
\item $\pi_v(L_w)=0$ for all $v \in V_0$ with $v \nprec w$ where $\pi_v: \mathbb{F}^{m_0} \to  \mathbb{F}^{\mu_0(v)}$ denotes the canonical projection.
\item if $v \in V_1$ with $v\prec w$, then $L_{v}\subseteq L_w$. 
\item $\mathrm{dim}_\mathbb{F} \left( L_w / \sum_{v \prec w} L_v \right) = \mu_1(w)$.
\end{enumerate}
Note that every $ (L_w)_{w \in V_1} \in \prod_{w \in V_1} \mathrm{G}_{\mathbb{F}}(\delta_1(w), m_0)$ corresponds to a familiy of matrices $(\gamma_{I_w}(L_w))_{w \in V_1} \in W(I)$ for some $I=(I_w)_{w \in V_1} \in \mathcal{J}$.

\begin{enumerate}

\item For $w\in V_1$, let $Q_0(w):=\{ v \in V_0 \mid v \nprec w \}$. Condition $1$ tells us that if $v \in Q_0(w)$ we have  $\pi_v(L_w)=0$ where $\pi_v: \mathbb{F}^{m_0} \to  \mathbb{F}^{\mu_0(v)}$ denotes the canonical projection. In particular, $\mathbb{F}^{\mu_0(v)}$ corresponds to a subsummand of \begin{equation}
\mathbb{F}^{m_0}=\bigoplus_{v\in V_0} \mathbb{F}^{\mu_0(v)}=\bigoplus_{j=1}^{s_0} \mathbb{F}^{\mu_0(v_j)} \notag.
\end{equation}
Define
\begin{equation}
T(v_j):=\left \{ \left(\sum_{i=1}^{ \mu_0 (v_{j-1})} \mu_0(i) \right) +1, \dots, \mu_0 (v_j) \right \} \notag
\end{equation}
Then
\begin{equation}
\mathbb{F}^{m_0}=\bigoplus_{v\in V_0} \mathbb{F}^{\mu_0(v)}=\bigoplus_{j=1}^{s_0} \mathbb{F}^{\mu_0(v_j)} =\bigoplus_{j=1}^{s_0} \bigoplus_{i \in T(v_j)} \mathbb{F}_i \notag
\end{equation}
where $\mathbb{F}_i=\mathbb{F}$. For $(\gamma_{I_w}(L_w))_{w \in V_1} \in W(I)$, condition $1$ now means that \begin{align}
\forall \, w \in V_1 \, \forall \, k \in \{1, \dots, s_0\} \, \forall \, j \in \{1, \dots, m_0\} \,  \forall \, i \in \{1, \dots, \delta_1(w)\}:  \notag \\ v_k \in Q_0(w) , \,  j \in T(v_k) \Longrightarrow (\gamma_{I_w}(L_w))_{i,j}=0 \notag .
\end{align}
Define
\begin{equation} P:=\bigcup_{w \in V_1 }  \bigcup_{v_k \in Q_0(w) }  \left \{x^{(w)}_{i,j} \mid j \in T(v_k), \,  i=1, \dots, \delta_1(w) \right \} \notag.
\end{equation}
Then 
\begin{equation}
(\gamma_{I_w}(L_w))_{w \in V_1} \in \mathrm{Spec}\left( R/\braket{P, K(I)} \right) (\mathbb{F}) \subseteq W(I) \notag
\end{equation}
if and only if $(L_w)_{w \in V_1}$ satisfies condition $1.$

\item Now we interprete condition $2$ via polynomial equations. So, let $v,w \in V_1$ with $v \prec w$. We wish to find a condition in terms of polynomial equations which captures the property $ L_v \subseteq L_{w}$. $ L_v \subseteq L_{w}$ is equivalent to that each row of  $\gamma_{I_v}(L_v)$ lies in the $\mathbb{F}$-span  of the rows of $\gamma_{I_w}(L_w)$. Since the rank of
\begin{equation}
A:=\begin{pmatrix}
\gamma_{I_w}(L_w) \\
\gamma_{I_v}(L_v)
\end{pmatrix} \notag
\end{equation}
is at least $\delta_1(w)$ this is equivalent to that all $\delta_1(w)+1$ minors of $A$ vanish. This minor vanishing condition can interpreted via polynomial equations as follows: let
\begin{equation}
M(v,w):= \begin{pmatrix}
x_{1,1}^{(w)} & \dots & x_{1, m_0}^{(w)} \\ 
\vdots &  & \vdots \\
x_{\delta_1(w),1}^{(w)} & \dots & x_{\delta_1(w) ,m_0}^{(w)} \\
x_{1,1}^{(v)}  ,& \dots &, x_{1, m_0}^{(v)} \\
\vdots &  & \vdots \\
 x_{\delta_1(v),1}^{(v)}  ,& \dots &,  x_{\delta_1(v), m_0}^{(v)}  
\end{pmatrix} \notag
\end{equation}
Let $\mathcal{D}(v,w)$ be the set of all $J=(i_1 < \dots <  i_{\delta_1(w)+1})$ with $1 \leq i_j \leq i_{m_0}$ such that $M(v,w)_{J}$ is a submatrix of  $M(v,w)$. Now define
\begin{align}
M &:=  \{ \mathrm{det}(M(v,w)_{J} )\mid v,w \in V_1: v \prec w,  J \in \mathcal{D}(v,w) \} \notag , \\
X(I) & := R/\braket{P,M, K(I)}, \notag \\
\mathcal{X}(I)&:= \mathrm{Spec}(X(I)). \notag
 \notag
\end{align}
Then $\mathcal{X}(I)$ is a closed subscheme of $\mathcal{W}(I)$ and \begin{equation}
(\gamma_{I_w}(L_w))_{w \in V_1} \in \mathcal{X}(I)(\mathbb{F}) \subseteq W(I) \notag
\end{equation}
if and only if $(L_w)_{w \in V_1}$ satisfies conditions $2$ and $1.$

\item Recall that condition $3$ is that $\mathrm{dim}_\mathbb{F} \left( L_w / \sum_{v \prec w} L_v \right) = \mu_1(w)$ for all $w \in V_1$. Note that this condition already includes the condition that $L_v \subseteq L_w$ for all $v \prec w$. Let us assume that $(L_w)_{w \in V_1}$ satisfies condition $2$
and let $w \in V_1$. Define ${Q_1(w):=\{v  \in V_1 \mid v \prec  w\}}$ and enumerate \begin{equation}
Q_1(w)=\{v(w)_1, \dots, v(w)_u\} \notag
\end{equation}
where $u=\abs{Q_1(w)}$. Now $\mathrm{dim}_\mathbb{F} \left( L_w / \sum_{v \prec w} L_v \right) = \mu_1(w)$ is equivalent to that
\begin{equation}
N(w):=\begin{pmatrix}
\gamma_{I_{v(w)_1}}(L_{v(w)_1}) \\
\vdots \\
\gamma_{I_{v(w)_u}}(L_{v(w)_u})
\end{pmatrix} \notag
\end{equation}
has rank $m_0-\mu_1(w)$. This is  equivalent to  that all the $m_0-\mu_1(w)+1$ minors vanish and that some $m_0-\mu_1(w)$ minor does not vanish. So, consider the matrix
\begin{equation}
H(w):= \begin{pmatrix}
 x_{1,1}^{(v(w)_1)}  ,& \dots &, x_{1, m_0}^{(v(w)_1)}  \\
\vdots &  & \vdots \\
 x_{\delta_1(v(w)_1),1}^{(v(w)_1)}  ,& \dots &,  x_{\delta_1(v(w)_1), m_0}^{(v(w)_1)}  \\
\vdots &  & \vdots \\
 x_{1,1}^{(v(w)_u)}  ,& \dots &,  x_{1, m_0}^{(v(w)_u)}  \\
\vdots &  & \vdots \\
 x_{\delta_1(v(w)_u),1}^{(v(w)_u)}  ,& \dots &,  x_{\delta_1(v(w)_u), m_0}^{(v(w)_u)}  \\
\end{pmatrix} \notag
\end{equation}
amd define $\mathcal{A}(w)$ as the set of all $J=(i_1<\dots < i_{m_0-\mu_1(w)+1})$ with $1 \leq i_j \leq m_0$ such that $H(w)_J$ is a submatrix of $H(w)$. Furthermore, define $\mathcal{B}(w)$ as the set of all $J=(i_1<\dots < i_{m_0-\mu_1(w)})$ with $1 \leq i_j \leq m_0$ such that $H(w)_J$ is a submatrix of $H(w)$. Define 
\begin{align} N_1 &  := \bigcup_{w \in V_1 }  \{ \mathrm{det}(Z(w)_J) \mid J \in \mathcal{A}(w)\} \notag , \\
N_2 & :=  \bigcup_{w \in V_1 }   \{ \mathrm{det}(Z(w)_J) \mid J \in \mathcal{B}(w)\} \notag.
\end{align}
Now define \begin{enumerate}
\item $Z(I)   := R/\braket{P,M,N_1, K(I)}$ and 
$\mathcal{Z}(I)   := \mathrm{Spec}(Z(I))$.
\item $Y(I)  :=R/\braket{P,M,N_1, K(I)} \left[N_2^{-1} \right] $ and
$\mathcal{Y}(I) := \mathrm{Spec}(Y(I)) $.
\end{enumerate}
$\mathcal{Y}(I)$ is an open subscheme of $\mathcal{Z}(I)$ and $\mathcal{Z}(I)$ is a closed subscheme of $\mathcal{X}(I)$ which is a closed subscheme of $\mathcal{W}(I)$. So, $\mathcal{Z}(I)$ is a closed subscheme of $\mathcal{W}(I)$. We have
\begin{equation}(\gamma_{I_{w}}(L_w))_{w \in V_1} \in \mathcal{Y}(I)(\mathbb{F}) \subseteq W(I) \notag
\end{equation}
if and only if $(L_w)_{w \in V_1}$ satisfies conditions $1$, $2$ and $3.$ 
\end{enumerate}
The next step is to show that we can glue everything to a  subprevariety of \begin{equation}
\prod_{w \in V_1} \mathcal{G}_{\mathbb{F}} (\delta_1(w), m_0) \notag.
\end{equation}
To see that we indeed have a subprevariety, we divide the gluing construction in two steps:
\begin{enumerate}
\item first of all, we show that we can glue the $\mathcal{Z}\left(I \right)$ together. For this, pick another familiy $ J\in \mathcal{J}$ and let
\begin{equation}
\mathcal{Z}\left(I, J \right):= \mathrm{Spec} \left( Z(I)\left [K(J)^{-1} \right] \right) \notag
\end{equation}
Then the $\varphi_{JI}$ induce canonical isomorphisms 
\begin{equation}
\psi_{JI}:\mathcal{Z}(I,J) \longrightarrow \mathcal{Z}(J,I) \notag.
\end{equation}
The cocycle condtion is satisfied as it is satisfied for $\varphi_{JI}$. Now use \cref{Gluing schemes} and let $\mathcal{Z} \subseteq \prod_{w\in V_1} \mathcal{G}_{\mathbb{F}}(\delta_1(w), m)$ be the closed subscheme that one obtains by gluing the $\mathcal{Z}(I)$ along the $\psi_{JI}$. 
\item the next is step is to show that we can glue the $\mathcal{Y}(I)$ together. For this, pick another familiy $J\in \mathcal{J}$ and let
\begin{equation}
\mathcal{Y}(I,J):= \mathrm{Spec}\left(Y(I) \left[K(J)^{-1}  \right] \right) \notag
\end{equation}
Then the $\psi_{JI}$ induce canonical isomorphisms
\begin{equation}
\vartheta_{JI}:\mathcal{Y}(I,J) \longrightarrow \mathcal{Y}(J,I) \notag
\end{equation} 
The cocycle condition is satisfied as it is satisfied for $\psi_{JI}$. Now use \cref{Gluing schemes} and let $\mathcal{Y}_n^{\mathbb{F}}(\xi_0, \xi_1) \subseteq \mathcal{Z}$ be the open subscheme that one obtains by gluing the $\mathcal{Y}(I)$ along the $\vartheta_{JI}$.
\end{enumerate}
So, $\mathcal{Y}_n^{\mathbb{F}}(\xi_0, \xi_1) \subseteq \mathcal{Z}$ is an open subscheme and $\mathcal{Z} \subseteq \prod_{w \in V_1} \mathcal{G}_{\mathbb{F}} (\delta_1(w), m_0)$ is a closed subscheme which shows that
$\mathcal{Y}_n^{\mathbb{F}}(\xi_0, \xi_1) \subseteq \prod_{w \in V_1} \mathcal{G}_{\mathbb{F}} (\delta_1(w), m_0)$ is a subprevariety. The set theoretic bijection $(\eta \circ \lambda)^{-1}:\mathcal{G}_{\mathbb{F}}(d,m_0) \xlongrightarrow{\sim} \mathrm{G}_{\mathbb{F}}(d,m_0)$ induces a set theoretic bijection 
\begin{equation}
\prod_{w\in V_1} \mathcal{G}_{\mathbb{F}}(\delta_1(w), m_0)(\mathbb{F}) \xlongrightarrow{\sim} \prod_{w\in V_1} \mathrm{G}_{\mathbb{F}}(\delta_1(w), m_0) \notag
\end{equation}
which induces a set theoretic bijection
\begin{equation}
\mathcal{Y}_n^{\mathbb{F}}(\xi_0, \xi_1)(\mathbb{F}) \xlongrightarrow{\sim} Y_{n}^{\mathbb{F}}(\xi_0, \xi_1) \notag
\end{equation}
by construction.
\end{proof}

\newpage

\subsubsection{Algebraic group action}\label{Functor of points}

In this section, we give some ideas that provide evidence why the action of $\mathrm{GL}_{\preceq}^{\mathbb{F}}(\xi_0)$ on $Y_{n}^{\mathbb{F}}(\xi_0, \xi_1)$ is algebraic. In this section, we use the same notations as in the proofs of Theorems \ref{Grassmann scheme} and \ref{Paramet Thm on F points}.

First of all note that for every $m\in \mathbb{N}$, there exists an algebraic $\mathbb{F}$-group
$\mathcal{GL}^\mathbb{F}_m$ such that for any $T  \in \mathrm{NAlg}_{\mathbb{F}}$
\begin{equation}
\mathcal{GL}^\mathbb{F}_m(T)= \mathrm{GL}_m(T). \notag
\end{equation}
$\mathcal{GL}^\mathbb{F}_m$ is given as the affine Noetherian $\mathbb{F}$-scheme
\begin{equation}
\mathcal{GL}^\mathbb{F}_m=\mathrm{Spec}\left( \mathbb{F}[(x_{ij})][\mathrm{det}(x_{ij})^{-1}] \right) \notag.
\end{equation}
$\mathcal{GL}^\mathbb{F}_m$ is called the \textit{general linear group scheme} (see \cite[4.4.3]{GW}). Recall that \begin{equation}
Z(\xi_0)=\{(i,j) \in \{1, \dots, s_0\}^2 \mid v_i \npreceq v_j \} \notag
\end{equation}
(see \cref{transformation matrix defi}). For $T \in \mathrm{NAlg}_{\mathbb{F}}$, let
\begin{equation}
\mathrm{GL}^{T}_{\preceq}(\xi_0):= \left \{ M\in \mathrm{GL}_{m_0}(T) \mid  M(v_{i},v_{j})=0 \, \forall \, (i,j) \in Z(\xi_0)  \right \} \notag.
\end{equation} 
Then there exists a closed subscheme $\mathcal{GL}_{\preceq}^{\mathbb{F}}(\xi_0) \subseteq \mathcal{GL}_{m_0}^{\mathbb{F}}$ such that for any $T \in \mathrm{NAlg}_{\mathbb{F}}$,
\begin{equation} \mathcal{GL}_{\preceq}^{\mathbb{F}}(\xi_0)(T) = \mathrm{GL}_{\preceq}^{T}(\xi_0) \notag
\end{equation}
which we obtain by setting the variables in $\mathbb{F}[(x_{ij})][\mathrm{det}(x_{ij})^{-1}]$, which correspond to $M(v_{i},v_{j})$, equal to zero.
It remains to show that $\mathcal{GL}_{\preceq}^{\mathbb{F}}(\xi_0)\subseteq \mathcal{GL}_{m_0}^{\mathbb{F}}$ is an algebraic subgroup. 

Now we turn our attention to the determination of  $\mathcal{Y}_n^{\mathbb{F}}(\xi_0, \xi_1)(T)$  for $T \in \mathrm{NAlg}_{\mathbb{F}}$. In the following, we often refer to \cite[Ex.~VI-18]{eisenbud2006geometry}. Note that our notations are different and that we are in the relative setting over $\mathbb{F}$. The so-called \textit{Grassmann functor} (or \textit{Grassmannian functor} (see \cite[Ex.~VI-18]{eisenbud2006geometry})) is given by
\begin{equation}
\mathrm{G}_{(-)}(d,m):\mathrm{NAlg}_{\mathbb{F}} \longrightarrow \mathrm{Sets}, \quad T \longmapsto \mathrm{G}_T(d,m) \notag
\end{equation}
where 
\begin{equation}
\mathrm{G}_T(d,m):=\{ N \subseteq T^m \mid T^m/N \textrm{ is locally free of rank } m-d \} \notag
\end{equation}
(\cite[Ex.~VI-18]{eisenbud2006geometry} gives a different but equivalent definition for $G_{(-)}(d,m)$) and for a morphism $f: T \to S$ in $\mathrm{NAlg}_{\mathbb{F}}$,
\begin{equation}
\mathrm{G}_f(d,m):\mathrm{G}_{T}(d,m) \longrightarrow \mathrm{G}_{S}(d,m), \quad N \longmapsto N \otimes_T S \notag
\end{equation}
where $N \otimes_T S$ denotes the scalar extension from $N$ to $S$ along $f$. By \cite[Ex.~VI-18]{eisenbud2006geometry}, there exists an isomorphism of functors
\begin{equation}
\mathrm{G}_{(-)}(d,m) \cong \mathcal{G}_{\mathbb{F}}(d,m)(-) \notag,
\end{equation}
i.e. the Grassmann scheme represents the Grassmann functor.
The proof idea is as follows: following \cite[Ex.~VI-18]{eisenbud2006geometry}, $\mathrm{G}_{(-)}(d,m)$ is a sheaf in the zariski topology and for every $I \in \mathcal{I}$,
\begin{align}
\mathrm{G}_{I}: \mathrm{NAlg}_{\mathbb{F}}\longrightarrow \mathrm{Sets}, \quad T  \longmapsto \{  N  \subseteq T^m \mid  N \textrm{ free of rank } d, \,  N \oplus \braket{\{e_i\}_{i \in I^{c}}} = T^m \} \notag
\end{align}
defines a subfunctor of $\mathrm{G}_{(-)}(d,m)$ (\cite[Ex.~VI-18]{eisenbud2006geometry} gives a different but equivalent definition for $\mathrm{G}_I$). The next step is to show that $\mathrm{G}_{I}$ is represented by 
\begin{equation}\mathbb{F}[\mathcal{W}(I)]= \mathbb{F}[(x_{i,j})]/\braket{(x_{i,j})_I- 1_d } \notag.
\end{equation}
Now we have for every field $\mathbb{E} \in \mathrm{NAlg}_{\mathbb{F}}$, 
\begin{equation}
\mathcal{G}_{\mathbb{F}}(d,m)(\mathbb{E})\overset{(*)}=\bigcup_{I \in \mathcal{I}}\mathcal{W}(I)(\mathbb{E})=\bigcup_{I \in \mathcal{I}} \mathrm{G}_I(\mathbb{E}) \notag.
\end{equation}
Note that $(*)$ is true for fields, but not generally. Now the idea is to apply \cite[Thm.~VI.14]{eisenbud2006geometry}\footnote{where we have to claim that also a relative version of  \cite[Thm.~VI.14]{eisenbud2006geometry} over $\mathbb{F}$ is true} in order to show that $\mathrm{G}_{(-)}(d,m)$ is represented by a Noetherian $\mathbb{F}$-scheme\footnote{note that it is not directly clear that the scheme obtained from \cite[Thm.~VI.14]{eisenbud2006geometry} is the same as the scheme obtained from the construction in the proof of \cref{Grassmann scheme}}. The idea is now to define an action of $\mathcal{GL}_{\preceq}^{\mathbb{F}}(\xi_0)$ on $\prod_{w \in V_1} \mathcal{G}_{\mathbb{F}}(\delta_1(w), m_0)$. This  could be done as follows: 

for $T \in \mathrm{NAlg}_{\mathbb{F}}$, $M \in \mathcal{GL}_{\preceq}^{\mathbb{F}}(\xi_0)(T)$ and ${(L_w)_{w \in V_1} \in \prod_{w \in V_1} \mathcal{G}_{\mathbb{F}}(\delta_1(w), m_0)(T)}$, define
\begin{align}
\sigma_T: \mathcal{GL}_{\preceq}^{\mathbb{F}}(\xi_0)(T) \times \prod_{w \in V_1} \mathcal{G}_{\mathbb{F}}(\delta_1(w), m_0)(T)  &\longrightarrow \prod_{w \in V_1} \mathcal{G}_{\mathbb{F}}(\delta_1(w), m_0)(T) , \notag \\
 (M, (L_w)_{w\in V_1})  &\longmapsto (M \cdot L_w)_{w \in V_1} \notag.
\end{align}
$\sigma_T$ is clearly  natural in $T$ and we therefore obtain a morphism 
\begin{align}
\sigma: \mathcal{GL}_{\preceq}^{\mathbb{F}}(\xi_0) \times\prod_{w \in V_1} \mathcal{G}_{\mathbb{F}}(\delta_1(w), m_0) \longrightarrow \prod_{w \in V_1} \mathcal{G}_{\mathbb{F}}(\delta_1(w), m_0) \notag
\end{align}
in $\mathrm{LNSch}_{\mathbb{F}}$. It remains to show that this action is algebraic. 
The next step is to determine for $T \in\mathrm{NAlg}_{\mathbb{F}} $ the set of $T$-points $\mathcal{Y}_n^{\mathbb{F}}(\xi_0,\xi_1)(T)$. For $T \in\mathrm{NAlg}_{\mathbb{F}} $, let \begin{equation}Y_n^{T}(\xi_0,\xi_1) \subseteq \prod_{w\in V_1} \mathrm{G}_{T}(\delta_1(w), m_0) \notag
\end{equation}
be the set of families $(L_w)_{w  \in  V_1} \in \prod_{w \in V_1} \mathrm{G}_{T}(\delta_1(w), m_0)$ such that for all $w \in V_1$:
\begin{enumerate}
\item $\pi_v(L_w)=0$ for all $v \in V_0$ with $v \nprec w$ where $\pi_v: T^{m_0} \to  T^{\mu_0(v)}$ denotes the canonical projection.
\item if $v \in V_1$ with $v\prec w$, then $L_{v}\subseteq L_w$. 
\item $ L_w / \sum_{v \prec w} L_v $ is locally free of rank $\mu_1(w)$.
\end{enumerate}
Now we propose the following: the assignment
\begin{equation}
Y_n^{(-)}(\xi_0,\xi_1):\mathrm{NAlg}_{\mathbb{F}} \longrightarrow \mathrm{Sets}, \quad T \longmapsto Y_n^T(\xi_0, \xi_0) \notag
\end{equation}
defines a subfunctor of $\prod_{w \in V_1} \mathrm{G}_{(-)}(\delta_1(w), m_0)$ and we have a natural isomorphism of functors
\begin{equation}\label{Y points}
Y_n^{(-)}(\xi_0,\xi_1) \cong \mathcal{Y}_n^{\mathbb{F}}(\xi_0,\xi_1)(-).
\end{equation}
As a proof strategy for (\ref{Y points}), we propose the same as in the case of Grassmannians: for $I=(I_w)_{w \in V_1} \in \mathcal{J}$, define subunctors 
\begin{equation}
Y_I:\mathrm{NAlg}_{\mathbb{F}} \longrightarrow \mathrm{Sets}, \quad T \longmapsto Y_n^{T}(\xi_0,\xi_1) \cap \prod_{w \in V_1} \mathrm{G}_{I_w}(T) \notag
\end{equation}
and show that $Y_I$ is represented by $\mathbb{F}[\mathcal{Y}(I)]$. Then the rest is analogously to the case of Grassmannians. Now the claim is that $\mathcal{Y}_n^{\mathbb{F}}(\xi_0,\xi_1)$ is invariant under the action of $\mathcal{GL}_{\preceq}^{\mathbb{F}}(\xi_0)$ on $\prod_{w \in V_1} \mathcal{G}_{\mathbb{F}}(\delta_1(w), m_0)$.

\newpage

\bibliographystyle{amsalpha}
\bibliography{mybib}

\end{document}